\documentclass[a4paper,11pt,twoside,reqno]{article}

\usepackage{hyperref}

\usepackage[english]{babel}
\usepackage{fancyhdr}
\usepackage{fullpage} 
\usepackage[dvips]{graphicx}
\usepackage[T1]{fontenc}
\usepackage{fouriernc}
\usepackage[latin1,applemac]{inputenc}
\usepackage{amsmath}
\usepackage{amsfonts}
\usepackage{amssymb}
\usepackage{mathrsfs}
\usepackage{amsthm}
\usepackage{latexsym}
\usepackage{array} 
\usepackage{mathrsfs}
\usepackage{amsxtra} 
\usepackage{amscd} 
\usepackage{mathtools} 
\usepackage{verbatim}

\usepackage{enumerate}

\usepackage[toc,page]{appendix}

\usepackage[usenames]{color}

\definecolor{beige}{rgb}{0.96, 0.96, 0.86}
\definecolor{airforceblue}{rgb}{0.36, 0.54, 0.66}
\definecolor{antiquefuchsia}{rgb}{0.57, 0.36, 0.51}
\definecolor{awesome}{rgb}{1.0, 0.13, 0.32}

\usepackage{tikz}
\usetikzlibrary{matrix}

\newtheorem{Theorem}{Theorem}[section]
\newtheorem{Definition}[Theorem]{Definition}
\newtheorem{Proposition}[Theorem]{Proposition}
\newtheorem{Lemma}[Theorem]{Lemma}
\newtheorem{Corollary}[Theorem]{Corollary}

\newtheorem{Assumption}[Theorem]{Assumption}
\newtheorem{Remark}[Theorem]{Remark}
\newtheorem{Notation}[Theorem]{Notation}

\theoremstyle{definition}
\newtheorem{Example}[Theorem]{Example}

\renewcommand{\epsilon}{\varepsilon}
\newcommand*{\medcap}{\mathbin{\scalebox{1.5}{\ensuremath{\cap}}}}
\newcommand*{\medcup}{\mathbin{\scalebox{1.5}{\ensuremath{\cup}}}}

\numberwithin{equation}{section}

\def\R{\mathbb R}
\def\N{\mathbb N}
\def\E{\mathbb E}

\def\ca{\mathbf{ca}}

\def\<{\left\langle }
\def\>{\right\rangle }

\usepackage{xcolor}
\hypersetup{
    colorlinks,
    linkcolor={green!60!blue},
    citecolor={green!60!blue},
    urlcolor={green!60!blue}
}

\begin{document}

\title{
$C_0$-sequentially equicontinuous semigroups:\\theory and applications
\footnote{This research has been partially supported by the INdAM-GNAMPA project ``PDE correlate a sistemi stocastici con ritardo''. The authors are sincerely grateful to Fausto Gozzi for valuable discussions.}
}

\author{Salvatore Federico $^\mathrm{(a)}$ \and Mauro Rosestolato $^\mathrm{(b)}$}

\maketitle

\begingroup
\renewcommand*{\thefootnote}{\alph{footnote}}

\footnotetext[1]{\thinspace  Universit\`a di Firenze, Dipartimento per l'Economia e l'Impresa (Italy). Email:
\texttt{salvatore.federico@unifi.it}.}

\footnotetext[2]{\thinspace LUISS University, Rome (Italy). Email:\ \texttt{mrosestolato@luiss.it}. The author acknowledges the financial support of DAAD (Deutscher Akademischer Austausch Dienst) for his visit at the University of Wuppertal.}

\endgroup
%
%
%
%

\begin{abstract}
We present and apply a theory of one parameter $C_0$-semigroups of linear operators in locally convex spaces. Replacing the notion of equicontinuity considered by the literature with the weaker notion of sequential equicontinuity, we prove the basic results of the classical theory of $C_0$-equicontinuous semigroups:\ we show that the semigroup is uniquely identified by its generator and we provide a generation theorem in the spirit of the celebrated Hille-Yosida theorem. Then, we particularize the theory in some functional spaces and identify two locally convex topologies that allow to gather under a unified framework various notions $C_0$-semigroup introduced by some authors to deal with Markov transition semigroups. Finally, we apply the results to transition semigroups associated to stochastic differential equations.

\vskip 5pt
\noindent\textbf{Keywords:}
One parameter semigroup, sequential equicontinuity, transition semigroup.

\vskip 5pt
\noindent\textbf{AMS 2010 subject classification:} 46N30, 47D06, 47D07, 60J35.
\end{abstract}

\setcounter{tocdepth}{3}


%

\section{Introduction}
\label{sec:2015-05-01:00}
The aim of this paper is to present and apply a notion of  one parameter  strongly continuous ($C_0$) semigroups of linear operators in locally convex spaces  based on the notion of  sequential equicontinuity and  following the spirit and the methods of the classical theory in Banach spaces. 

The theory of $C_0$-semigroups was first stated in Banach spaces (a widespread presentation can be found in several monographs, e.g.\ \cite{EngelNagel99, HillePhillips57, Pazy83}).
 The theory was extended to locally convex spaces by introducing the notions of $C_0$-equicontinuous semigroup (\cite[Ch.\  IX]{Yosida80}), $C_0$-quasi-equicontinuous semigroup (\cite{Choe85}), $C_0$-locally equicontinuous semigroup  (\cite{Dembart74, Komura68}), weakly integrable semigroup (\cite{Jefferies86, Jefferies87}).  A mixed approach is the one followed by \cite{Kuehnemund2001}, which  introduces the notion of bi-continuous semigroup:\ in a framework of Banach spaces, semigroups that are strongly continuous with respect  to a weaker locally convex topology are considered. 
 
In this paper we deal with semigroups of linear operators in locally convex spaces that are only \emph{sequentially} continuous. The idea is due to the following key observation:\ the theory of $C_0$-(locally) equicontinuous semigroups  can be developed, with appropriate adjustments, to semigroups of operators which are only $C_0$-(locally) \emph{sequentially} equicontinuous (in the sense specified by Definition \ref{def:2015-05-02:06}). 
On the other hand, as we will show by examples, the passage from equicontinuity to sequential equicontinuity is motivated and fruitful:\ as discussed in Remark \ref{rem..} and shown by
Example \ref{2015-05-04:07-II},  in concrete  applications, replacing equicontinuity with sequential equicontinuity
might turn out to be much more convenient.

The main motivation that led us to consider sequential continuity is that it allows a convenient treatment of Markov  transition semigroups. 
The employment of Markov transition semigroups to the study of partial differential equations through the use of stochastic representation formulas is the subject of a wide mathematical literature (here we only refer to \cite{Cerrai01book} in finite and infinite dimension and to \cite{DaPratoZabczyk14} in infinite dimension). Also, the
regularizing properties of 
such semigroups
 is the core of a regularity theory for second order PDEs (see, e.g., \cite{Lunardi95}).
Unfortunately, the framework of $C_0$-semigroup in 
Banach spaces is not always appropriate to treat such semigroups. Indeed, on Banach spaces of functions not vanishing at infinity, the $C_0$-property fails already in  basic cases, such as the
one-dimensional Ornstein-Uhlenbeck semigroup,  when considering it in the space of bounded uniformly continuous real-valued functions  $(UC_b(\mathbb{R}),|\cdot|_\infty)$ (see, e.g., \cite[Ex.\ 6.1]{Cerrai94} for a counterexample, or  \cite[Lemma\ 3.2]{DaPratoLunardi95},
 which implies this semigroup is strongly continuous in $(UC_b(\R),|\cdot|_\infty)$ if and only if the drift of the associated  stochastic differential equation vanishes). 
 On the other hand,   
 finding a locally convex topology on these spaces  to frame Markov transition semigroups within the theory of $C_0$-locally equicontinuous semigroups is not an easy task (see also the considerations of Remark \ref{rem..}). In the case of the Ornstein-Uhlenbeck semigroup, such approach is adopted by  \cite{GoldysKocan01}. 
Some authors have bypassed these difficulties by introducing some (more or less \emph{ad hoc}) notions, relying on some sequential continuity properties, to treat such semigroups  (weakly continuous semigroups  \cite{Cerrai94}, $\pi$-continuous semigroups \cite{Priola99}, bi-continuous semigroups  \cite{Kuehnemund2001}). 
The theory developed in our paper allows to gather all the aforementioned notions under a unified framework.
 
We end the introduction by describing in detail the contents of the paper.  
 Section \ref{sec:not} contains notations that will hold throughout the paper. 
 
  In Section \ref{sec:main}  we first provide and  study  the notions of sequential continuity of linear operators and sequential equicontinuity of families of linear operators on locally convex spaces. 
Then, we give  the definition of  $C_0$-sequentially (locally) equicontinuous semigroup in locally convex spaces. Next, we define the generator of the semigroup and the resolvent of the generator. In order to guarantee the existence of the resolvent,  the theory is developed under  Assumption \ref{ass:int}, requiring the existence  of the Laplace transform \eqref{eq:2015-05-01:16} as  Riemann integral (see Remark \ref{2015-10-29:13}). This assumption is immediately verified if the underlying space $X$ is sequentially complete. Otherwise, the Laplace transform always exists in the (sequential) completion of $X$ and then one should check that it  lies in $X$, as we do in Proposition \ref{2015-10-13:06}.
The properties of generator and resolvent  are  stated through a series of results: their synthesis is represented by Theorem 
  \ref{theo:2015-07-29:04}, stating that the semigroup is uniquely identified by its generator, and by Theorem  \ref{teo:2015-05-01:12}, stating that the resolvent coincides with the Laplace transform.
 Then we provide a generation theorem (Theorem \ref{theo:2015-05-02:10}), characterizing, in the same spirit  of the Hille-Yosida theorem, the linear operators generating  $C_0$-sequentially equicontinuous semigroups.
Afterwards, we show that the notion of bi-continuous semigroups can be seen as a specification of ours (Proposition \ref{prop:2015-05-02:05}).   Finally, we  provide some examples which illustrate  our notion  in relation to the others.

Section \ref{sec:SE}  implements the  theory of Section \ref{sec:main} in  spaces of bounded  Borel functions,  continuous and bounded functions, or uniformly continuous and  bounded functions defined on a metric space.  The main aim of this section is to find and study  appropriate locally convex topologies in these functional spaces allowing a comparison between our notion with the aforementioned other  ones.  We identify them in two topologies  belonging to a class  of locally convex topologies defined through the family of seminorms \eqref{maurosemi}.  We study the relation between them  and the topology induced by the uniform norm (Proposition \ref{prop:2015-04-29:01}). Then,  we  study these topological spaces through a series of results ending with Proposition \ref{prop:seqcom} and we  characterize  their topological dual  in Proposition \ref{prop:2015-04-27:00}. We end the section with the desired comparison:   in Subsections \ref{sec:wc}, \ref{2015-10-13:03}, and \ref{sub:goldys}, we show that the notions developed in  \cite{Cerrai94}, \cite{Priola99}, and \cite{GoldysKocan01} to treat Markov transition semigroups can be reintepreted in our framework.

Section \ref{sec:trans} applies the results of Section \ref{sec:SE} to transition semigroups. This is done, in Subsection \ref{sub2B}, in the space of bounded continuous functions endowed with the topology $\tau_\mathcal{K}$ defined in \eqref{tauk}. Then,   in Subsection \ref{2015-10-29:16}, we provide an extension to weighted spaces of continuous functions,  not necessarily bounded. Finally, in Subsection \ref{2015-10-29:16}, we treat the case of Markov transition semigroups associated to stochastic differential equations in Hilbert spaces. 

\vskip20pt

\section{Notation}
\label{sec:not}

\begin{enumerate}[(N1)]
\itemsep=-1mm
\item $X,Y$ denote Hausdorff topological vector spaces.
   Starting from Subsection \ref{SCO}, Assumption 
\ref{XYconv} will hold and $X,Y$ will be Hausdorff  locally convex topological vector spaces. 


\item The topological dual of a topological vector space $X$ is denoted by  $X^*$. 

\item If $X$ is a vector space and $\Gamma$ is a vector space of linear functionals on $X$ separating points in $X$, we denote by $\sigma(X,\Gamma)$ the weakest locally convex topology on $X$ making continuous the elements of $\Gamma$.

\item The weak topology on the topological vector space $X$ is denoted by $\tau_w$, that is $\tau_w\coloneqq \sigma (X,X^*)$.

\item If $X$ and $Y$ are topological vector spaces, the space of continuous operators from $X$ into $Y$ is denoted by $L(X,Y)$, and the space of sequentially continuous operators from $X$ into $Y$
 (see Definition \ref{def:2015-05-02:06})
 is denoted by $\mathcal{L}_0(X,Y)$. We also denote $L(X)\coloneqq L(X,X)$ and $\mathcal{L}_0(X)\coloneqq \mathcal{L}_0(X,X)$.
\item Given a locally convex topological vector space $X$, the symbol $\mathcal{P}_X$  denotes a family of seminorm on $X$ inducing the locally convex topology.

\item $E$ 
denotes a metric space; $\mathcal{E}\coloneqq \mathcal{B}(E)$ denotes the Borel $\sigma$-algebra of subsets of $E$ .

\item Given the metric space $E$, $\mathbf{ba}(E)$ denotes the space of finitely additive signed measures with bounded total variation on $\mathcal{E}$,
  $\mathbf{ca}(E)$ denotes the subspace of $\mathbf{ba}(E)$ of countably additive finite measure, and  $\mathbf{ca}^+(E)$ denotes the subspace of $ \mathbf{ca}(E)$ of positive countably additive finite measures. 

\item Given the metric space $E$, we denote by $B(x,r)$ the open ball centered at $x\in E$ and with radius $r$ and by $B(x,r]$ the closed ball centered at $x$ and with radius $r$.
\item The common symbol  $\mathcal{S}(E)$  denotes indifferently one of the spaces 
$B_b(E)$, $C_b(E)$, $UC_b(E)$, that is, respectively,  the space of real-valued \emph{bounded  Borel}  /  \emph{continuous and bounded} / \emph{uniformly continuous and  bounded} functions  defined on $E$.

\item  On $\mathcal{S}(E)$, we consider the sup-norm $|f|_\infty\coloneqq \sup_{x\in E} |f(x)|$, which makes it a Banach space. The topology on $\mathcal{S}(E)$ induced by such norm is denoted by $\tau_\infty$.

\item On $\mathcal{S}(E)$, the symbol  $\tau_\mathcal{C}$ denotes  the topology of the uniform convergence on compact sets.

\item By $\mathcal{S}(E)_\infty^*$ we denote the topological  dual of $(\mathcal{S}(E),|
\cdot|_\infty)$ and  by $|\cdot|_{\mathcal{S}(E)^*_\infty}$  the operator norm in $\mathcal{S}(E)^*_\infty$.
\end{enumerate}
\noindent We make use of  the conventions $\inf\emptyset=+\infty$, $\sup\emptyset=-\infty$, $1/\infty=0$. 

\section{$C_0$-sequentially equicontinuous semigroups}\label{sec:main}

 In this section, we  introduce and investigate the notion of $C_0$-sequentially equicontinuous semigroups on  locally convex topological vector spaces.

\subsection{Sequential continuity and equicontinuity}
\label{sub1}
We recall the notion of sequential continuity for functions and define the notion of sequential equicontinuity for families of functions   on topological spaces.
\begin{Definition}
\label{def:2015-05-02:06}
Let $X$, $Y$ be Hausdorff topological spaces.
\begin{enumerate}[(i)]
\item \label{2015-10-06:02}
A function  $f\colon X\rightarrow Y$ is said to be \emph{sequentially continuous} if,  for every  sequence $\{x_n\}_{n\in \mathbb{N}}$ converging to $x$ in $X$, we have 
$
f(x_n)\rightarrow f(x)$  in $Y$.

\item   \label{2015-10-06:03} If $Y$ is a vector space,
a family of functions  
$\mathcal{F}=\{f_\iota\colon X\rightarrow Y\}_{\iota\in \mathcal{I}}$  is said to be \emph{sequentially equicontinuous} if 
for every $x\in X$, for every  sequence $\{x_n\}_{n\in \mathbb{N}}$ converging to $x$ in $X$ and for every neighborhood $U$ of $0$ in $Y$, there exists $\overline n\in \N$ such that $f_\iota (x_n)\in f_\iota(x)+ U$ for every $\iota\in \mathcal{I}$ and $n\geq \overline n$.
\end{enumerate}

\end{Definition}


\begin{Remark}
  \label{rem:2015-05-01:06}
Let $E$ be a metric space. If $g:X\rightarrow Y$ is  sequentially continuous  and  $f:E\rightarrow X$ is continuous, then $g\circ f:E\rightarrow Y$ is continuous. It is sufficient to recall that  continuity for a function defined on a metric space is equivalent to sequential continuity.
\end{Remark}

  If $Y$ is a locally convex topological vector space, then Definition \ref{def:2015-05-02:06}\emph{(\ref{2015-10-06:03})} is equivalent to   
    \begin{equation}
\label{eq:2015-05-01:04}
\{x_n\}_{n\in\N}\subset X, \ x_n\rightarrow x \
\mbox{in} \ X \ \Longrightarrow \    \lim_{n\rightarrow +\infty}\sup_{\iota\in \mathcal{I}}q(f_\iota(x_n)-f_\iota(x))=0, \ \ \ \ \ \forall  q\in \mathcal{P}_Y,
  \end{equation}
  where $\mathcal{P}_Y$ is a set of seminorms inducing the topology on $Y$.
%
The characterization of sequential continuity \eqref{eq:2015-05-01:04} will be very often used throughout the paper.

\subsection{The space of sequentially continuous linear  operators}\label{SCO}
Starting from this subsection, we make the following
\begin{Assumption}\label{XYconv}
$X$ and $Y$ are Hausdorff locally convex topological vector spaces,
 and $\mathcal{P}_X$, $\mathcal{P}_Y$ denote  families of seminorms inducing the topology on $X$, $Y$, respectively. 
\end{Assumption}

\begin{Remark}\label{lemma1} 
 We recall that
a subset $B\subset X$ is bounded if and only if $\sup_{x\in B} p(x)<+\infty$ for every $p\in\mathcal{P}_X$ and that Cauchy (and, therefore, also convergent) sequences are bounded in $X$.  
\end{Remark}

We define the vector space 
$$\mathcal{L}_0(X,Y)\coloneqq \{ F\colon X\rightarrow Y\ \mbox{s.t.} \ F \ \mbox{is linear and sequentially continuous}\}.$$ 
We will use  $\mathcal{L}_0(X)$ to denote the space $\mathcal{L}_0(X,X)$. Clearly, we have the inclusion
\begin{equation}\label{sad}
 L(X,Y)\subset \mathcal{L}_0(X,Y).
 \end{equation}

We recall that a linear operator $F\colon X\rightarrow Y$ is a called  \emph{bounded} if $F(B)$ is bounded in $Y$ for each bounded subset $B\subset X$. As well known (see \cite[Th.\ 1.32, p.\ 24]{Rudin1991})
 \begin{equation}\label{ssd}
 F\in L(X,Y)\ \Longrightarrow \ F \ \mbox{is bounded}.
 \end{equation}
 On the other hand,
if  $X$ is bornological (see \cite[p.\ 95, Definition 4.1]{Osborne2014}), 
then,
by 
\cite[Ch.\ 4, Prop.\ 4.12]{Osborne2014}, also the converse holds true, that is
\begin{equation}\label{sad1}
X \ \mbox{bornological}, \ F\colon X \rightarrow Y \ \mbox{linear and  bounded} \ \Longrightarrow \ F\in L(X,Y).
\end{equation}
 
\begin{Proposition}
  \label{2015-07-30:02}
  Let $F\in \mathcal{L}_0(X,Y)$. Then 
  \begin{enumerate}[(i)]
  \item\label{2015-08-31:01} $F$ is a bounded operator;
  \item\label{2015-08-31:02}  $F$ maps Cauchy sequences into Cauchy sequences.
  \end{enumerate}
\end{Proposition}
\begin{proof}
  \emph{(\ref{2015-08-31:01})} See \cite[Ch.\ 4, Prop.\ 4.12]{Osborne2014}.

\emph{(\ref{2015-08-31:02})}
 Let  $\{x_n\}_{n\in \mathbb{N}}$ be a Cauchy sequence in $X$. In order to prove that $\{Fx_n\}_{n\in \mathbb{N}}$ is a Cauchy sequence in $Y$, we need to prove that, for every $q\in \mathcal{P}_Y$ and $\varepsilon>0$, there exists $\overline n$ such that $n,m\geq \overline n$ implies $q(F(x_m-x_n))\leq \varepsilon$. Fix $q\in \mathcal{P}_Y$ and $\epsilon>0$. As, by Remark \ref{lemma1}, $\{x_n\}_{n\in \mathbb{N}}$ is bounded in $X$, by \emph{(\ref{2015-08-31:01})} the sequence $\{Fx_n\}_{n\in \mathbb{N}}$ is bounded in $Y$. Then, for every $n\in \mathbb{N}$, we can choose $k_n\in \mathbb{N}$, with $k_n\geq n$, such that
\begin{equation}
  \label{eq:2015-07-30:03}
  q(F(x_{k_n}-x_n))+2^{-n}\geq \sup_{k\geq n}q(F(x_k-x_n)).
\end{equation}
Define $z_n\coloneqq x_{k_n}-x_n$, for $n\in \mathbb{N}$. As $\{x_n\}_{n\in \mathbb{N}}$ is a Cauchy sequence in $X$, we have $z_n\rightarrow 0$ as $n\rightarrow +\infty$. By sequential continuity of $F$, also $Fz_n\rightarrow 0$. Then \eqref{eq:2015-07-30:03} entails, for every $\overline n\in \mathbb{N}$ and every $n,m\geq \overline n$,
\begin{equation*}
    q(F(x_m-x_n))\leq
q(F(x_m-x_{\overline n}))+
q(F(x_n-x_{\overline n}))
\leq
2\sup_{k\geq \overline n}q(F(x_k-x_{\overline n}))\\
\leq
2^{1-\overline n}
+2q(Fz_{\overline n}).
\end{equation*}
Passing to the limit $\overline n\rightarrow +\infty$, we conclude that $\{Fx_n\}_{n\in \mathbb{N}}$ is a Cauchy sequence  in $Y$.
\end{proof}

\begin{Remark}
 We notice that the fact that $F$ is a bounded linear operator from $X$ into $Y$ does not guarantee, in general, that it belongs to the space $\mathcal{L}_0(X,Y)$.   Indeed, the bounded sets in the weak topology $\tau_w$ of any Banach space $X$ are exactly the originally bounded sets 
  (see 
  Lemma \ref{lem:2015-05-04:01}; actually this is true for locally convex spaces:\ see \cite[p.\ 70, Theorem 3.18]{Rudin1991}). Then, if $\tau$ denotes the norm-topology in $X$, the identity $\mathbf{id}\colon (X,\tau_w)\rightarrow (X,\tau)$ is bounded. Nevertheless, this identity is in general not sequentially continuous (any  infinite dimensional Hilbert space provides an immediate counterexample).
\end{Remark}

\begin{Corollary}\label{2015-10-04:02}
  If $X$ is  bornological, then 
  $$
\mathcal{L}_0(X,Y)=\mathcal{L}_0(X_w,Y_w)=L(X,Y)=L(X_w,Y_w),
$$ where $X_w,Y_w$ denote, respectively, the spaces $X$,$Y$ endowed with their weak topologies.
\end{Corollary}
\begin{proof}
Since $X$ (\emph{resp.}\  $Y$) is locally convex, by \cite[p.\ 70, Theorem 3.18]{Rudin1991}, the weakly bounded sets of $X$ (\emph{resp.}\  $Y$) are exactly the originally bounded sets in $X$ (\emph{resp.}\  in $Y$). Hence, \eqref{ssd} and \eqref{sad1} yield $L(X_w,Y_w)\subset L(X,Y)$. On the other hand, the opposite inclusion holds true for every $X,Y$ vector topological spaces. So, we have proved that  $L(X_w,Y_w)= L(X,Y)$.
 
 Now, 
by Proposition \ref{2015-07-30:02}\emph{(\ref{2015-08-31:01})} and by \eqref{sad1}, we have $\mathcal{L}_0(X,Y)\subset L(X,Y)$. The opposite inclusion is obvious. So, $\mathcal{L}_0(X,Y)=L(X,Y)$.

Finally, considering that $\mathcal{L}_0(X_w,Y_w)\supset L(X_w,Y_w)$, in order to conclude we need to show that $\mathcal{L}_0(X_w,Y_w)\subset L(X,Y)$.
 Recalling that the 
 weakly bounded sets of $X$ (\emph{resp.}\  $Y$) are exactly the originally bounded sets in $X$ (\emph{resp.}\  in $Y$), the latter follows from \eqref{sad1} and 
 Proposition \ref{2015-07-30:02}\emph{(\ref{2015-08-31:01})}, as $X$ is bornological.  
\end{proof}

Let $\mathbf{B}$ be the set of all bounded subsets of $X$.
We introduce on $\mathcal{L}_0(X,Y)$ a locally convex topology as follows.
By Proposition \ref{2015-07-30:02}\emph{(\ref{2015-08-31:01})}
\begin{equation}\label{semi}
\rho_{q,D}(F)\coloneqq\sup_{x\in D}q\left(Fx\right)
\end{equation}
is  finite for 
all $F\in \mathcal{L}_0(X,Y)$, $D\in \mathbf{B}$, and $q\in \mathcal{P}_Y$. Given $D\in \mathbf{B}$ and $q\in \mathcal{P}_Y$, \eqref{semi} defines a seminorm in the space $\mathcal{L}_0(X,Y)$.
We denote by  $\mathcal{L}_{0,b}(X,Y)$ the
space $\mathcal{L}_0(X,Y)$ endowed with the
 locally convex vector  topology
$\tau_b$ induced by the family of seminorms $\{\rho_{q,D}\}_{q\in \mathcal{P}_Y,\  D\in \mathbf{B}}$. We notice  that $\tau_b$ does not depend on the choice of family $\mathcal{P}_Y$ inducing the topology of $Y$.
Since $\mathbf{B}$ contains all singletons $\{x\}_{x\in X}$ and $Y$ is Hausdorff, also $\mathcal{L}_{0,b}(X,Y)$ is Hausdorff.

\begin{Proposition}\label{2015-10-08:05}
  The  map 
  $$\mathcal{L}_{0,b}(X)\times \mathcal{L}_{0,b}(X)\rightarrow \mathcal{L}_{0,b}(X),\ \ \ (F,G)\mapsto FG,$$ is sequentially continuous.
\end{Proposition}
\begin{proof}
  Let $(  F,  G)\in \mathcal{L}_0(X)\times \mathcal{L}_0(X)$, and let $D\subset X$ be bounded. Let $\{(F_n ,G_n )\}_{n \in \mathbb{N}}$
be a sequence converging to $(  F,  G)$ in $\mathcal{L}_{0,b}(X)\times \mathcal{L}_{0,b}(X)$ .
Consider the set $D'\coloneqq   \medcup_{n\in \mathbb{N}} G_nD  $. We have 
$$
\sup_{n\in\N} \sup_{x\in D} q(G_nx)
\leq  \sup_{n\in\N} \sup_{x\in D} q((G_n-G)x)+ \sup_{x\in D} q(Gx)\ \ \ \ \ \forall  q\in\mathcal{P}_X.
$$
 On the other hand,  $G_n\rightarrow G$ yields
 $$
\sup_{n\in\N}\sup_{x\in D} q((G_n-G)x)=
\sup_{n\in\N}\rho_{q,D}(G_n-G)<+\infty, \ \ \ \forall  q\in\mathcal{P}_X.
 $$
Then, combining with Proposition \ref{2015-07-30:02}\emph{(\ref{2015-08-31:01})}, we conclude that $D'$ is bounded.

 Now fix $q\in \mathcal{P}_X$.
For every $n\in \mathbb{N}$, we can write 
$$
\rho_{q,D}((  F  G-F_nG_n))\leq \rho_{q,D}(  F(  G- G_n))+\rho_{q,D}((  F- F_n) G_n)\leq
\rho_{q,D}(  F(  G- G_n))+\rho_{q,D'}(F- F_n).
$$
Now $\lim_{n\rightarrow +\infty}\rho_{q,D'}(F-F_n)=0$, because $D'\in\mathbf{B}$ and $F_n\rightarrow F$ in $\mathcal{L}_{0,b}(X)$. Hence we conclude if we show  $\lim_{n\rightarrow+\infty}\rho_{q,D}(F(G-G_n))=0$. Assume, by contradiction, that there exist $\varepsilon>0$, $\{x_k\}_{k\in \mathbb{N}}\subset D$, and a subsequence $\{G_{n_k}\}_{k\in \mathbb{N}}$, such that
\begin{equation}\label{contra}
q(F(G-G_{n_k})x_k)\geq \varepsilon\qquad \forall  k\in \mathbb{N}.
\end{equation}
Since 
$$
\lim_{n\rightarrow +\infty}q'((G-G_{n_k})x_k)\leq \lim_{n\rightarrow +\infty}\rho_{q',D}(G-G_{n_k})=0\qquad \ \forall   q'\in \mathcal{P}_X,
$$
then $\{z_k\coloneqq (G-G_{n_k}x_k)\}_{k\in \mathbb{N}}$ is a sequence converging to $0$ in $X$. By sequential continuity of $F$, we have
$
\lim_{k\rightarrow+\infty}q(Fz_k)=0,
$
contradicting \eqref{contra} and concluding the proof.
\end{proof}

\begin{Proposition}
\label{prop:2015-08-04:00}
  \begin{enumerate}[(i)]
  \item\label{2015-12-04:00}  If $Y$ is complete, then $\mathcal{L}_{0,b}(X,Y)$ is complete.
  \item\label{2015-10-08:03}  If $Y$ is sequentially complete, then $\mathcal{L}_{0,b}(X,Y)$ is sequentially complete.
  \end{enumerate}
\end{Proposition}
\begin{proof} 
\emph{(\ref{2015-12-04:00})} Let $\{F_\iota\}_{\iota\in \mathcal{I}}$ be a Cauchy net in $\mathcal{L}_{0,b}(X,Y)$. 
Then , by definition of $\tau_b$,  the net $\{F_\iota(x)\}_{\iota\in \mathcal{I}}$ is  Cauchy in $Y$, for every $x\in X$.
 Since $Y$ is complete, for every $x\in X$, the limit
$
F(x)\coloneqq \lim_{\iota}F_\iota(x)
$
exists in $Y$.
Clearly, $F$ is linear. Now we show that it is sequentially continuous. Let $q\in \mathcal{P}_Y$ and denote by $D$ the bounded set  $D\coloneqq \{x_n\}_{n\in \mathbb{N}}\subset X$, where $x_n\rightarrow 0$ in $X$.
Then 
\begin{equation*}
       q\left(Fx_n\right)=\lim_{\substack{\iota\\\iota\succeq \overline \iota  }}q\left(F_\iota x_n\right)
\leq 
\lim_{
\substack{\iota\\\iota\succeq \overline \iota  }
}q\left((F_\iota-F_{\overline \iota})x_n\right)+
 q\left(F_{\overline \iota}x_n\right)\leq
 \sup_{\iota\succeq \overline \iota}\rho_{q,D}\left(F_\iota-F_{\overline \iota}\right)+
 q\left(F_{\overline \iota} x_n\right),
 \qquad\ \forall   \overline \iota\in \mathcal{I},\ \ \forall   n\in \mathbb{N}.
\end{equation*}
Taking  the $\limsup_{n\rightarrow +\infty}$ in the inequality above  and taking into account that $\{F_\iota\}_{\iota\in\mathcal{I}}$ is a Cauchy net in $\mathcal{L}_{0,b}(X,Y)$  yield the sequential continuity of $F$.

We now show that $\lim_{\iota}F_\iota=F$ in $\mathcal{L}_{0,b}(X,Y)$. Let $D\in \mathbf{B}$
and let $q\in \mathcal{P}_Y$. 
We have
\begin{equation*}
       q\left((F-F_{{\overline \iota}})x\right)=
\lim_{
\substack{\iota\\\iota\succeq \overline \iota  }
}q\left((F_{ \iota}-F_{\overline \iota})x\right)
 \leq
 \sup_{\iota\succeq \overline \iota}\rho_{q,D}\left(F_{\iota}-F_{\overline \iota}\right)
\qquad \qquad\ \forall   \overline \iota\in \mathcal{I},\ \ \forall  
 x\in D,
\end{equation*}
and the conclusion follows as $\{F_\iota\}_{\iota\in\mathcal{I}}$ is a Cauchy net in $\mathcal{L}_{0,b}(X,Y)$.

\emph{(\ref{2015-10-08:03})} 
It follows by similar arguments as those above, taking now $Y$ sequentially complete and replacing $\mathcal{I}$
 by  $\mathbb{N}$. 
\end{proof}

\subsection{Families of  sequentially  equicontinuous functions}
\begin{Proposition}
  \label{prop:2015-05-03:00}
  For $n\in \mathbb{N}$ and $i=1,\ldots,n$, let 
  $\mathcal{F}^{(i)}=\{F^{(i)}_\iota:X\rightarrow X\}_{\iota\in \mathcal{I}_i}$ be families of sequentially equicontinuous linear operators. Then the following hold.
  \begin{enumerate}[(i)]
  \item\label{2015-12-04:01}
The family 
$
\mathcal{F}=\{F^{(1)}_{\iota_1}  F^{(2)}_{\iota_2} \ldots   F^{(n)}_{\iota_n}\colon X\rightarrow X\}_{\iota_1\in \mathcal{I}_1,\ldots, \iota_n\in \mathcal{I}_n}
$
is sequentially equicontinuous.
\item \label{2015-12-04:02}
The family
$
\mathcal{F}=\{F^{(1)}_{\iota_1}+F^{(2)}_{\iota_2}+\ldots +F^{(n)}_{\iota_n}\colon X\rightarrow Y\}_{\iota_1\in \mathcal{I}_1,\ldots, \iota_n\in \mathcal{I}_n}
$
is sequentially equicontinuous.
\item \label{2015-10-06:04}
The family $\mathcal{F}$ is equibounded, that is, if $D$ is a bounded subset of $X$, then
 $
    \big\{F^{(i)}_{\iota_i}x\big\}_{\substack{\iota_i\in \mathcal{I}_i,\\i=1,\ldots,n\\x\in D}}
$
  is bounded in $X$.
\end{enumerate}
\end{Proposition}
\begin{proof}
\emph{(\ref{2015-12-04:01})}  It suffices to prove the statement for $n=2$. By contradiction, assume that  there exist a sequence $\{x_k\}
_{k\in \mathbb{N}}$ converging to $0$ in $X$, sequences $\{\iota_1^{(k)}\}_{k\in \mathbb{N}}$ in $\mathcal{I}_1$
and
$\{\iota_2^{(k)}\}_{k\in \mathbb{N}}$ in $\mathcal{I}_2$, $p\in \mathcal{P}_X$,  and $\varepsilon>0$, such that
$$
p\left(
\left(F^{(1)}_{\iota^{(k)}_1} 
F^{(2)}_{\iota^{(k)}_2}\right)
x_k
\right) \geq  \varepsilon\qquad \ \forall   k\in \mathbb{N}.
$$
Since $\mathcal{F}^{(2)}$ is sequentially equicontinuous, 
we have
$$
  \limsup_{k\rightarrow +\infty}q\left(
F^{(2)}_{\iota^{(k)}_2}
x_k
\right)
 \leq
 \lim_{k\rightarrow +\infty}\sup_{\iota_2\in \mathcal{I}_2}
q\left(
F_{\iota_2}^{(2)}
x_k
\right) = 0, \qquad\forall  q\in \mathcal{P}_X.
$$
This means that the sequence $\big\{y_k\coloneqq F^{(2)}_{\iota_2^{(k)}}x_k\big\}_{k\in\mathbb{N}}$  converges to $0$ in $X$.
Then, in the same way,  since 
 $\mathcal{F}^{(1)}$ is sequentially equicontinuous,
\begin{equation*}
\limsup_{k\rightarrow +\infty}p\left(
\left(F^{(1)}_{\iota^{(k)}_1} 
F^{(2)}_{\iota^{(k)}_2}\right)
x_k
\right) 
=
\limsup_{k\rightarrow +\infty}p\left(
F^{(1)}_{\iota^{(k)}_1} 
y_k
\right)   \leq 
\limsup_{k\rightarrow +\infty}
\sup_{\iota_1\in \mathcal{I}_1}
p\left(
F_{\iota_1}^{(1)}
y_k
\right) = 0,
\end{equation*}
and the contradiction arises.

\emph{(\ref{2015-12-04:02})}  The proof follows by the triangular inequality.

\emph{(\ref{2015-10-06:04})} 
Assume,  by contradiction, that there exist a bounded set $D$ and
 $p\in\mathcal{P}_X$ such that
$$
\sup_{\substack{\iota_i\in \mathcal{I}_i\\i=1,\ldots,n,\\x\in D}}p\left(F^{(i)}_{\iota_i}x\right)=+\infty.
$$
Then there exist $\bar \imath\in \{1,\ldots,n\}$ and 
 sequences
$\{x_k\}_{k\in \mathbb{N}}\subset D$,
$\{\iota_k\}_{k\in \mathbb{N}}\subset \mathcal{I}_{\bar \imath}$,  
 such that
 \begin{equation}
   \label{eq:2015-07-30:04}
   p\left(F^{(\bar \imath)}_{\iota_k}x_k\right)\geq k,\qquad \ \forall   k\in \mathbb{N}.
 \end{equation}
 On the other hand, since $D$ is bounded, the sequence $\left\{\frac{x_k
}{k}\right\}_{k\in \mathbb{N}\setminus \{0\}}$ converges to $0$, and then, since the family $\big\{F^{(\bar \imath)}_\iota\big\}_{\iota\in \mathcal{I}_{\bar \imath}}$ is sequentially equicontinuous, we have
$$
\lim_{k\rightarrow +\infty}p\left(F^{(\bar \imath)}_{\iota_k}\frac{x_k}{k}\right)=0,
$$
which contradicts \eqref{eq:2015-07-30:04}, concluding the proof.
\end{proof}


The following proposition clarifies when the notion of sequential equicontinuity for a family of linear operators  is equivalent to the notion of equicontinuity.

\begin{Proposition}\label{2015-10-06:00}
Let $\mathcal{F}\coloneqq \{F_\iota:X\rightarrow X\}_{\iota\in \mathcal{I}}$
be a family of linear operators.
If $\mathcal{F}\subset L(X)$ is  equicontinuous, then
$\mathcal{F}\subset \mathcal{L}_0(X)$ and $\mathcal{F}$ 
is sequentially equicontinuous.

Conversely, 
if  $X$ is metrizable
 and $\mathcal{F}\subset \mathcal{L}_0(X)$ is sequentially equicontinuous, then $\mathcal{F}\subset L(X)$ and $\mathcal{F}$ is equicontinuous.
\end{Proposition}
\begin{proof}
  The first statement being obvious,  we will only show the second one.

 Assume that $\mathcal{F}$ is sequentially equicontinuous and that $X$ is metrizable.
Since $X$ is metrizable, we have $\mathcal{L}_0(X)=L(X)$.
Assume, by contradiction, that  $\mathcal{F}$ is not equicontinuous.
Since  the topology of $X$   is induced by a countable family of seminorms $\{p_n\}_{n\in \mathbb{N}}$
(see \cite[Th.\ 3.35, p.\ 77]{Osborne2014}), it then follows that
 there exist  a continuous seminorm $q$ on $X$ and sequences $\{x_n\}_{n\in \mathbb{N}}\subset X$, $\{\iota_n\}_{n\in \mathbb{N}}\subset \mathcal{I}$ such that
$$
\sup_{k=1,\ldots,n}p_k(x_n)<\frac{1}{n},\ \quad q(F_{\iota_n} x_n)>1, \ \ \ \ \ \ \ \ \forall   n\in \mathbb{N}.
$$
But then
$$
\lim_{n\rightarrow +\infty}x_n=0\qquad
\mathrm{and}
\qquad
\liminf_{n\rightarrow +\infty}\left(\sup_{\iota\in \mathcal{I}}q(F_{\iota} x_n)\right)\geq 
\liminf_{n\rightarrow +\infty}q(F_{\iota_n} x_n)\geq 1,
$$
which implies that $\mathcal{F}$ is not sequentially equicontinuous, getting a contradiction and concluding the proof.
\end{proof}

\subsection{$C_0$-sequentially  equicontinuous semigroups}\label{sub3}

We  now introduce the notion of  $C_0$-sequentially (locally) equicontinuous semigroups.

\begin{Definition}[$C_0$-sequentially (locally) equicontinuous semigroup]
  \label{def:2015-05-01:01}
A family  of linear operators (not necessarily continuous) $$T\coloneqq\{T_t\colon X\rightarrow X\}_{t\in \mathbb{R}^+}$$  is called a $C_0$-\emph{sequentially equicontinuous semigroup} on $X$  if the following properties hold.
  \begin{enumerate}[(i)]
  \item\label{2015-10-06:01} (Semigroup property) $T_0=I$ and $T_{t+s}=T_t  T_s$ for all $t,s\geq 0$.
  \item\label{2015-10-13:11}  ($C_0$- or strong continuity property)  $\lim_{t\rightarrow 0^+}T_tx=x$, \ for every $x\in X$.
  \item\label{2015-08-31:00} (Sequential equicontinuity) $T$ is a sequentially equicontinuous family.
  \end{enumerate}
 The family $T$ is said to be a  $C_0$-sequentially locally  equicontinuous semigroup if (\ref{2015-08-31:00})
 is replaced by
\begin{enumerate}[(i)$'$]
\setcounter{enumi}{2}
\item\label{2015-12-05:24}  (Sequential local equicontinuity) $\{T_t\}_{t\in [0,R]}$ is sequentially locally equicontinuous for every $R>0$.
\end{enumerate}
\end{Definition}


\begin{Remark}\label{rem..}
The
 notion of $C_0$-sequentially (locally)  equicontinuous semigroup that we introduced is clearly a generalization of the notion of $C_0$-(locally) equicontinuous semigroup considered, e.g.,  in  \cite[Ch.\ IX]{Yosida80}, \cite{Komura68}. By Proposition \ref{2015-10-06:00} the two notions coincide if  $X$ is metrizable.
 In order to motivate the introduction of
$C_0$-sequentially equicontinuous semigroups,
 we stress two facts.

\begin{enumerate}[(1)]
\item Even if a semigroup on a sequentially complete space is $C_0$-(locally) equicontinuous, proving this property might  be  harder than proving that it is only $C_0$-sequentially equicontinuous. For instance, in locally convex functional spaces with topologies defined by seminorms involving integrals, one can use integral convergence theorems for sequence of functions which do not hold for nets of functions.
 \item\label{2015-12-05:25}
The notion of $C_0$-sequentially equicontinuous semigroup is a genuine generalization of the notion of $C_0$-equicontinuous semigroup of \cite{Yosida80}, as shown by  Example \ref{ex:bb}. 
\end{enumerate}
\end{Remark}

As for $C_0$-semigroups in Banach spaces, given a $C_0$-sequentially locally equicontinuous semigroup $T$,  
we define
  $$
  \mathcal{D}(A)\coloneqq \left\{x\in X\colon\exists\ \lim_{h\rightarrow
      0^+}\frac{T_hx-x}{h}\in X\right\}.
  $$
  Clearly, $\mathcal{D}(A)$ is a linear subspace of $X$. Then, we define the linear operator $A\colon \mathcal{D}(A)\rightarrow X$
as
  $$
   Ax\coloneqq \lim_{h\rightarrow    0^+}\frac{ T_hx-x}{h},\qquad  x\in \mathcal{D}( A),
  $$
and call it the \emph{infinitesimal generator} of  $T$.



\begin{Proposition}
  \label{prop:2015-05-01:02a}
 Let  $T\coloneqq \{T_t\colon X\rightarrow X\}_{t\in \mathbb{R}^+}$ be a  $C_0$-{sequentially locally equicontinuous semigroup}. 
\begin{enumerate}[(i)]
\item\label{2015-10-06:05} For every $x\in X$, the function $Tx:\mathbb{R}^+\rightarrow X,\ t\mapsto T_tx$, is continuous. 
\item\label{2015-10-06:06} If $T$ is sequentially equicontinuous, then, 
for every $x\in X$, the function $T x:\mathbb{R}^+\rightarrow X,\ t\mapsto T_tx$, is bounded.
\end{enumerate} 
  \end{Proposition}
 \begin{proof}
 \emph{(\ref{2015-10-06:05})}
Let  $\{t_n\}_{n\in\N}\subset\R^+$ be a sequence converging from the 	right (resp., from the left) to $t\in \R$.
 By Definition \ref{def:2015-05-01:01}\emph{(\ref{2015-10-06:01})}, we have, for every $p\in\mathcal{P}_X$ and $x\in X$,
$$
p(T_{t_n}x-T_{t}x)\  =  \
p(T_t(T_{t_n-t}x-x))   \ \ \ \ (\mbox{resp.,} \ 
  p(T_{t_n}x-T_{t}x)\  = 
p(T_{t_n}(T_{t-t_n}x-x))). 
$$
 By Definition \ref{def:2015-05-01:01}\emph{(\ref{2015-10-13:11})}, $\{T_{t_n-t}x-x\}_{n\in \mathbb{N}}$ (\emph{resp.}\  $\{T_{t-t_n}x-x\}_{n\in \mathbb{N}}$) converges to $0$.
Now conclude by using local sequential equicontinuity and \eqref{eq:2015-05-01:04}.

\emph{(\ref{2015-10-06:06})} This is provided by Proposition \ref{prop:2015-05-03:00}\emph{(\ref{2015-10-06:04})}.
\end{proof}

As well known, unlike the Banach space case, in locally convex spaces  the passage from $C_0$-locally \emph{equicontinuous} semigroups  
to  $C_0$-\emph{equicontinuous} semigroups through a renormalization with an exponential function is not obtainable in general (see Examples \ref{2015-10-13:05} and \ref{example2} in Subsection \ref{sec:examples}). Nevertheless, we have the following  partial result.

\begin{Proposition}
\label{prop:2015-08-01:00}
Let $\tau$ denote the locally convex topology on $X$ and let $|\cdot|_X$ be a norm on $X$. 
Assume 
that a set is $\tau$-bounded 
if and only if it is
$|\cdot|_X$-bounded.
  Let $T$ be a $C_0$-sequentially locally  equicontinuous semigroup on $(X,\tau)$.
%
%
\begin{enumerate}[(i)]
\item\label{2015-10-13:14}  If there exist $\alpha\in \R$ and $M\geq 1$ such that
  \begin{equation}
    \label{eq:2015-08-06:00}
    |T_t|_{L((X,|\cdot|_X))}\leq Me^{\alpha t},\ \ \  \ \forall    t\in \R^+,
  \end{equation}
then, for every $\lambda>\alpha$, the family $\{e^{-\lambda t}T_t:(X,\tau)\rightarrow (X,\tau)\}_{t\in \mathbb{R}^+}$ is a $C_0$-sequentially equicontinuous semigroup.
\item\label{2015-10-13:00}
If $(X,|\cdot|_X)$ is Banach, then  there exist $\alpha\in \R$ and $M\geq 1$ such that 
\eqref{eq:2015-08-06:00} holds.
\end{enumerate}
\end{Proposition}
\begin{proof}
\emph{(\ref{2015-10-13:14})} 
  Let $\lambda>\alpha$ and let $\{x_n\}_{n\in \mathbb{N}}$ be a sequence converging to $0$ in $(X,\tau)$. Then $\{x_n\}_{n\in\mathbb{N}}$
is bounded in $(X,\tau)$,  thus, by assumption, also in $(X,|\cdot|_X)$.
Set  $N\coloneqq \sup_{n\in \mathbb{N}}| x_n|_X $ and let $p\in \mathcal{P}_{(X,\tau)}$. 
  Then 
  \begin{equation*}
    \begin{split}
\sup_{t\in\mathbb{R}^+}p(e^{-\lambda t}T_tx_n)&\leq
\sup_{0\leq t\leq s}p(e^{-\lambda t}T_tx_n)+
\sup_{t>s}p(e^{-\lambda t}T_tx_n)\\
&\leq
\sup_{0\leq t\leq s}p(e^{-\lambda t}T_tx_n)+
L_pe^{(\alpha-\lambda)s}MN,
\end{split}
\end{equation*}
where $L_p\coloneqq \sup_{x\in X\setminus \{0\}}p(x)/|x|_X$ is finite, because $|\cdot|_X$-bounded sets are $\tau$-bounded.
Now we can conclude by applying to the right hand side of the inequality above first the $\limsup_{n\rightarrow +\infty}$ and considering that $T$ is a $C_0$-sequentially locally  equicontinuous semigroup on $(X,\tau)$, then  the $\lim_{s\rightarrow +\infty}$ and taking into account that $\lambda>\alpha$.

\emph{(\ref{2015-10-13:00})} By assumption, the bounded sets of $(X,|\cdot|_X)$ coincide with the bounded sets of $(X,\tau)$. By Proposition \ref{2015-07-30:02}\emph{(\ref{2015-08-31:01})}, we then have $\mathcal{L}_0((X,\tau))\subset L((X,|\cdot|_X))$. In particular $T_t\in L((X,|\cdot|_X))$, for all $t\in \mathbb{R}^+$.
Now,  by Proposition 
\ref{prop:2015-05-01:02a}\emph{(\ref{2015-10-06:05})}, the set $\{T_tx\}_{t\in [0,t_0]}$ is compact in $(X,\tau)$ for every $x\in X$ and $t_0>0$, hence bounded. We can then apply the Banach-Steinhaus Theorem in $(X,|\cdot|_X)$ and conclude that there exists $M\geq 0$ such that $|T_t|_{L((X,|\cdot|_X))}\leq M$ for all $t\in[0,t_0]$.
The conclusion now follows in a standard way from the semigroup property. 
\end{proof}

From here on in this subsection and in Subsections \ref{gen}-\ref{sub4}, unless differently specified, we will deal  with $C_0$-sequentially equicontinuous semigroups and, to simplify the exposition, we will adopt a standing notation for them and their generator, that is 
\begin{itemize}
\item
$T=\{T_t\}_{t\in {\R^+}}$ denotes a $C_0$-sequentially equicontinuous semigroup;
\item  $A$ denotes the infinitesimal generator of $T$.
\end{itemize}

Also, unless differently specified, from here on in this subsection and in Subsections \ref{gen}-\ref{sub4}, we will assume the following
\begin{Assumption}\label{ass:int}
For every $x\in X$
and $\lambda>0$, there exists the generalized Riemann integral in $X$\ (\footnote{That is, for every $a\geq 0$, the Riemann integral
$
\int_0^ae^{-\lambda t}T_txdt
$
exists in $X$, and the limit
$
\int_0^{+\infty}e^{-\lambda t}T_txdt\coloneqq \lim_{a\rightarrow+\infty}
\int_0^ae^{-\lambda t}T_txdt
$
exists in $X$.})
   \begin{equation}
       \label{eq:2015-05-01:16}
 R(\lambda)x\coloneqq \int_0^{+\infty} e^{-\lambda t}T_tx dt. 
  \end{equation}
\end{Assumption}

\begin{Remark}\label{2015-10-29:13}
By Proposition \ref{prop:2015-05-01:02a}, the generalized Riemann integral \eqref{eq:2015-05-01:16} always exists in the sequential completion of $X$. In particular, Assumption \ref{ass:int} is satisfied  if $X$ is sequentially complete.
\end{Remark}

For every $p\in \mathcal{P}_X$, and every $\lambda,\hat\lambda\in (0,+\infty)$,  we have the following inequalities, whose proof is straightforward, by triangular inequality and  definition of Riemann integral, and by recalling 
Proposition \ref{prop:2015-05-01:02a}:
\begin{align}
  p(R(\lambda)x-y) &\leq \int_0^{+\infty} e^{-\lambda t}p(T_tx-\lambda y )dt, \qquad \ \forall  x,y\in X\label{2015-10-07:02}\\
  p(R(\lambda)x-R(\hat \lambda)x) &\leq \int_0^{+\infty} |e^{-\lambda
    t}-e^{-\hat \lambda t}| p(T_tx)dt,\qquad \ \forall  x\in X.\label{2015-10-07:01}
\end{align}


\begin{Proposition}
  \label{prop:2015-05-01:07}
 If $L\in \mathcal{L}_0(X,Y)$, 
    then $\mathbb{R}^+\rightarrow Y, x \mapsto  LT_tx$ is continuous and  bounded. Moreover, for every
    $x\in X$, every $a\geq 0$, and every $\lambda>0$,
    \begin{equation}
      \label{eq:2015-05-01:08}
      L\int_0^a e^{-\lambda t}T_tx  d t =  \int_0^a e^{-\lambda t}LT_tx d t\qquad \text{and}\qquad
      L \int_0^{+\infty} e^{-\lambda t}T_tx d t = \int_0^{+\infty} e^{-\lambda t}LT_t d t,
    \end{equation}
where the Riemann integrals on the right-hand side of the equalities exist in $Y$.
\end{Proposition}

\begin{proof}  
Continuity of the map $\mathbb{R}^+\rightarrow X,\ t \mapsto LT_tx$, follows from Remark \ref{rem:2015-05-01:06}, from sequential continuity of $L$ and from 
Proposition \ref{prop:2015-05-01:02a}\emph{(\ref{2015-10-06:05})}.
By Proposition \ref{prop:2015-05-01:02a}\emph{(\ref{2015-10-06:06})},
we have that $\{T_tx\}_{t\in \mathbb{R}^+}$ is bounded, for all $x\in X$.
From Proposition \ref{2015-07-30:02}\emph{(\ref{2015-08-31:01})}, it then follows that 
$\{LT_tx\}_{t\in \mathbb{R}^+}$ is bounded.

Let $\{\pi^k\}_{k\in \mathbb{N}}$ be a sequence of partitions of   $[0,a]\subset \mathbb{R}^+$ of the form $\pi^k\coloneqq \{0=t^k_0<t^k_1<\ldots<t^k_{n_k}=a\}$, with $|\pi^k|\rightarrow 0$ as $k\rightarrow +\infty$, where $|\pi^k|\coloneqq \sup\{|t_{i+1}-t_i|\colon i=0,\ldots,n_{k}-1\}$.
Then, by recalling Assumption \ref{ass:int} and by continuity of $\mathbb{R}^+\rightarrow X,\ t \mapsto  T_tx$, we have in $Y$
$$
\int_0^a e^{-\lambda t}T_tx  d t =  \lim_{k\rightarrow +\infty} \sum_{i=0}^{n_k-1} e^{-\lambda t_i^k}T_{t_i^k} x (t_{i+1}^k-t_i^k). 
$$
By sequential continuity of $L$ we then have
\begin{equation}
  \label{eq:2015-10-06:09}
  L\int_0^a e^{-\lambda t}T_tx  d t =  \lim_{k\rightarrow +\infty} \sum_{i=0}^{n_k-1} e^{-\lambda t_i^k} L T_{t_i^k} x (t_{i+1}^k-t_i^k).
\end{equation}
Since
$\mathbb{R}^+\rightarrow X,\ t \mapsto  LT_tx$
is continuous, equality
 \eqref{eq:2015-10-06:09}
entails that $\mathbb{R}^+\rightarrow X,\ t \mapsto  e^{-\lambda t}LT_tx$ is Riemann integrable and that the first equality of \eqref{eq:2015-05-01:08} holds true.

The second equality of \eqref{eq:2015-05-01:08} follows from the  first one and from sequential continuity of $L$, by letting $a\rightarrow +\infty$ .
\end{proof}

\begin{Proposition}
  \label{prop:2015-05-01:02}
  \begin{enumerate}[(i)]
\item\label{2015-10-08:00} For every $\lambda>0$, the operator $R(\lambda):X\rightarrow X$ is linear and sequentially continuous.
\item\label{2015-10-13:02}  For every $x\in X$, the function $(0,+\infty)\rightarrow X,\ \lambda\mapsto R(\lambda)x$, is continuous.
\end{enumerate}
\end{Proposition}
\begin{proof} 
\emph{(\ref{2015-10-08:00})} 
 The linearity of $R(\lambda)$ is clear.
It remains to show its sequential continuity. Let $\{x_n\}_{n\in \mathbb{N}}\subset X$ be a sequence convergent to $0$. 
Then, for all $p\in \mathcal{P}_X$,
$$
\lim_{n \rightarrow +\infty}p(R(\lambda)x_n) \leq  \lim_{n\rightarrow +\infty}\int_0^{+\infty} e^{-\lambda t}p(T_tx_n) dt = \lambda^{-1}\lim_{n\rightarrow +\infty}\sup_{t\in\mathbb{R}^+}p(T_t x_n) = 0
$$
where the last limit is obtained by sequential equicontinuity and by recalling \eqref{eq:2015-05-01:04}.

\emph{(\ref{2015-10-13:02})}
 For $p\in \mathcal{P}_X$, $x\in X$, $\lambda,
\hat \lambda\in (0,+\infty)$,
 by
\eqref{2015-10-07:01},
$$
p\left(R(\lambda)x-R(\hat\lambda)x\right) \leq  \int_0^{+\infty} |e^{-\lambda t}-e^{-\hat\lambda t}|p(T_tx)dt \leq 
\sup_{r\in\mathbb{R}^+}p(T_rx)
\int_0^ {+\infty} |e^{-\lambda t}-e^{-\hat\lambda t}|dt.
$$
The last integral converges to $0$ as $\lambda\rightarrow \hat \lambda$, and we conclude as $\sup_{r\in\R^+} p(T_rx)<+\infty$ by Proposition   \ref{prop:2015-05-01:02a}\emph{(\ref{2015-10-06:06})}.
\end{proof}

The following proposition 
will be used 
in Subsection
\ref{sec:wc}
 to fit the theory of weakly continuous semigroups of \cite{Cerrai94, Cerrai01book}. 
\begin{Proposition}
  \label{prop:2015-07-30:00}
Let $C\subset X$ be sequentially closed, convex, and containing the origin, let $\hat t>0$, and  let $x\in X$. If $T_tx\in C$ for all $t\in[0,\hat t]$, then
\begin{equation}
  \int_0^{\hat t}e^{-\lambda  t}T_tx dt \in \frac{1}{\lambda}C,\ \ \  \forall\lambda>0\label{2015-10-13:18}.
\end{equation}  
 If $T_tx\in C$ for all $t\in\R^+$ then, 
 \begin{equation}
R(\lambda)x\in \frac{1}{\lambda}C,\ \ \ 
\ \forall  \lambda>0.  \label{eq:2015-10-13:17}
\end{equation}
\end{Proposition}
\begin{proof}
We prove the first claim, as the second one is a straightforward consequence of it because of the sequential completeness of $C$.
Let $\hat t>0$.
The Riemann integral in \eqref{2015-10-13:18} is the limit of a sequence of Riemann sums $\{\sigma(\pi^k)\}_{k\in \N}$ of the form
$$
\sigma(\pi^k) =  \sum_{i=1}^{m_k}e^{-\lambda t_i^k}(t^k_{i}-t^k_{i-1}) T_{t_i^k}x,
$$
with $\pi^k\coloneqq\{0=t_0^k<t_1^k<\ldots <t^k_{m_k}=\hat t\}$
 and  $|\pi^k|\rightarrow 0$ as $k\rightarrow +\infty$,
where $|\pi^k|\coloneqq \sup\{|t_{i}-t_{i-1}|\colon i=1,\ldots,m_k\}$.
Then, by sequential closedness of $C$, we are reduced to show that $\sigma(\pi^k)\in \frac{1}{\lambda}C$ for every $k\in\N$. Denote
$$
\alpha_k \coloneqq \sum_{i=1}^{m_k}e^{-\lambda t_i^{k}}(t^k_{i}-t_{i-1}^k), \ \ \ \ \ \forall  k\in \N.
$$
Then 
$$
0<\alpha_k <\int_0^{+\infty}e^{-\lambda t}dt=\lambda^{-1}, \ \ \ \ \forall  k\in\N.
$$
As $\sigma(\pi^k)/\alpha_k $ is a convex combination of the elements $\{T_{t_i^k}x\}_{i=1,\ldots,m_k}$, which  belong to $C$ by assumption, recalling that $C$ is convex and contains the origin, we conclude $\sigma(\pi^k)\in \alpha_k C\subset \frac{1}{\lambda}C$, for every $k\in\N$, and the proof is complete.
\end{proof}

\subsection{Generators of $C_0$-sequentially  equicontinuous semigroups}\label{gen}

%
%
%


In this subsection we study the generator $A$ of the $C_0$-sequentially equicontinuous semigroup $T$.

Recall that a subset $U$ of a topological space $Z$ is said to be sequentially dense in $Z$ if, for every $z\in Z$, there exists a sequence $\{u_n\}_{n\in \mathbb{N}}\subset U$ converging to $z$ in $Z$.
In such a case, it is clear that $U$ is also  dense in $Z$.

\begin{Proposition}
  \label{prop:2015-05-01:10}
      $\mathcal{D}(A)$ is sequentially dense in $X$.
\end{Proposition}
\begin{proof}
Let $\lambda>0$ and set $\psi_\lambda \coloneqq \lambda R(\lambda)\in X$.
By  \eqref{eq:2015-05-01:08},
$$
T_hR(\lambda)x = \int_0^{+\infty} e^{-\lambda t}T_{h+t}x dt\ \in \ X,\qquad \ \forall   x\in X.
$$
Then, following the proof of \cite[p.\ 237, Theorem 1]{Yosida80}(\footnote{In the cited result, $X$ is assumed sequentially complete. However, the completeness of $X$   is
used in the proof only
to define the integrals. In our case, existence for the integrals involved in the proof holds by assumption.
}), we have
$$
\frac{T_h \psi_\lambda x-\psi_\lambda x}{h} = 
\frac{e^{\lambda h}-1}{h}\left(
\psi_\lambda x-\lambda \int_0^h e^{-\lambda t} T_tx dt
\right)
-\frac{\lambda}{h}\int_0^h e^{-\lambda t}T_t xdt \ \in \ X,
\qquad \ \forall   x\in X.
$$
Passing to the limit for $h\rightarrow 0^+$, we obtain
$$
\lim_{h\rightarrow 0^+}
\frac{T_h \psi_\lambda x-\psi_\lambda x}{h}
=
\lambda (\psi_\lambda x-x) \ \in \ X,
\qquad \ \forall  x\in X.
$$
Then $\psi_\lambda x\in \mathcal{D}(A)$ and 
\begin{equation}
  \label{eq:2015-05-01:18}
  A\psi_\lambda x =  \lambda (\psi_\lambda -I)x \ \in \ X,
\qquad \ \forall   x\in X.
\end{equation}
For future reference, we notice that this shows, in particular, that
\begin{equation}
  \label{eq:2015-07-26:00}
  \operatorname{Im}(R(\lambda)) \subset  \mathcal{D}(A).
\end{equation}
Now we prove that 
\begin{equation}\label{eq:2015-07-28:05}
\lim_{\lambda\rightarrow +\infty}  \psi_\lambda x
=x\qquad \ \forall  x\in X, 
\end{equation}
which concludes the proof.
By \eqref{2015-10-07:02}, we have
$$
p(\psi_\lambda x-x)=
\lambda p(R(\lambda) x-\lambda^{-1}x)\leq
\int_0^{+\infty}\lambda
e^{-\lambda t}p(T_tx-x)dt
\quad \ \forall   x\in X, \ \ \forall  p\in \mathcal{P}_X.
$$
By Proposition \ref{prop:2015-05-01:02a}\emph{(\ref{2015-10-06:06})}, we can apply 
 the dominated convergence theorem  to the last integral above when
$\lambda\rightarrow+\infty$. Then 
 we have
 $$
p(\psi_\lambda x-x) \rightarrow 0, \ \ \ \  \ \forall   x\in X, \ \ \forall  p\in \mathcal{P}_X,
$$
 and we obtain \eqref{eq:2015-07-28:05} by arbitrariness of $p\in \mathcal{P}_X$.
\end{proof}

\begin{Remark}\label{2015-12-07:03}
  We notice that, if $X$ is sequentially complete, then Proposition \ref{prop:2015-05-01:10} can be refined. Indeed, as for $C_0$-semigroups in Banach spaces, we can define $D_\infty\coloneqq \medcap_{n=1}^{+\infty}D(A^n)$. If $X$ is sequentially complete, then,
for every $\varphi\in C^{\infty}_c((0,+\infty))$ and every $x\in X$,
 we can define the integral
$$
\varphi_Tx\coloneqq \int_0^{+\infty}\varphi(t)T_txdt.
$$
Then one can show that $\varphi_Tx\in D_\infty$, $A^n\varphi_Tx=(-1)^n(\varphi^{(n)})_Tx$, for all $n\geq 1$, and the set $\{\varphi_Tx\colon \varphi\in C^\infty_c((0,+\infty)),\ x\in X\}$  is sequentially dense in $X$.
\end{Remark}

\begin{Proposition}
  \label{prop:2015-05-01:11}
  Let $x\in \mathcal{D}(A)$. Then 
  \begin{enumerate}[(i)]
  \item\label{2015-12-05:00}
  $T_tx\in \mathcal{D}(A)$ for all $t\in \R^+$;
  \item the map $T x\colon \mathbb{R}^+\rightarrow X,\  t\mapsto T_tx$ is differentiable;
  \item the following identity holds
  \begin{equation}
    \label{eq:2015-05-01:20}    
\frac{d}{dt}T_tx
= AT_tx =  T_tAx,\qquad \ \forall   t\in \R^+.
\end{equation}
\end{enumerate}
\end{Proposition}
\begin{proof}
Let $x\in \mathcal{D}(A)$. Consider the function 
  $\Delta:\R^+\rightarrow X$
defined by 
$$\Delta(h) \coloneqq \  \begin{dcases} 
\frac{T_h-I}{h}x, &\mbox{if} \ h\neq 0\\
\Delta(0)=Ax.&
\end{dcases}$$ 
This function is continuous by definition of $A$.
Then,  by Remark \ref{rem:2015-05-01:06}, 
 $$T_t Ax =  T_t\lim_{h\rightarrow 0^+} \Delta(h)  =   \lim_{h\rightarrow 0^+} T_t \Delta(h) =  \lim_{h\rightarrow 0^+}  \frac{T_hT_tx-T_tx}{h}, \qquad \forall   t\in \R^+,$$
 which shows that \emph{(\ref{2015-12-05:00})} holds and  that
$$T_t Ax = AT_t x, \qquad \forall   t\in \R^+.$$
  The rest of the proof follows exactly as in  \cite[p.\ 239, Theorem 2]{Yosida80}.
\end{proof}

We are going to show that the infinitesimal generator identifies uniquely the semigroup $T$. For that, we need the following lemma, which will be also used afterwards.
\begin{Lemma}
  \label{prop:2015-05-03:12}
  Let $0\leq a<b$, $f,g\colon(a,b)\rightarrow \mathcal{L}_0(X)$, $t_0\in (a,b)$, and $x\in X$. 
   Assume that
   \begin{enumerate}[(i)]
   \item\label{2015-12-05:01} the family $\{f(t)\}_{t\in[a',b']}$ is sequentially equicontinuous, for every $a<a'<b'<b$;
   \item $g(\cdot)x\colon (a,b)\rightarrow X$ is differentiable at $t_0$; 
   \item $f(\cdot)g(t_0)x\colon(a,b)\rightarrow X$ is differentiable at $t_0$. 
   \end{enumerate}
   Then there exists the derivative of $f(\cdot)g(\cdot)x\colon(a,b)\rightarrow X$ at $t=t_0$ and 
$$
 \frac{d}{dt}\left[f(t)g(t)x)\right]|_{t=t_0} = \frac{d}{dt}\left[f(t)g(t_0)x\right]|_{t=t_0}+f(t_0)\frac{d}{dt}[g(t)x]|_{t=t_0}.
$$

\end{Lemma}
\begin{proof}
  For $h\in \mathbb{R}\setminus \{0\}$ such that $[t_0-|h|,t_0+|h|]\subset (a,b)$, write
  \begin{equation*}
    \begin{split}
      f(t_0+h)g(t_0+h)x-f(t_0)g(t_0)x=&
f(t_0+h)\left(g(t_0+h)-g(t_0)-h\frac{d}{dt}[g(t)x]|_{t=t_0}\right)\\
&+
hf(t_0+h)\frac{d}{dt}[g(t)x]|_{t=t_0}
+
\left(f(t_0+h)-f(t_0)\right)g(t_0)x\\
&\eqqcolon I_1(h)+I_2(h)+I_3(h).
\end{split}
\end{equation*}
Letting $h\rightarrow 0$, we have  $h^{-1}I_2(h)\rightarrow f(t_0)\frac{d}{dt}[g(t)x]|_{t=t_0}x$ and $h^{-1}I_3(h)\rightarrow \frac{d}{dt}[f(t)g(t_0)x]|_{t=t_0}$. Moreover, 
$$
p\left(
h^{-1}I_1(h)
\right)\leq
\sup_{s\in [t_0-|h|,t_0+|h|]}p\left(
f(s)\left(
{
g(t_0+h)-g(t_0)
\over
h
}-\frac{d}{dt}[g(t)x]|_{t=t_0}\right)x
\right),\quad \ \forall   p\in \mathcal{P}_X,
$$
and the member at the right-hand side of the inequality above tends to $0$ as $h\rightarrow 0$, because of   sequential local  equicontinuity of the family $\{f(s)\}_{s\in (a,b)}$ (part \emph{(\ref{2015-12-05:01})} of the assumptions) and because of differentiability of $g(\cdot)x$ in $t_0$.
\end{proof}

\begin{Theorem}
  \label{theo:2015-07-29:04}
  Let $S$ be a $C_0$-sequentially equicontinuous semigroup on $X$ with infinitesimal generator $A_S=A$. Then $S=T$.
\end{Theorem}
\begin{proof}
 For $t>0$ and $x\in \mathcal{D}(A)$, consider the function  $f\colon [0,t]\rightarrow X,\ s\mapsto T_{t-s}S_sx$. By Proposition \ref{prop:2015-05-01:11} and Lemma \ref{prop:2015-05-03:12}, $f'(s)=0$ for all $s\in [0,t]$, and then $T_tx=f(0)=f(t)=S_tx$. Since $ \mathcal{D}(A)$ is sequentially dense in $X$ and the operators $T_t$, $S_t$ are sequentially continuous, we have $T_tx=S_tx$ for all $x\in X$, and we conclude by arbitrariness of $t>0$.~
\end{proof}

\begin{Definition}
  \label{def:2015-05-01:19}
Let $\mathcal{D}(C)\subset X$ be a linear subspace. For a  linear operator $C\colon \mathcal{D}(C)\rightarrow X$, we define the \emph{spectrum}  $\sigma_0(C)$ as the set of  $\lambda\in \R$ such that one of the following holds:
\begin{enumerate}[(i)]
\item
 $\lambda-C$ is not one-to-one;
 \item
 $\operatorname{Im} (\lambda - C)\neq X$;
 \item
there exists $(\lambda -C)^{-1}$, but it is not sequentially continuous. 
\end{enumerate}
We denote $\rho_0(C)\coloneqq  \mathbb{R}\setminus \sigma_0(C)$, and call it \emph{resolvent set}  of $C$. If $\lambda\in \rho_0(C)$, we denote by $R(\lambda,C)$ the sequentially continuous inverse $(\lambda-C)^{-1}$ of $\lambda-C$.
\end{Definition}

\begin{Theorem}
  \label{teo:2015-05-01:12}
If $\lambda>0$, then
$\lambda\in \rho_0(A)$ and
 $R(\lambda,A) =R(\lambda)$. 
\end{Theorem}
\begin{proof}
\underline{\emph{Step 1.}} Here we show that $\lambda-A$ is one-to-one for every $\lambda>0$.
 Let  $x\in \mathcal{D}(A)$.
By Proposition \ref{prop:2015-05-01:11}, for any ${f}\in X^*$, the function $F\colon \mathbb{R}^+\rightarrow \mathbb{R},\ t\mapsto {f}(e^{-\lambda t} T_tx)$ is differentiable, and $F'(t)=f(e^{-\lambda t}T_t(A-\lambda)x)$. If $(A-\lambda )x=0$, then $F$ is constant.
By Proposition \ref{prop:2015-05-01:02a}\emph{(\ref{2015-10-06:06})},
 $F(t)\to 0$ as $t\to +\infty$, hence it must be $F\equiv 0$.
 Then $f(x)=F(0)=0$.
As $f$ is arbitrary, we conclude that $x=0$ and, therefore, that $\lambda-A$ is one-to-one.

\underline{\emph{Step 2.}} Here we show that $\lambda-A$ is invertible and $R(\lambda,A)=R(\lambda)$, for every $\lambda>0$.
By \eqref{eq:2015-07-26:00} 
and
\eqref{eq:2015-05-01:18},
\begin{equation}
  \label{eq:2015-05-01:21}
  (\lambda-A)R(\lambda)=I
\end{equation}
which shows that $\lambda-A$ is onto, and then invertible (by recalling also Step 1), and that
$(\lambda-A)^{-1}=R(\lambda)$.

\underline{\emph{Step 3.}} The fact $(\lambda-A)^{-1}\in \mathcal{L}_0(X)$ 
follows  
from Step 2 and 
 Proposition \ref{prop:2015-05-01:02}\emph{(\ref{2015-10-08:00})}.
\end{proof}

\begin{Corollary}
  \label{cor:2015-05-01:13}
The operator $A$ is  sequentially closed, that is,
its graph $\operatorname{Gr}(A)$ is sequentially closed in $X\times X$.
\end{Corollary}
\begin{proof}
 Observe that $(x,y)\in \operatorname{Gr}(A)$ if and only if $(x, x-y)\in \operatorname{Gr}(I-A)$, and hence if and only if $( x-y,x)\in \operatorname{Gr}(R(1,A))$. As $R(1,A)\in \mathcal{L}_0(X)$, then its graph is sequentially closed in $X\times X$, and we conclude.
\end{proof}

\begin{Corollary}\label{cor:2015-07-28:06}
We have the following.
\begin{enumerate}[(i)]
\item\label{2015-12-05:02}  $AR(\lambda,A)x=\lambda R(\lambda,A)x-x$, for all $\lambda >0$ and $x\in X$.
\item\label{2015-12-05:03}  $R(\lambda,A)Ax=AR(\lambda,A)x$, for all $\lambda>0$ and $x\in \mathcal{D}(A)$.
\item\label{2015-12-05:04}  (Resolvent equation) 
For every  $\lambda>0$ and $\mu>0$,
  \begin{equation}
    \label{eq:2015-05-03:10}
    R(\lambda,A)-R(\mu,A)=(\mu-\lambda)R(\lambda,A)R(\mu,A).
  \end{equation}
\item\label{2015-12-05:05}  
For every $x\in X$, $\lim_{\lambda\rightarrow \infty}\lambda R(\lambda,A)x=x$.
  \end{enumerate}
  \end{Corollary}
  \begin{proof}
\emph{(\ref{2015-12-05:02})}  It follows from  \eqref{eq:2015-05-01:21}. 

 \emph{(\ref{2015-12-05:03})} By \emph{(\ref{2015-12-05:02})} 
and considering that $x\in \mathcal{D}(A)$, we can write
$$AR(\lambda,A)x=\lambda R(\lambda,A)x-x=\lambda R(\lambda,A)x-R(\lambda,A)(\lambda-A)x=R(\lambda,A)Ax.$$

 \emph{(\ref{2015-12-05:04})} It follows from \emph{(\ref{2015-12-05:02})} by standard algebraic computations.

\emph{(\ref{2015-12-05:05})} This follows from \eqref{eq:2015-07-28:05}  and from Theorem 
\ref{teo:2015-05-01:12}.
  \end{proof}

  \begin{Remark}
    \label{rem:2015-08-04:02}
  The computations  involved in the proof of Corollary \ref{cor:2015-07-28:06}{(\ref{2015-12-05:04})} require only that $A\colon  \mathcal{D}(A)\subset X  \rightarrow X$ is a linear operator and $\lambda,\mu\in \rho_0(A)$.
  \end{Remark}

\begin{Proposition}
  \label{prop:2015-05-01:14}
  The family of operators
$
\left\{\lambda^n R(\lambda,A)^n\colon X\rightarrow X\right\}_{\lambda>0,n\in \mathbb{N}}
$
is sequentially equicontinuous.
\end{Proposition}
\begin{proof}
Arguing as in the proof of \cite[p.\ 241, Theorem 2]{Yosida80} (\footnote{Also here, we remark that the sequential completeness of the space is not necessary, once that Assumption \ref{ass:int} is standing.}),
we obtain
the inequality 
$$
\sup_{\substack{n\in \mathbb{N}\\\lambda>0}}p\left(
\lambda^{n+1} R(\lambda,A)^{n+1}x\right) \leq 
\sup_{t\in\mathbb{R}^+}p(T_tx), \qquad\ \forall  p\in \mathcal{P}_X,
$$
which provides the  sequential equicontinuity
due to sequential equicontinuity of $T$.
\end{proof}

\begin{Proposition}
  \label{lem:2015-05-02:07} 
  Let $\lambda_1,\ldots,\lambda_j$ be strictly positive real numbers. Then  
$$
p\left(\left(\prod_{i=1}^j \lambda_i R(\lambda_i,A)\right)x\right) \leq  \sup_{t\in\mathbb{R}^+}p\left(T_tx\right),\qquad \ \forall   p\in \mathcal{P}_X,\ \ \forall  x \in X.
$$
\end{Proposition}
\begin{proof}
We can assume $j=1$, since the general case follows immediately by recursion
and by Proposition~\ref{prop:2015-05-01:07}.
For $\lambda>0$ and $x\in X$, by Theorem \ref{teo:2015-05-01:12}
we have
\begin{equation*}
    p\left(
 R(\lambda,A)x
\right)
\leq
\left(\int_0^{+\infty}
 e^{-\lambda t}
\,dt\right)
\sup_{t\in\mathbb{R}^+}
p\left(T_t x\right)
=
\lambda^{-1}
\sup_{t\in\mathbb{R}^+}p\left(T_t x\right).
\end{equation*}
\end{proof}

\subsection{Generation of $C_0$-sequentially equicontinuous semigroups}\label{sub4}

The aim of this subsection is to state a generation theorem for $C_0$-sequentially equicontinuous semigroups in the spirit of the Hille-Yosida theorem stated for $C_0$-semigroups in Banach spaces.
In order to implement the classical arguments (with slight variations due to our ``sequential continuity'' setting), and, more precisely, in order to define the Yosida approximation, we need the sequential completeness of the space $X$. 

\begin{Proposition}
  \label{prop:2015-05-01:15}
Let $X$ be sequentially complete and  let $B\in \mathcal{L}_0(X)$. Assume that  the family
 $
 \left\{ B^n\colon X\rightarrow X\right\}_{n\in \mathbb{N}}
 $
is sequentially equicontinuous.
Let $f\colon\R\rightarrow \mathbb{R}$ be an analytic function 
of the form
$
f(t)=\sum_{n=0}^{+\infty} a_n t^n$, with $t\in \R.$
Then the following hold.
\begin{enumerate}[(i)]
\item\label{2015-12-05:06}
The series
\begin{equation}
  \label{eq:2015-05-03:02}
  f_B(t) \coloneqq  \sum_{n=0}^{+\infty} a_n t^nB^n
\end{equation}
converges in $\mathcal{L}_{0,b}(X)$ uniformly for $t$ on compact sets of $\R$. 
 \item\label{2015-12-05:07}  The function $f_B\colon \mathbb{R}\rightarrow \mathcal{L}_{0,b}(X), \ t\mapsto f_B(t)$ is continuous.
  \item\label{2015-12-05:08}
The family  $\{f_B(t)\}_{t\in [-r,r]}$ is sequentially equicontinuous for every $r>0$.
 \end{enumerate}
\end{Proposition}
\begin{proof}
\emph{(\ref{2015-12-05:06})} For $0\leq n\leq m$, $p\in \mathcal{P}_X$, $D\subset X$ bounded, $r>0$,
$x\in D$, $t\in[-r,r]$,
we write
\begin{equation}
  \label{eq:2015-07-27:00}
p\left(\sum_{k=n}^m a_k t^kB^kx\right)\leq
\sum_{k=n}^m |a_k||t|^kp\left(B^kx\right)\leq
\left(
\sum_{k=n}^{+\infty} |a_k|r^k
\right)\sup_{i\in \mathbb{N}}p\left(B^ix\right)
\leq
\left(
\sum_{k=n}^{+\infty} |a_k|r^k
\right)
\sup_{
y\in\bigcup_{i\in \mathbb{N}}B^iD}
p(y).
\end{equation}
Observe that, by Proposition \ref{prop:2015-05-03:00}\emph{(\ref{2015-10-06:04})}, the supremum appearing in last term of \eqref{eq:2015-07-27:00}  is finite. Then
\begin{equation}
\label{eq:2015-08-04:01}
\sup_{t\in [-r,r]}\rho_{p,D}\left(\sum_{k=n}^m a_k t^kB^k\right)\leq
\left(
\sum_{k=n}^{+\infty} |a_k|r^k
\right)
\sup_{
y\in\bigcup_{i\geq 0}B^iD}
p(y)\qquad  \forall  n\in \mathbb{N}
\end{equation}
shows that the sequence of the partials sums of \eqref{eq:2015-05-03:02} is Cauchy in $\mathcal{L}_{0,b}(X)$, uniformly for $t\in[-r,r]$, and then, by
Proposition \ref{prop:2015-08-04:00}\emph{(\ref{2015-10-08:03})}, the sum is convergent, uniformly for $t\in[-r,r]$.

\emph{(\ref{2015-12-05:07})} This follows from  convergence of the partial sums in the space $C([-r,r],\mathcal{L}_{0,b}(X))$ endowed with the compact-open topology, as shown above.

\emph{(\ref{2015-12-05:08})}
By continuity of $p$, estimate
\eqref{eq:2015-07-27:00}
 shows that
$$
\sup_{t\in [-r,r]}p\left(f_B(t)x\right)=
\sup_{t\in [-r,r]}\lim_{n\rightarrow +\infty}p\left(\sum_{k=0}^n a_k t^k B^kx\right)
\leq\left(\sum_{k=0}^{+\infty}  |a_k| r^k \right)
\sup_{i\in \mathbb{N}}p\left(B^ix\right)\ \ \ \ \ \ \forall  x\in X,
$$
which provides the sequential equicontinuity of $\left\{f_B(t)\right\}_{t\in [-r,r]}$.
\end{proof}


\begin{Lemma} Let $X$ be sequentially complete.
  \label{lem:2015-07-27:02}
 Let $B,C\in \mathcal{L}_0(X)$  be  such that $\{B^n\}_{n\in \mathbb{N}}$ and $\{C^n\}_{n\in \mathbb{N}}$ are sequentially equicontinous.
Let $f(t)=\sum_{n=0}^{+\infty} a_n t^n$, $g(t)=\sum_{n=0}^{+\infty} b_n t^n$ be analytic functions defined on $\R$.
Then
\begin{equation}\label{eq:2015-07-27:02}
  p\left(
f_B(t)g_C(s)x
\right)
\leq
\left(
\sum_{n=0}^{+\infty} |a_n||t|^n
\right)
\left(
\sum_{n=0}^{+\infty} |b_n||s|^n
\right)
\sup_{i,j\in \mathbb{N}}p\left(
B^i C^jx
\right),\qquad \forall   p\in \mathcal{P}_X,\ \forall  x\in X,\ \forall   t,s\in \R,
\end{equation}
and
the family $\{f_B(t)g_C(s)\}_{t,s\in [-r,r]}$ is sequentially equicontinuous for every $r>0$.
\end{Lemma}
\begin{proof}
  By Proposition \ref{prop:2015-05-01:15} and by recalling that every partial sum
$\sum_{i=0}^n a_i t^i B^i$ 
 is sequentially continuous, we can write
$$
p\left(
f_B(t)g_C(s)x
\right)
=\lim_{n\rightarrow +\infty}\lim_{m\rightarrow +\infty}p\left(
\left(\sum_{i=0}^n a_i t^i B^i\right)
\left(\sum_{j=0}^m b_j s^j C^j\right)x
\right)
\qquad \forall   p\in \mathcal{P}_X,\ \forall  x\in X,\ \forall  t,s\in \R.
$$
Then, we obtain 
\eqref{eq:2015-07-27:02}
 by the properties of the seminorms.
The sequential equicontinuity of the family $\{f_B(t)g_C(s)\}_{t,s\in[-r,r]}$ comes from \eqref{eq:2015-07-27:02} and Proposition \ref{prop:2015-05-03:00}\emph{(\ref{2015-12-04:01})}.
\end{proof}

\begin{Proposition}
  \label{prop:2015-05-02:09} Let $X$ be sequentially complete.
 Let $B$, $C$, $f$, $g$,
as in Lemma \ref{lem:2015-07-27:02}.
We have  the following:
\begin{enumerate}[(i)]
\item 
$(f+g)_B=f_B+g_B$ and $(fg)_B=f_B g_B$;
\item \label{2015-12-05:09}
if   $BC=CB$,
 then $f_B(t)g_C(s)=g_C(s)f_B(t)$, for every $t,s\in \R$,
and $\{f_B(t)g_C(s)\}_{t,s\in [-r,r]}$ is sequentially equicontinuous for every $r>0$.
\end{enumerate}

\end{Proposition}
\begin{proof}
The proof follows by algebraic computations on the partial sums and then passing to the limit.
%
\end{proof}

\begin{Notation}
We denote 
 $e^{tB}\coloneqq f_B(t)$ when
$f(t)=e^t$.
\end{Notation}

\begin{Proposition}
  \label{prop:2015-05-03:06}
  Let $X$ be sequentially complete.
  \begin{enumerate}[(i)]
\item\label{2015-10-08:10}
Let $B,C\in \mathcal{L}_0(X)$ be 
such that
 $BC=CB$, and assume that the families $\{B^n\}_{n\in \mathbb{N}}$ and $\{C^n\}_{n\in \mathbb{N}}$ are sequentially equicontinous.
Then, for every $t,s\in \mathbb{R}$,
\begin{enumerate}[(a)]
\item\label{2015-12-05:10}  the sum
$
e^{tB+sC}\coloneqq \sum_{n=0}^{+\infty}\frac{(tB+sC)^n}{n!}
$
converges in $\mathcal{L}_{0,b}(X)$;
\item\label{2015-12-05:11}
$e^{tB+sC}=e^{tB}e^{sC}=e^{sC}e^{tB}$;
\item\label{2015-12-05:12}
 the family $\left\{e^{tB+sC}\right\}_{t,s\in [-r,r]}$ is sequentially equicontinuous for every $r>0$.
 \end{enumerate}
  \item\label{2015-10-08:11}
 Let $B\in \mathcal{L}_0(X)$ be such that
 the family $\{B^n\}_{n\in \mathbb{N}}$ is sequentially equicontinous.
Then $\{e^{tB}\}_{t\in \mathbb{R}^+}$ is a $C_0$-sequentially locally  equicontinuous semigroup on $X$ with infinitesimal generator  $B$.
 \end{enumerate}
\end{Proposition}
\begin{proof}
\emph{(\ref{2015-10-08:10})}
Let $r>0$ and $t\in [-r,r]$. By standard computations, we have 
\begin{equation}
  \label{eq:2015-10-08:04}
  \sum_{i=0}^{n}\frac{(B+C)^i}{i!}t^i=
   \left( \sum_{i=0}^{n}\frac{B^i}{i!}t^i \right) 
   \left( \sum_{i=0}^{n}\frac{C^i}{i!}t^i \right) 
   -  \sum_{i=0}^{n}\frac{B^i}{i!}t^i \left( \sum_{k=n-i+1}^{n}\frac{C^k}{k!}t^k \right).
\end{equation}
Let $D\subset X$  be a bounded set. For $x\in D$ and $p\in \mathcal{P}_X$, we have
\begin{equation*}
  \begin{split}
      p \left( 
\sum_{i=0}^{n}\frac{B^i}{i!}t^i \left( \sum_{k=n-i+1}^{n}\frac{C^kx}{k!}t^k \right)
 \right)&\leq
\sum_{i=0}^{n}\sum_{k=n-i+1}^{n}\frac{1}{i!k!}r^{i+k} p\left( B^iC^kx \right)\\
&\leq
 \left( \sum_{i=0}^{n}\sum_{k=n-i+1}^{n}\frac{1}{i!k!}r^{i+k} \right) 
\sup_{i,k\in \mathbb{N}}\rho_{p,D}
\left( B^iC^k \right).
\end{split}
\end{equation*}
By Proposition \ref{prop:2015-05-03:00}\emph{(\ref{2015-12-04:01})}, the family $\{B^iC^k\}_{i,k\in \mathbb{N}}$ is sequentially equicontinuous. Hence,  by Proposition \ref{prop:2015-05-03:00}\emph{(\ref{2015-10-06:04})}, we have
$
\sup_{i,k\in \mathbb{N}}\rho_{p,D}
\left( B^iC^k \right)<+\infty
$.
Moreover,  Lebesgue's dominated convergence theorem applied in discrete spaces yields
$$
\lim_{n\rightarrow +\infty}  \sum_{i=0}^{n}\sum_{k=n-i+1}^{n}\frac{1}{i!k!}r^{i+k}  =0.
$$
So, we conclude
\begin{equation}
  \label{eq:2015-10-08:07}
  \lim_{n\rightarrow +\infty}\rho_{p,D}
 \left( 
\sum_{i=0}^{n}\frac{B^i}{i!}t^i \left( \sum_{k=n-i+1}^{n}\frac{C^k}{k!}t^k \right)
 \right) =0.
\end{equation}
On the other hand, by Proposition \ref{2015-10-08:05},
\begin{equation}
  \label{2015-10-08:06}
\lim_{n\rightarrow +\infty}     \left( \sum_{i=0}^{n}\frac{B^i}{i!}t^i \right) 
   \left( \sum_{i=0}^{n}\frac{C^i}{i!}t^i \right) 
=
\lim_{n\rightarrow +\infty}    \left( \sum_{i=0}^{n}\frac{B^i}{i!}t^i \right) 
\lim_{n\rightarrow +\infty}     \left( \sum_{i=0}^{n}\frac{C^i}{i!}t^i \right) =e^{tB}e^{tC},
\end{equation}
where the limits are taken in the space  $\mathcal{L}_{0,b}(X)$.
By \eqref{eq:2015-10-08:04}, \eqref{eq:2015-10-08:07} and 
\eqref{2015-10-08:06}, we obtain
\begin{equation}
  \label{eq:2015-10-08:08}
\lim_{n\rightarrow +\infty}  \sum_{i=0}^{n}\frac{(B+C)^i}{i!}t^i=e^{tB}e^{tC},
\end{equation}
with the limit taken in $\mathcal{L}_{0,b}(X)$.

Now, let
 $t\neq 0$ and $|s|\leq|t|$(\footnote{If $|t|< |s|$, we can exchange the role of $B$ and $C$, by simmetry of the sums appearing in \eqref{2015-10-08:09}.}). Then $\left\{ \left( \frac{s}{t}C \right) ^n\right\}_{n\in \mathbb{N}}$ is sequentially equicontinuous. By replacing $C$ by $\frac{s}{t}C$ in 
\eqref{eq:2015-10-08:08}, we have 
\begin{equation}
\label{2015-10-08:09}
\lim_{n\rightarrow +\infty}  \sum_{i=0}^{n}\frac{( tB+sC)^i}{i!}=
%
e^{tB}e^{ \left( \frac{s}{t}C \right) t}=e^{tB}e^{sC},
\end{equation}
where the limits are in $\mathcal{L}_{0,b}(X)$. So we have proved \emph{(a)}.
Properties \emph{(b)} and 
 \emph{(c)} now follow from
\eqref{2015-10-08:09} and from  Proposition \ref{prop:2015-05-02:09}\emph{(\ref{2015-12-05:09})}. 

\emph{(\ref{2015-10-08:11})}
  First we notice that $e^{0B}=I$ by definition. The semigroup property for $\{e^{tB}\}_{t\in\R^+}$ is given by \emph{(\ref{2015-10-08:10})}, which also provides 
 the sequential local equicontinuity.
Proposition \ref{prop:2015-05-01:15} provides the continuity of
the map $\mathbb{R}^+\rightarrow X,\ t\mapsto e^{tB}x$, for every $x\in X$.
Hence, we have proved that $\{e^{tB}\}_{t\in \R^+}$ is a  $C_0$-sequentially locally equicontinuous semigroup.
It remains to show that the infinitesimal generator is $B$.
For $h>0$,
define $f(t;h)\coloneqq e^{h t}-1-ht$.
By applying
\eqref{eq:2015-07-27:02}
to the map $\mathbb{R}\rightarrow \mathbb{R},\  t\mapsto f(t;h)$, with $B$ in place of $B$, and with  $C=I$ and $g\equiv 1$, we obtain
\begin{equation*}
    p\left(\frac{e^{hB}-I}{h}x-Bx\right)=h^{-1}p\left(f_{B}(1;h)\right)\leq
h^{-1}f(1;h)\sup_{n\in \mathbb{N}}p\left(B^nx\right)
\end{equation*}
and the last term converges to $0$ as $h\rightarrow 0^+$, because of sequential equicontinuity of $\{B^n\}_{n\in\N}$. This shows that the domain of the generator is the whole space $X$ and that the generator is $B$.
\end{proof}

We can now state the equivalent of the Hille-Yosida generation theorem in our framework of $C_0$-sequentially equicontinuous semigroups.

\begin{Theorem}
  \label{theo:2015-05-02:10}
  Let $\hat{A}\colon  \mathcal{D}(\hat{A})\subset X \rightarrow X$ be a linear operator. Consider the following two statements.
\begin{enumerate}[(i)]
\item \label{2015-12-05:13}
 $\hat A$ is the infinitesimal generator of a $C_0$-sequentially equicontinuous  semigroup $\hat T$ on $X$.
\item\label{2015-12-05:14}
 $\hat A$ is a sequentially closed  linear operator, $\mathcal{D}(\hat A)$ is sequentially dense in $X$, and there exists a sequence  $\{\lambda_n\}_{n\in \mathbb{N}}\subset \rho_0(\hat A)$, with $\lambda_n\rightarrow +\infty$, such that the family  $\left\{\left(\lambda_n R(\lambda_n,\hat A)\right)^m\right\}_{n,m\in \mathbb{N}}$ is sequentially equicontinuous.
\end{enumerate}
Then {(\ref{2015-12-05:13})}$\Rightarrow${(\ref{2015-12-05:14})}. If $X$ is sequentially complete, then  {(\ref{2015-12-05:14})}$\Rightarrow${(\ref{2015-12-05:13})}.
\end{Theorem}
\begin{proof}
\emph{{(\ref{2015-12-05:13})}$\Rightarrow${(\ref{2015-12-05:14})}}.
The fact that $\hat A$ is a sequentially closed linear operator was proved in   
Corollary \ref{cor:2015-05-01:13}. The fact that $\mathcal{D}(\hat A)$ is sequentially dense in $X$ was proved in  Proposition \ref{prop:2015-05-01:10}. The remaining facts follow by Proposition \ref{lem:2015-05-02:07} and 
Theorem \ref{teo:2015-05-01:12}.

\emph{{(\ref{2015-12-05:14})}$\Rightarrow${(\ref{2015-12-05:13})}}. We split this part of the proof in several steps.

\underline{\emph{Step 1}.}
Let $\{\lambda_n\}_{n\in \mathbb{N}}\subset \rho_0(\hat A)$ be a sequence as in \emph{(\ref{2015-12-05:14})}. 
For $n\in \mathbb{N}$, define $J_{\lambda_n}\coloneqq \lambda_n R(\lambda_n,\hat A)$. 
Observe that, for all $x\in \mathcal{D}(\hat A)$, it is $(J_{\lambda_n}-I)x=R(\lambda_n,\hat A)\hat Ax$. By
assumption, the family $\{J_{\lambda_n}\}_{n\in \mathbb{N}}$ is sequentially equicontinuous, and then, for every
$x\in \mathcal{D}(\hat A)$ and $p\in \mathcal{P}_X$,
\begin{equation}
  \label{eq:2015-12-05:15}
  \lim_{n\rightarrow +\infty }p\left(
J_{\lambda_n} x-x
\right)=
\lim_{n\rightarrow +\infty}
p\left(
 R(\lambda_n,\hat A)\hat Ax
 \right)\leq \lim_{n\rightarrow +\infty}\lambda_{n}^{-1}\left(\sup_{k\in \mathbb{N}}p
 \left(
 J_k\hat Ax
 \right)\right)=0.
\end{equation}
Now let $x\in X$. By assumption, there exists a sequence $\{x_k\}_{k\in \mathbb{N}}$ in $\mathcal{D}(\hat A)$ converging to $x$ in $X$. 
We have
\begin{equation*}
  \begin{split}
p\left(
J_{\lambda_n} x-x
\right)
&\leq
p\left(
x-x_k
\right)
+p\left(
J_{\lambda_n}x_k-x_k
\right)+
p\left(
J_{\lambda_n}(x-x_k)
\right),
\qquad\forall  k\in \mathbb{N},\ \forall  n\in \mathbb{N},\ \forall  p\in \mathcal{P}_X.
\end{split}
\end{equation*}
By taking first the $\limsup$ in $n$ and then the limit as $k\rightarrow+\infty$ in the inequality above, 
and recalling  \eqref{eq:2015-12-05:15} and the sequential equicontinuity of $\{J_{\lambda_n}\}_{n\in \mathbb{N}}$, we conclude
  \begin{equation}
    \label{eq:2015-07-28:07}
      \lim_{n\rightarrow +\infty}J_{\lambda_n} x=x,\qquad \ \forall   x\in X.
    \end{equation}

\underline{\emph{Step 2}.}
Here we show that, for $t \in\mathbb{R}^+$ and  $n\in \mathbb{N}$,
 $T^{(n)}_t\coloneqq e^{t\hat AJ_{\lambda_n}}$ is well-defined as a convergent series in $\mathcal{L}_{0,b}(X)$, and that $\big\{T^{(n)}_t\big\}_{t\in \mathbb{R}^+,n\in \mathbb{N}}$ is sequentially equicontinuous. 
Taking into account that  $\hat AJ_{\lambda_n}=\lambda_n(J_{\lambda_n}-I)$,
we have (as formal sums)
$T^{(n)}_t= e^{t\hat AJ_{\lambda_n}}=e^{t\lambda_n(J_{\lambda_n}-I)}$.
Since
 $\{J^k_{\lambda_n}\}_{k\in \mathbb{N}}$ 
is assumed to be
  sequentially equicontinuous, 
by  Proposition \ref{prop:2015-05-03:06}\emph{(\ref{2015-10-08:10})}, $T^{(n)}_t$ is well-defined 
as a convergent series in $\mathcal{L}_{0,b}(X)$,
and
\begin{equation}\label{xc}
T^{(n)}_t= 
e^{-t\lambda_n I}e^{t\lambda_n J_{\lambda_n}}.
\end{equation}
  Hence, using   Proposition \ref{prop:2015-05-03:06}\emph{(\ref{2015-10-08:11})}, 
the family
$\{T^{(n)}_t \}_{t\in\R^+}$ is a $C_0$-sequentially locally  equicontinuous semigroup for each fixed $n\in\N$. 
 On the other hand, 
 by \eqref{xc} and by Lemma \ref{lem:2015-07-27:02}, we have
$$
\sup_{n\in\N}p\left(T^{(n)}_tx\right)=\sup_{n\in \N} \left(e^{-t\lambda_n}p\left(e^{t\lambda_n J_{\lambda_n}}x\right)\right)\leq \sup_{n,k\in \N}p(J_{\lambda_{n}}^kx),\qquad  \forall  p\in \mathcal{P}_X,\ \forall  x\in X.
$$
As, by assumption,
 $\{J^k_{\lambda_n}\}_{n,k\in \mathbb{N}}$ 
is
  sequentially equicontinuous, this shows that
 $\big\{T^{(n)}_t\big\}_{t\in \mathbb{R}^+,n\in \mathbb{N}}$ is sequentially equicontinuous.

\underline{\emph{Step 3}.} Here we show that   the sequence $\big\{T_t^{(n)}x\big\}_{n\in\N}$ is Cauchy for every  $t\in \mathbb{R}^+$ and $x\in \mathcal{D}(\hat A)$.
First note that, since the family $\{R(\lambda_n,\hat A)\}_{n\in \mathbb{N}}$ is a commutative set (see \eqref{eq:2015-05-03:10} and Remark \ref{rem:2015-08-04:02}), also the family $\{J_{\lambda_n}\}_{n\in \mathbb{N}}$ is a commutative set. Then
  $\lambda_m(J_{\lambda_m}-I)$ commutes with every $J_{\lambda_n}$. Since the sum defining $T^{(m)}_t$ is convergent in $\mathcal{L}_{0,b}(X)$,
 we have $T^{(m)}_tJ_{\lambda_n}=J_{\lambda_n}T_t^{(m)}$ and $T^{(m)}_tT^{(n)}_s=T^{(n)}_sT^{(m)}_t$ for every $m,n\in \mathbb{N}$, $t,s\in \mathbb{R}^+$.
By Lemma \ref{prop:2015-05-03:12} and by the commutativity just noticed,
if $x\in X$ and $t\in\R^+$, the map
 $F\colon [0,t]\rightarrow X,\ s\mapsto T^{(n)}_{t-s}T^{(m)}_sx$, is differentiable and
$$
T^{(m)}_tx-T^{(n)}_tx
=
\int_0^t F'(s)ds
=\int_0^t T^{(n)}_{t-s}T^{(m)}_s\hat A\left(J_{\lambda_m}-J_{\lambda_n}\right)x ds,
$$
where the integral is well-defined by sequential completeness of $X$.
We notice that $J_{\lambda_n}\hat A=\hat AJ_{\lambda_n}$ on $\mathcal{D}(\hat A)$.
Then, from the equality above we deduce
\begin{equation*}
p\left(T^{(m)}_tx-T^{(n)}_tx\right)\leq
\int_0^t p\left(T^{(n)}_{t-s}T^{(m)}_s\left(J_{\lambda_m}-J_{\lambda_n}\right)\hat Ax\right) ds,\qquad  \forall   x\in \mathcal{D}(\hat A),\ \forall  p\in \mathcal{P}_X,
\end{equation*}
and then
\begin{equation}\label{eq:2015-07-29:00}
  \sup_{t\in[0,\hat t]}p\left(T^{(m)}_tx-T^{(n)}_tx\right)\leq
\hat t\sup_{t,s\in[0,\hat t]}p\left(
T^{(n)}_{t}T^{(m)}_s
\left(J_{\lambda_m}-J_{\lambda_n}\right)\hat Ax
\right),\qquad  \forall   \hat t>0,\   \forall  x\in \mathcal{D}(\hat A),\  \forall  p\in \mathcal{P}_X.
\end{equation}
Now observe that,
by Proposition \ref{prop:2015-05-03:00}\emph{(\ref{2015-12-04:01})} and Step 2, the family
$
\big\{T^{(m)}_tT^{(n)}_s\big\}_{
\substack{
 t,s\in \mathbb{R}^+ \\m,n\in \mathbb{N}}}
$
is sequentially equicontinuous, and then the term on the right-hand side of  \eqref{eq:2015-07-29:00} goes to $0$ as $n,m\rightarrow +\infty$, because of \eqref{eq:2015-07-28:07}.
Hence,  the sequence $\{T_t^{(n)}x\}_{n\in\N}$ is Cauchy for every $t\in \mathbb{R}$ and $x\in\mathcal{D}(\hat A)$.

\underline{\emph{Step 4}.}
By Step 3 and by sequential completeness of $X$, we conclude that there exists  in $X$
\begin{equation}\label{limlim}
\hat T_tx\coloneqq \lim_{n\rightarrow+\infty}T^{(n)}_tx, \qquad\forall  t\in\R^+, \ \forall  x\in \mathcal{D}(\hat A).
\end{equation} 
 Moreover, by \eqref{eq:2015-07-29:00}, the limit \eqref{limlim}  is uniform in $t\in [0,\hat t]$, for every $\hat t>0$.

\underline{\emph{Step 5}.}
We  extend the result of Step 4, stated for $x\in\mathcal{D}(\hat A)$, to all $x\in X$.
Let $\hat t>0$ and let $\{x_k\}_{k\in \mathbb{N}}\subset D(\hat A)$ be a sequence converging to $x$ in $X$. 
We can write
\begin{equation*}
    T^{(m)}_tx-T^{(n)}_tx=
     \left( T^{(m)}_t-T^{(n)}_t \right) (x-x_k)+
     \left( T^{(m)}_t-T^{(n)}_t \right) x_k,\qquad\forall  t\in[0,\hat t],\ \forall  m,n,k\in \mathbb{N}.
\end{equation*}
Then, using Step 4, we have, uniformly for $t\in[0,\hat t]$,
\begin{equation*}
  \begin{split}
    \limsup_{n,m\rightarrow +\infty} \sup_{t\in[0,\hat t]}   p \left( T^{(m)}_tx-T^{(n)}_tx \right) 
    & \leq
    \limsup_{n,m\rightarrow +\infty}\sup_{t\in[0,\hat t]}p    \left(   \left( T^{(m)}_t-T^{(n)}_t \right) (x-x_k) \right)\\
    & \leq
    \sup_{\substack{n,m\in \mathbb{N}\\t\in[0,\hat t]}}p    \left(   \left( T^{(m)}_t-T^{(n)}_t \right) (x-x_k) \right),
    \qquad \forall  k\in \mathbb{N}, \ \forall  p\in \mathcal{P}_X.
  \end{split}
\end{equation*}
The last term goes to $0$ as $k\rightarrow +\infty$, because of sequential equicontinuity of the family $\big\{T^{(n)}_t\big\}_{n\in \mathbb{N},t\in\mathbb{R}^+}$ (Step 2).

Hence, recalling that  $\mathcal{D}(\hat A)$ is sequentially dense in $X$,  we have proved that that there exists in $X$, uniformly for $t\in [0,\hat t]$,
\begin{equation}\label{limlim2}
\hat T_tx\coloneqq \lim_{n\rightarrow+\infty}T^{(n)}_tx, \ \ \forall  x\in X.
\end{equation}

 \underline{\emph{Step 6}.}
 We  show that the family $\hat T=\{\hat T_t\}_{t\in \R^+}$ is a $C_0$-sequentially equicontinuous semigroup on $X$.  First we notice that, as by Step 5 the  limit in \eqref{limlim2} defining $\hat T_tx$ is  uniform for $t\in[0,\hat t]$, for every $\hat t>0$,
 then the function $\mathbb{R}^+\rightarrow X,\ t\mapsto \hat T_tx$, is continuous. In particular, $\hat T_tx \rightarrow \hat T_0x$ as $t\rightarrow 0^+$ for every $x\in X$. Moreover, $\hat T_0 =I$ as $T^{(n)}_0=I$ for each $n\in \N$. The linearity of $\hat T_t$ and the semigroup property come from the same properties holding for every $T^{(n)}_t$. It remains to show that the family $\hat T$ is sequentially equicontinuous. This comes from
sequential equicontinuity of the family $\big\{T^{(n)}_t\big\}_{n\in \mathbb{N},t\in\mathbb{R}^+}$ (Step 2), and from
 the estimate 
$$
p\left(\hat T_tx\right)\leq 
p\left(\hat T_tx-T_t^{(n)}x\right)+
p\left(T^{(n)}_tx\right)
\leq
p\left(\hat T_tx-T^{(n)}_tx\right)+
\sup_{\substack{t\in \mathbb{R}^+\\ n\in \mathbb{N}}}p\left(T^{(n)}_tx\right)\qquad \forall  t\in \mathbb{R}^+,\ \forall  n\in \mathbb{N},
$$
 by taking first the limit as $n\rightarrow +\infty$ and then the supremum over $t$.
 
 \underline{\emph{Step 7}.}
To conclude the proof, we only need to show that the infinitesimal generator of $\hat T$ is $\hat A$. 
Let $p\in \mathcal{P}_X$ and $x \in \mathcal{D}(\hat A)$.
By applying 
Proposition \ref{prop:2015-05-01:11}
to $T^{(n)}$, we can write
$$ \hat T_tx-x =\lim_{n\rightarrow +\infty} (T_t^{(n)}x-x) =\lim_{n\rightarrow +\infty}\int_0^t T_s^{(n)}\hat AJ_{\lambda_n}xds,
$$
where the integral on the right-hand side exists because of sequential completeness of $X$ and of continuity of the integrand function, and where the latter equality is obtained, as usual, by pairing the two members of the equality with funtionals $\Lambda\in X^*$ and by using
\eqref{eq:2015-05-01:08}.

Now we wish to exchange the limit with the integral. This is possible, as, by Step 2, Step 5, and 
 and \eqref{eq:2015-07-28:07}, we have 
 $$\lim_{n\rightarrow+\infty}T^{(n)}_tJ_{\lambda_n}\hat Ax=\hat T_t\hat Ax
\ \ \ \mbox{uniformly for $t$ over compact sets}.$$
Then 
$$ \hat T_tx-x  =\int_0^t \lim_{n\rightarrow +\infty}  T_s^{(n)}\hat AJ_{\lambda_n}xds = \int_0^t \hat T_s\hat Ax.
$$
Dividing by $t$ and letting $t\rightarrow 0^+$, we conclude that $x\in \mathcal{D}(\tilde{A})$, where $\tilde{A}$ is the infinitesimal generator of $\hat T$, and that $\tilde{A}=\hat A$ on $\mathcal{D}(\hat A)$.
 But, by assumption, for some $\lambda_n>0$, the operator   $\lambda_n-\hat A$ is one-to-one and full-range. By  Theorem \ref{teo:2015-05-01:12}, the same thing holds true for $\lambda_n-\tilde A $. Then we conclude $\mathcal{D}(\tilde A)=\mathcal{D}(\hat {A})$ and $\tilde A=\hat A$.
\end{proof}
%

  \begin{Remark}\label{2015-12-05:23}
  Let $X$ be a Banach space with  norm $|\cdot|_X$ and let $\tau$ be a sequentially complete locally convex topology on $X$
such that the $\tau$-bounded sets are exactly the $|\cdot|_X$-bounded sets.
Then, by Proposition \ref{2015-07-30:02}(\ref{2015-08-31:01}), we have $\mathcal{L}_0((X,\tau))\subset L((X,|\cdot|_X))$. Let $\hat T$ be a $C_0$-sequentially equicontinuous semigroup on $(X,\tau)$ with infinitesimal generator $\hat A$. By
    referring to the notation of the proof of 
Theorem \ref{theo:2015-05-02:10}, we make the following observations.
\begin{enumerate}[(1)]
\item 
Since
$R(\lambda_n,\hat A)\in \mathcal{L}_0((X,\tau))\subset L((X,|\cdot|_X))$, 
 then the Yosida approximations
 $\{T^{(n)}\}_{n\in \mathbb{N}}$, approximating
$\hat T$ according to
 \eqref{limlim2},
 are uniformly continuous semigroups on the Banach space $(X,|\cdot|_X)$.
\item The fact that $\big\{(\lambda_n R(\lambda_n,\hat A))^m\big\}_{n,m\in \mathbb{N}}$ is sequentially equicontinuous implies that such a family is uniformly bounded in $L((X,|\cdot|_X))$. 
Indeed, as the unit ball $B$ in $(X,|\cdot|_X)$ is bounded in $(X,\tau)$,  by Proposition \ref{prop:2015-05-03:00}{(\ref{2015-10-06:04})} the set
$\big\{(\lambda_n R(\lambda_n,\hat A))^mx\big\}_{n,m\in \mathbb{N},\,x\in B}$ is bounded in $(X,\tau)$. Hence, it is also bounded  in $(X,|\cdot|_X)$, as we are assuming that the bounded sets are the same in both the topologies.
As a consequence, by recalling the Hille-Yosida theorem for $C_0$-semigroups in Banach spaces, we have that $\hat T$ is also a $C_0$-semigroup in the Banach space $(X,|\cdot|_X)$ if and only if $D(\hat A)$ is norm dense in $X$.
\end{enumerate}
\end{Remark}

\subsection{Relationship with   \emph{bi-continuous} semigroups}\label{sub:K}
In this subsection we establish a comparison of our notion of $C_0$-sequentially equicontinuous semigroup with the notion of \emph{bi-continuous} semigroup developed in \cite{ Kuehnemund2001, Kuhmenund01}.  The latter requires to deal with Banach spaces as underlying spaces.

\begin{Definition}
Let $(X,|\cdot|_X)$ be a Banach space and  let $X^*$ be its  topological dual. A linear subspace $\Gamma\subset X^*$ is called  \emph{norming} for $(X,|\cdot|_X)$ if
 $
|x|_X=\sup_{\gamma\in \Gamma,\   |\gamma|_{X^*}\leq 1}|\gamma(x)|
$, for every $x\in X$.
\end{Definition}

\begin{Lemma}
\label{lem:2015-05-04:01}
 Let $(X,|\cdot|_X)$ be a Banach space and let  $\Gamma\subset X^*$ be  norming for $(X,|\cdot|_X)$ and closed with respect to the operator norm $|\cdot|_{X^*}$.
 Then $B\subset X$ is $\sigma(X,\Gamma)$-bounded if and only if it is  $|\cdot|_X$-bounded.
\end{Lemma}
\begin{proof}
As  $\sigma(X,\Gamma)$ is weaker than the $|\cdot|_X$-topology,  clearly $|\cdot|_X$-bounded sets are also  $\sigma(X,\Gamma)$-bounded.
Conversely, let $B\subset X$ be $\sigma(X,\Gamma)$-bounded and   consider the family of continuous functionals 
$$
\{\Lambda_b\colon \Gamma\rightarrow \mathbb{R}, \ \gamma\mapsto \gamma(b)\}_{b\in B}, 
$$
 By assumption, $\sup_{b\in B} |\gamma (b)|<+\infty$ for every $\gamma\in \Gamma$.
The Banach-Steinhaus theorem applied in the Banach space $({\Gamma},|\cdot|_{X^*})$ yields
$$
M\coloneqq \sup_{b\in B}\sup_{\gamma\in  \Gamma,\ | \gamma |_{X^*}\leq 1}|\gamma(b)|  <+\infty.
$$
Then, since $\Gamma$ is norming  for $(X,|\cdot|_X)$, we have
$$
| b|_X=\sup_{\gamma\in \Gamma,\ | \gamma |_{X^*}\leq 1}|\gamma(b)| \leq M <+\infty\qquad \forall   b\in B,
$$
and then $B$ is $|\cdot|_{X}$-bounded.
\end{proof}


We recall the definition of bi-continuous semigroup as given in
\cite[Def.\ 3]{Kuhmenund01} and \cite[Def.\ 1.3]{Kuehnemund2001}.

\begin{Definition}\label{def:bi}
 Let $(X,|\cdot|_X)$ be a Banach space with topological dual $X^*$. Let $\tau$ be a Hausdorff locally convex topology on $X$ with the following properties.
\newcounter{saveenum}
 \begin{enumerate}[(i)]
 \item\label{2015-12-05:19}
The space $(X,\tau)$ is sequentially complete on $|\cdot|_X$-bounded sets.
 \item
 $\tau$ is weaker than the topology induced by the norm $|\cdot|_X$.
 \item
 The topological dual of $(X,\tau)$ is norming for $(X,|\cdot|_X)$.
\setcounter{saveenum}{\value{enumi}}
 \end{enumerate}
 A family of linear operators $T=\{T_t\colon X\rightarrow X\}_{t\in \R^+}\subset L((X,|\cdot|_X))$ is called a  \emph{bi-continuous semigroup with respect to $\tau$ and of type $\alpha \in \mathbb{R}$}   if the following conditions hold:
 \begin{enumerate}[(i)]
\setcounter{enumi}{\value{saveenum}}
 \item
 $T_0=I$
  and  $T_tT_s=T_{t+s} $ for every $t,s\in \R^+$;
 \item\label{2015-12-05:21}
for some $M\geq 0$, $|T_t|_{L((X,|\cdot|))}
 \leq Me ^{\alpha  t}$,  for every $t\in \R^+$;
 \item\label{2015-12-05:22}
  $T$ is strongly $\tau$-continuous, that is the map 
  $\R^+\rightarrow (X,\tau), 
  \ t\mapsto T_tx
  $
 is continuous for every $x\in X$;
 \item\label{2015-10-09:00}
 $T$ is locally bi-continuous, that is, for every
 $|\cdot|_X$-bounded 
 sequence $\{x_n\}_{n\in \N}\subset X$  $\tau$-convergent to $x\in X$ and every $\hat t>0$, we have
 $$
\lim_{n\rightarrow+\infty} T_t x_n =T_tx \qquad \mbox{in  $(X,\tau)$,
 uniformly in $t\in[0,\hat t]$}.
 $$
 \end{enumerate}
\end{Definition}

The following proposition 
shows that the notion of bi-continuous semigroup  is a specification of our notion of $C_0$-sequentially locally equicontinuous semigroup in sequentially complete spaces. Indeed, given a bi-continuous semigroup on a Banach space  $(X,|\cdot|_X)$ with respect to a topology $\tau$,  one can define a locally convex sequentially complete topology $\tau'\supset \tau$ and  see the bi-continuous semigroup as a $C_0$-sequentially locally equicontinuous semigroup on $(X,\tau')$.

\begin{Proposition}
  \label{prop:2015-05-02:05}
  Let $\{T_t\}_{t\in \mathbb{R}^+}$ be a bi-continuous semigroup on $X$ with respect to $\tau$ and of type $\alpha$. 
  Then there exists a 
locally convex topology $\tau'$ with the following properties:
  \begin{enumerate}[(i)]
  \item \label{2015-12-05:16}
$\tau\subset \tau'$ and $\tau'$ is weaker than the $|\cdot|_X$-topology;
\item \label{2015-12-05:17}
a sequence converges in $\tau'$ if and only if it is $|\cdot|_X$-bounded and convergent in $\tau$;
\item \label{2015-12-05:18}
$(X,\tau')$ is sequentially complete;
\item \label{2015-12-05:20}
$T$  is a $C_0$-sequentially locally equicontinuous semigroup in $(X,\tau')$; moreover, for every $\lambda>\alpha$, $\{e^{-\lambda t}T_t\}_{t\in \mathbb{R}^+}$ is a 
$C_0$-sequentially  equicontinuous semigroup on $ (X,\tau')$ satisfying 
Assumption~\ref{ass:int}.
\end{enumerate}
\end{Proposition}
\begin{proof}   Denote by $X^*$ 
the topological dual of $(X,|\cdot|_X)$, and let $\mathcal{P}_X$ be a set of seminorms on $X$ inducing $\tau$. Denote by $\Gamma$ the dual of $(X,\tau)$.
 On $X$, define the seminorms
  $$
  q_{p,\gamma}(x)=p(x)+|\gamma(x)|,\qquad   p\in \mathcal{P}_X,\ \gamma\in \overline \Gamma,
  $$
  where $\overline \Gamma$ is the closure of $\Gamma$ with respect to the  norm $|\cdot|_{X^*}$.
  Let $\tau'$ be the locally convex topology induced by the family of seminorms
  $\{q_{p,\gamma}\}_{p\in \mathcal{P}_X,\gamma\in \overline \Gamma}$.

\emph{(\ref{2015-12-05:16})}
Clearly $\tau\subset \tau'$ and 
$\tau'$ is weaker than the $|\cdot|_X$-topology.

\emph{(\ref{2015-12-05:17})}
As $\tau\subset \tau'$, the
$\tau'$-convergent sequences are $\tau$-convergent. 
Moreover, as $\Gamma$ is norming, $\overline \Gamma$ is norming too.
Then, by Lemma \ref{lem:2015-05-04:01}, every $\sigma(X,\overline \Gamma)$-bounded set is $|\cdot|_X$-bounded. In particular, every convergent sequence in $\tau'$ is $|\cdot|_X$-bounded.

Conversely, consider a  sequence $\{x_n\}_{n\in \N}\subset X$ which is $\tau$-convergent to $0$ in $X$ and $|\cdot|_X$-bounded by a constant $M>0$. 
To show that $x_n\stackrel{\tau'}{\rightarrow} 0$, we only  need to show that $\gamma(x_n)\rightarrow 0$ for every $\gamma\in \overline{\Gamma}$. For that, notice first that the convergence to $0$ with respect to $\tau$ implies the convergence $\gamma(x_n)\rightarrow 0$ for every $\gamma\in \Gamma$. Take now $\gamma\in \overline{\Gamma}$ and a sequence $\{\gamma_k\}_{k\in\N}\subset \Gamma$ converging to $\gamma$ with respect to $|\cdot|_{X^*}$.
Then the estimate
$$
|\gamma(x_n-x)|
\leq  M
| \gamma-\gamma_k|_{X^*}+|\gamma_k(x_n)| \qquad \forall  n,k\in \N,
$$
yields
$$
\limsup_{n\rightarrow+\infty}
|\gamma(x_n)|
\leq
M |\gamma-\gamma_k|_{X^*}\qquad  \forall  k\in \mathbb{N}.
$$
Since $\gamma_k\rightarrow \gamma$ with respect to $|\cdot|_{X^*}$ when $k\rightarrow +\infty$, we now conclude that sequence $\{x_n\}_{n\in \N}$ converges to $0$ also with respect to $\tau'$.

%

\emph{(\ref{2015-12-05:18})}
A Cauchy sequence $\{x_n\}_{n\in \mathbb{N}}$ in $(X,\tau')$ is $\tau'$-bounded. By Lemma \ref{lem:2015-05-04:01}, it is $|\cdot|_X$-bounded. Clearly, $\{x_n\}_{n\in \mathbb{N}}$ is also $\tau$-Cauchy. Then, 
by Definition \ref{def:bi}\emph{(\ref{2015-12-05:19})}, $\{x_n\}_{n\in \mathbb{N}}$ converges to some $x$ in $(X,\tau)$.
Since the sequence is $|\cdot|_X$-bounded, by \emph{(\ref{2015-12-05:17})} the convergence takes place also in $\tau'$. This proves that $(X,\tau')$ is sequentially complete.

\emph{(\ref{2015-12-05:20})}
We start by proving that $\{T_t\}_{t\in \mathbb{R}^+}$  is a sequentially locally equicontinuous family of operators  in the space  $(X,\tau')$.
Let $\{x_n\}_{n\in\N}$ be a sequence $\tau'$-convergent to $0$.
By \emph{(\ref{2015-12-05:17})},
 $\{x_n\}_{n \in\mathbb{N}}$ is    $|\cdot|_X$-bounded and $\tau$-convergent to $0$.
 By Definition
\ref{def:bi}\emph{(\ref{2015-10-09:00})}
\begin{equation}\label{eqq}
\lim_{n\rightarrow +\infty}\sup_{t\in [0,\hat t]}p(T_tx_n)=0,
\qquad
\forall  p\in \mathcal{P}_X, \ \forall  \hat t>0.
\end{equation}
 Assume now, by contradiction, that there exist $R>0$, $p\in \mathcal{P}_X$, $\gamma\in \overline{\Gamma}$, and $\varepsilon>0$, such that
$$
\limsup_{n\rightarrow +\infty}\sup_{t\in [0,R]}q_{p,\gamma}(T_tx_n)\geq \varepsilon.
$$
Then, due to \eqref{eqq}, there exist a sequence $\{t_n\}
_{n\in\N}\subset [0,R]$ convergent to some $t\in [0,R]$ and a subsequence of $\{x_n\}_{n\in\N}$, still denoted  by $\{x_n\}_{n\in\N}$, such that
\begin{equation}
  \label{eq:2015-10-09:01}
  |\gamma(T_{t_n}x_n)|\geq \varepsilon\qquad  \forall  n\in \mathbb{N}.
\end{equation}
By Definition \ref{def:bi}\emph{(\ref{2015-12-05:21})}, the family $\{T_t\}_{t\in[0,R]}$ is uniformly bounded in the operator norm. Then, by recalling that
$\{x_n\}_{n\in \mathbb{N}}$ is $|\cdot|_X$-bounded, we have
$$
\hat M\coloneqq\sup_{n\in \mathbb{N}}|T_{t_n}x_n|_X<+\infty.
$$
Let $\hat\gamma\in \Gamma$ be such that $|\hat \gamma-\gamma|_{X^*}\leq\epsilon/(2\hat M)$. Then
\begin{equation}
  \label{eq:2015-10-09:02}
  \limsup_{n\rightarrow +\infty}|\gamma(T_{t_n}x_n)|\leq \frac{\epsilon}{2}+\limsup_{n\rightarrow +\infty}|\hat \gamma(T_{t_n}x_n)|=\frac{\epsilon}{2},
\end{equation}
where the last equality is due  to
\eqref{eqq}
and to the fact that $\hat \gamma\in \Gamma=(X,\tau)^*$.
But \eqref{eq:2015-10-09:02} contradicts \eqref{eq:2015-10-09:01}.
The fact that $T$ is strongly continuous with respect to $\tau'$ follows from 
\emph{(\ref{2015-12-05:17})} and from
Definition \ref{def:bi}\emph{(\ref{2015-12-05:21})}-\emph{(\ref{2015-12-05:22})}.

Finally, by Definition \ref{def:bi}\emph{(\ref{2015-12-05:21})} we can apply Proposition \ref{prop:2015-08-01:00}\emph{(\ref{2015-10-13:14})} and conclude that $\{e^{-\lambda t}T_t\}_{t\in \mathbb{R}^+}$ is a 
$C_0$-sequentially  equicontinuous semigroup on $ (X,\tau')$ for every $\lambda>\alpha$. 
Due to part \emph{(\ref{2015-12-05:18})}, such a semigroup
 satisfies Assumption \ref{ass:int} (recall Remark \ref{2015-10-29:13}).
\end{proof}

\subsection{A note on a weaker definition}\label{2015-10-30:07}

In this subsection we point out how, under weaker
requirements in Definition \ref{def:2015-05-01:01},
 some of the results appearing in the previous sections still hold.
The definition that we are going to introduce below will not be used in the sequel, 
except in Subsection \ref{2015-10-13:03}, where we briefly clarify the relationship between the notion of $\pi$-semigroup, introduced in \cite{Priola99}, and our notion of $C_0$-sequentially locally equicontinuos semigroup.
\begin{Definition}\label{2015-10-13:01}
Let $X$ be a Hausdorff locally convex space. Let $T\coloneqq \{T_t\}_{t\in \mathbb{R}^+}\subset \mathcal{L}_0(X)$ be a family of sequentially continuous linear operators. We say that $T$ is a \emph{bounded $C_0$-sequentially continuous semigroup} if
  \begin{enumerate}[(i)]
  \item  $T_0=I$ and $T_{t+s}=T_tT_s$ for all $t,s\in \mathbb{R}^+$;
  \item\label{2015-10-13:04} for each $x\in X$, the map $\mathbb{R}^+\rightarrow X,\ t \mapsto T_tx$, is continuous and bounded.
  \end{enumerate}
\end{Definition}

By recalling 
Proposition \ref{prop:2015-05-01:02a},
we see that 
Definition \ref{def:2015-05-01:01} is stronger than
Definition \ref{2015-10-13:01}.

Let  $T$ be a bounded $C_0$-sequentially continuous semigroup on $X$ and let us assume that, for every $x\in X$, the Riemann integral
\begin{equation}\label{Lap}
R(\lambda)x\coloneqq \int_0^{+\infty}e^{-\lambda t}T_txdt,
\end{equation}
(which exists in the completion of $X$, by Definition \ref{2015-10-13:01}\emph{(\ref{2015-10-13:04})}) belongs to $X$ (this happens, for example, if $X$ is sequentially complete).

Then, a straightforward inspection of the proofs shows that the following results still hold:
 Proposition \ref{prop:2015-08-01:00}\emph{(\ref{2015-10-13:00})};
 Proposition \ref{prop:2015-05-01:07};
 Proposition \ref{prop:2015-05-01:02}\emph{(\ref{2015-10-13:02})};
 Proposition \ref{prop:2015-05-01:10};
 Proposition \ref{prop:2015-05-01:11};
 Theorem \ref{teo:2015-05-01:12}, except for the conclusion $(\lambda-A)^{-1}\in \mathcal{L}_0(X)$;
 Corollary \ref{cor:2015-07-28:06}.

To summarize, if the Laplace transform \eqref{Lap} of a bounded $C_0$-sequentially continuous semigroup is well-defined, then the domain $\mathcal{D}(A)$ of the generator $A$ is sequentially dense in $X$ and $\lambda-A$ is one-one and onto for every $\lambda>0$.

We outline that, without the sequential local equicontinuity of $T$, the proof of
Lemma \ref{prop:2015-05-03:12} does not work, and consequently the proof of Theorem \ref{theo:2015-07-29:04} does not work.  

\subsection{Examples and counterexamples}\label{sec:examples}
In this subsection we provide some examples to clarify some features of the notion of $C_0$-sequentially (locally) equicontinuous semigroup.

First, with respect to the case of $C_0$-semigroups on Banach spaces, we notice two relevant basic implications that we loose when dealing with strong continuity and (sequential) local equicontinuity in locally convex spaces.
The first one is related to the growth rate of the orbits of the semigroup, and consequently to the possibility to define the Laplace transform.
The fact that  $T$ is a $C_0$-locally  (sequentially) equicontinuous semigroup does not imply, in general,  the existence 
  of $\alpha>0$ such that $\{e^{-\alpha t}T_t\}_{t\in \R^+}$ is a $C_0$-(sequentially) locally  equicontinuous semigroup.
We give two examples.

\begin{Example}\label{2015-10-13:05}
Consider the vector space $X\coloneqq C(\mathbb{R})$, endowed with the topology of the uniform convergence on compact sets, which makes $X$ a Fr\'echet space.
 Define   $T_t\colon X\rightarrow X$ by
$$
T_t\varphi(s)\coloneqq e^{st}\varphi(s)\qquad \forall  s\in \mathbb{R},\  \forall  t\in \mathbb{R}^+,\ \forall\varphi\in X.
$$
One verifies that  $T=\{T_t\}_{t\in\R^+}$ is a $C_0$-sequentially locally  equicontinuous semigroup on $X$ (actually, locally equicontinuous,  by Proposition \ref{2015-10-06:00}). 
On the other hand, for whatever $\alpha>0$, the family $\{e^{-\alpha t}T_t\}_{t\in\R^+}$ is not  sequentially equicontinuous. 
Indeed,  one has that  $\{e^{-\alpha t}T_tf\}_{t\in \mathbb{R}^+}$ is unbounded in $X$ for every $f$ not identically zero on $(\alpha,+\infty)$.
\end{Example}

\begin{Example}\label{example2}
  Another classical example is given in  \cite{Komura68}.
Let $X$ be as in Example \ref{2015-10-13:05}, with the same topology.
For $t\in \mathbb{R}^+$, we define $T\coloneqq \{T_t\}_{t\in \mathbb{R}^+}$ by
$$
T_t\colon X\rightarrow X,\ \varphi \mapsto \varphi(t+\cdot).
$$
Then
$T$ is a $C_0$-sequentially locally equicontinuous semigroup on $X$ (equivalently, $T$ is a $C_0$-locally equicontinuous semigroup, by Proposition \ref{2015-10-06:00}),
but there does not exist any $\alpha>0$ such that $\{e^{-\alpha t}T_t\}_{t\in \mathbb{R}^+}$ is equicontinuous. 
\end{Example}

The second relevant difference with respect to $C_0$-semigroups in Banach spaces is that the strong continuity does not imply, in general, the sequential local equicontinuity. %
The following example 
shows that 
Definition \ref{def:2015-05-01:01}\emph{(\ref{2015-12-05:24}$'$)} in general cannot be derived by 
Definition \ref{def:2015-05-01:01}\emph{(\ref{2015-10-06:01})}-\emph{(\ref{2015-10-13:11})}, even if 
Definition \ref{def:2015-05-01:01}\emph{(\ref{2015-10-13:11})} is strengthened by requiring the continuity of $\mathbb{R}^+\rightarrow X,\ t \mapsto T_tx$, $x\in X$.

\begin{Example}\label{exex}
  Let $X\coloneqq C(\mathbb{R})$ be endowed with the topology of the pointwise convergence. Define the semigroup
 $T\coloneqq \{T_t\}_{t\in \mathbb{R}^+}$ by
$$
T_t\colon X\rightarrow X,\ \varphi \mapsto \varphi(t+\cdot).
$$
Then $T_t\in \mathcal{L}_0(X)$ for all $t\in \mathbb{R}^+$.
It is clear that, for every $\varphi\in C(\mathbb{R})$, the map $\mathbb{R}^+\rightarrow X,\ t \mapsto T_t\varphi$, is continuous.
Nevertheless, for each $\hat t>0$ we can find a sequence $\{\varphi_n\}_{n\in \mathbb{N}}\subset C(\mathbb{R})$ of functions converging pointwise to $0$ and such that 
$$
\liminf_{n\rightarrow +\infty}\sup_{t\in[0,\hat t]} |(T_t\varphi_n)(0)|=
\liminf_{n\rightarrow +\infty}\sup_{t\in[0,\hat t]} |\varphi_n(t)|> 0.
$$
Hence,  $T$ is not a $C_0$-sequentially locally equicontinuous semigroup.
We observe that the same conclusion holds true if we restrict the action of $T$ to the space $C_b(\R)$.
\end{Example}

Referring to Remark \ref{rem..}\emph{(\ref{2015-12-05:25})}, we provide the following example\ (\footnote{Example \ref{ex:bb} could seem a bit artificious and \emph{ad hoc}.  
In the next section
we will provide 
 another more meaningful example  by a very simple  Markov transition semigroup (Example \ref{2015-05-04:07-II}).}).
\begin{Example}\label{ex:bb}
 Consider the Banach space $\ell^1$, with its usual  norm $|\mathbf{x}|_1=\sum_{k=0}^{+\infty} |x_k|$, where $\mathbf{x}\coloneqq \{x_k\}_{k\in\N}\in \ell^1$, and denote by
$\tau_1$ and
 $\tau_{w}$  the $|\cdot|_1$-topology and
 the weak topology  respectively.
Define $Z\coloneqq \ell^1\times \ell^1$ and endow it with the product topology $\tau_w \otimes \tau_1$.
Let
$$
B\colon Z\rightarrow Z, \ (x_1,x_2)\mapsto (x_1,x_1).
$$
We recall that  $\ell^1$ enjoys Schur's property (weak convergent sequences are strong convergent; see \cite[p.\ 85]{Diestel1984}). As a consequence, we have that $Z$ is sequentially complete and   $B\in \mathcal{L}_0(Z)$. On the other hand, as $\tau_w$ is strictly weaker than $\tau_1$, 
we have $B\notin L(Z)$.
By induction, we see
that $(I-B)^n= (I-B)$ for each $n\geq 1$, and then $\{(I-B)^{n}\}_{n\in \N}$ is a family of sequentially equicontinuous operators. 
By Proposition \ref{prop:2015-05-03:06},
if we define $T_t\coloneqq e^{t(B-I)}$ for $t\in \R^+$,
then
  $T\coloneqq \{T_t\}_{t\in \R^+}$ is a $C_0$-sequentially 
  equicontinuous semigroup on $Z$.
Actually, we have $e^{t(B-I)} = e^{-t}(I-B)+B$.
However, if $t>0$, the operators $e^{t(B-I)}=e^{-t}I+(1-e^{-t})B$ are not continuous on $Z$.
\end{Example}

\section{Developments in functional spaces}\label{sec:SE}

The aim of this section is to develop the theory of the previous section in some specific functional spaces. 
Throughout the rest of the paper, $E$ will denote a metric space, $\mathcal{E}$ will denote the associated Borel $\sigma$-algebra, and 
 $\mathcal{S}(E)$
will denote one of the spaces
 $UC_b(E)$, $C_b(E)$, $B_b(E)$.
We recall that $(\mathcal{S}(E),|\cdot|_\infty)$, where $|\cdot|_{\infty}$ is the usual sup-norm,  is a Banach space. For simplicity of notation,  we denote by $\mathcal{S}(E)_\infty^*$  the dual of $(\mathcal{S}(E),|
\cdot|_\infty)$ and  by $|\cdot|_{\mathcal{S}(E)^*_\infty}$  the operator norm in $\mathcal{S}(E)^*_\infty$.

 We are going to define on $\mathcal{S}(E)$ two particular locally convex topologies. The motivation for introducing such topologies is  that they allow to frame under a general unified viewpoint some of the approaches used in the literature of   Markov transition semigroups.
In particular, we are able to cover 
 the following types of  semigroups.
\begin{enumerate}
\item 
   Weakly continuous semigroups,  introduced in \cite{Cerrai94} for the space $UC_b(E)$ with $E$ separable Hilbert space (an overview can also be found in \cite[Appendix B]{Cerrai01book}, with $E$ separable Banach space).
\item  $\pi$-semigroups, introduced  in \cite{Priola99} for the space $UC_b(E)$, with $E$ separable metric space.
\item $C_0$-locally equicontinuous semigroups with respect to the so called mixed topology in the space $C_b(E)$, considered by \cite{GoldysKocan01}, with $E$ separable Hilbert space.
\end{enumerate}
\subsection{A family of locally convex topologies on $\mathcal{S}(E)$}\label{sub:tauk}

In the following, by $\mathbf{ba}(E)$ we denote the space of finitely additive signed measures on $(E,\mathcal{E})$ with bounded total variation. The space $\mathbf{ba}(E)$ is Banach  when endowed with the norm $| \cdot |_1$ given by the total variation and  is canonically identified with $(B_b(E)_\infty^*,|\cdot|_{B_b(E)^*_\infty})$ (see \cite[Theorem 14.4]{Aliprantis2006}) through the isometry
\begin{equation}\label{eqq1}
  \Phi\colon (\mathbf{ba}(E),|\cdot|_1)\rightarrow (B_b(E)_\infty^*,|\cdot|_{B_b(E)^*_\infty}),\ \mu\mapsto \Phi_\mu,
\end{equation}
where
\begin{equation}\label{eqq2}
\Phi_\mu(f)\coloneqq \int_Efd\mu\qquad\ \forall  f\in B_b(E),
\end{equation}
with $\int_E\# d\mu$  interpreted in the Darboux sense (see \cite[Sec. 11.2]{Aliprantis2006}).

We denote by $\mathbf{ca}(E)$  the space of elements of $\mathbf{ba}(E)$ that are countably additive.
 The space 
$(\mathbf{ca}(E),|\cdot|_1)$ is  Banach as well.
If $\mu\in \mathbf{ca}(E)$, then the Darboux integral in \eqref{eqq2} coincides with the Lebesgue integral.

For future reference, we recall the following result (see \cite[Th.\ 5.9, p.\ 39]{Parthasarathy67}).
\begin{Lemma}
  \label{rem:2015-04-27:02}
Let $\nu\in\mathbf{ca}(E)$
be such that  $\int_E {f} d\nu=0$ for all ${f}\in UC_b(E)$. Then  $\nu=0$. 
\end{Lemma}
\begin{Proposition}\label{2015-08-07:01}
 The space $(\mathbf{ca}(E),|\cdot|_1)$ is isometrically
embedded into $(\mathcal{S}(E)_\infty^*, |\cdot|_{\mathcal{S}(E)^*_\infty})$ by
\begin{equation}
  \label{eq:2015-08-07:00}
  \Phi\colon \mathbf{ca}(E)\rightarrow\mathcal{S}(E)_\infty^*,\ \mu\mapsto \Phi_\mu,
\end{equation}
where
\begin{equation}\label{phimu}
  \Phi_\mu(f)\coloneqq \int_Efd\mu, \qquad \forall  f\in \mathcal{S}(E).
\end{equation}
\end{Proposition}
\begin{proof}
It is clear that $\Phi$ is linear.

Let $\mu\in \mathbf{ca}(E)$.
As $|\Phi_\mu(f)|\leq |f|_\infty|\mu|_1$ 
for every $f\in \mathcal{S}(E)$, then $\Phi_\mu\in\mathcal{S}(E)^*$ and $|\Phi_\mu|_{\mathcal{S}(E)^*}\leq |\mu|_1$.
To  show  that  $\Phi$ is an isometry it remains to show that  $|\Phi_\mu|_{\mathcal{S}(E)^*}\geq |\mu|_1$.  
 Let $\mu=\mu^+-\mu^-$ be the Jordan decomposition of $\mu$, and let $C^+\coloneqq \mbox{supp}(\mu^+), \ C^-\coloneqq \mbox{supp}(\mu^-)$.
 Let $\varepsilon>0$.
  Then we can find a closed set $C^+_\varepsilon\subset C^+$  such that
 $\mu^+
 (C^+\setminus C^+_\varepsilon)<\varepsilon$,
 and $d(C^+_\varepsilon,C^-)>0$. 
 Let $f$ be defined by
 $$
 f(x)\coloneqq \frac{d(x,C^-)-d(x,C_\varepsilon^+)}{d(x,C^-)+d(x,C_\varepsilon^+)}\qquad \forall  x\in E.
 $$
 Then $f\in UC_b(E)$,  $f\equiv 1$ on $C^+_\varepsilon$,  $f\equiv -1$ on $C^-$, and $|f|_\infty=1$. Therefore,
 $$
 \int_E fd \mu= \int_{C_\varepsilon^+} fd \mu^+ +\int_{C^+\setminus C^+_\varepsilon} f d\mu^+
 -\int_{C^-} fd\mu^-  \geq \mu^+(C^+_\varepsilon)-\varepsilon +\mu^-(C^-)\geq |\mu|_1-2\varepsilon.
 $$
 Then $|\Phi_\mu|_{\mathcal{S}(E)^*_\infty}\geq |\mu|_1-2\epsilon$.
 We conclude by arbitrariness of $\varepsilon$.
\end{proof}

Let $\mathbf{P}$ be a set of non-empty parts of $E$ such that $E=\bigcup_{P\in \mathbf{P}}P$.
For every $P\in \mathbf{P}$ and every $\mu\in \mathbf{ca}(E)$, let us introduce the seminorm
\begin{equation}\label{maurosemi}
p_{P,\mu}({f})\coloneqq [ f]_P+\left|\int_E  {f} d\mu\right|, \quad \forall  f\in \mathcal{S}(E),
\end{equation}
where 
$$[ f]_P\coloneqq \sup_{x\in P}|{f}(x)|.$$
Denote by $\tau_\mathbf{P}$
the locally convex topology on $\mathcal{S}(E)$ induced by the family of seminorms 
$$\{ p_{P,\mu} \colon P\in \mathbf{P},
 \ \mu\in \mathbf{ca}(E) \}.$$
 Since $E=\bigcup_{P\in \mathbf{P}}P$, $\tau_\mathbf{P}$ is Hausdorff.

Let us denote by $\tau_\infty$ the topology induced by the norm $| \cdot|_\infty$ on $\mathcal{S}(E)$. 
Since the functional $\Phi_\mu$ defined in \eqref{phimu} is $\tau_\mathbf{P}$-continuous for every $\mu\in \mathbf{ca}(E)$, and since
  $p_{P,\mu}$ is $\tau_\infty$-continuous for every  $P\in \mathbf{P}$ and every $\mu\in \mathbf{ca}(E)$, we have the  inclusions
 \begin{equation}\label{eqrr}
 \sigma(\mathcal{S}(E),\mathbf{ca}(E))\subset \tau_\mathbf{P}\subset \tau_\infty.
 \end{equation}
Observe that, when $\mathbf{P}$ contains only finite parts of $E$, then $\tau_\mathbf{P}=\sigma(\mathcal{S}(E),\mathbf{ca}(E))$, because $\mathbf{ca}(E)$ contains all Dirac measures. The opposite case is when $E\in \mathbf{P}$, and then
 $\tau_\mathbf{P}=\tau_\infty$.
 

\begin{Proposition}
  \label{prop:2015-04-29:00}
Let $B\subset \mathcal{S}(E)$.
The following are equivalent.
\begin{enumerate}[(i)]
\item\label{2015-12-05:26}
$B$   is  $\sigma(\mathcal{S}(E),\mathbf{ca}(E))$-bounded.
\item\label{2015-12-05:27}
 $B$   is  $\tau_\mathbf{P}$-bounded.
\item\label{2015-12-05:28}
$B$   is  $\tau_\infty$-bounded.
\end{enumerate}
\end{Proposition}
\begin{proof}
By \eqref{eqrr}, it is sufficient to prove that \emph{(\ref{2015-12-05:26})}$\Rightarrow$\emph{(\ref{2015-12-05:28})}.
Let $B$ be $\sigma(\mathcal{S}(E),\mathbf{ca}(E))$-bounded.
By Proposition \ref{2015-08-07:01}, $\mathbf{ca}(E)$ is closed in $\mathcal{S}(E)^*_\infty$. Moreover, since $\mathbf{ca}(E)$ contains the Dirac measures, it is norming.
Then we conclude  by applying
Lemma \ref{lem:2015-05-04:01}.
\end{proof}

\begin{Corollary}
  $\mathcal{L}_0((\mathcal{S}(E),\tau_\mathbf{P}))\subset L((\mathcal{S}(E),|\cdot|_\infty))$.
\end{Corollary}
\begin{proof}
By Proposition \ref{prop:2015-04-29:00}, the   bounded sets of 
$\tau_\mathbf{P}$ are
 exactly the bounded sets of $\tau_\infty$.
Then, we conclude by applying  Proposition \ref{2015-07-30:02}\emph{(\ref{2015-08-31:01})}.
\end{proof}

\begin{Corollary}
Let $T$ be a $C_0$-sequentially locally  equicontinuous semigroup on $(\mathcal{S}(E),\tau_\mathbf{P})$.
Then there exists $M\geq 1$ and $\alpha >0$ such that $|T_t|_{L((\mathcal{S}(E),|\cdot|_\infty))}\leq Me^{\alpha t}$
for all $t\in \R^+$.
\end{Corollary}
\begin{proof}
Due to Proposition \ref{prop:2015-04-29:00}, 
we can conclude by applying Proposition \ref{prop:2015-08-01:00}\emph{(\ref{2015-10-13:00})}.
\end{proof}

\noindent We now  focus on the following two cases:
\begin{enumerate}[(a)]
\item 
 $\mathbf{P}$ is the set of all finite subsets of $E$, and then $\tau_\mathbf{P}=\sigma(\mathcal{S}(E),\mathbf{ca}(E))$;
 \item
 $\mathbf{P}$ is the set of all non-empty compact subsets of $E$; in  this case, we denote $\tau_\mathbf{P}$ by $\tau_\mathcal{K}$, that is
 \begin{equation}\label{tauk}
   \tau_\mathcal{K}\coloneqq \mbox{l.c.\ topology on } \mathcal{S}(E) \ \mbox{generated by} \  \{ p_{K,\mu} \colon K\subset E \ \mbox{compact}, 
   \ \mu\in \mathbf{ca}(E) \}.
 \end{equation}
\end{enumerate}

\begin{Proposition}
  \label{prop:2015-04-29:01} We have the following characterizations.
\begin{enumerate}[(i)]
\item\label{2015-12-05:29} $\tau_{\mathcal{K}}=\tau_\infty$ if and only if $E$ is compact.
\item\label{2015-12-05:30}
$\sigma(\mathcal{S}(E),\mathbf{ca}(E))=\tau_\infty$ if and only if $E$ is finite.
\end{enumerate}
\end{Proposition}
\begin{proof}
First, note that the inclusions $\sigma(\mathcal{S}(E),\mathbf{ca}(E))\subset \tau_\infty$ and $\tau_{\mathcal{K}}\subset\tau_\infty$ have been already observed in \eqref{eqrr}.

\emph{(\ref{2015-12-05:29})}
If $E$ is compact,
 we have $| \cdot|_\infty=p_{E,0}$, hence $\tau_\mathcal{K}=\tau_\infty$.
Conversely, assume that $\tau_\mathcal{K}=\tau_\infty$ on $\mathcal{S}(E)$. 
Then there exist a non-empty compact set  $K\subset E$, measures $\mu_1,\ldots,\mu_n\in \mathbf{ca}(E)$, and $L >0$, such that
  \begin{equation}
    \label{eq:2015-04-29:02}
| {f}|_\infty\leq L\left([f]_K+\sum_{i=1}^n\left|\int_E {f} d\mu_i\right|\right),\qquad  \forall   {f}\in \mathcal{S}(E).
\end{equation}
For $\varepsilon>0$, 
define $A_\varepsilon\coloneqq \{x\in E\colon B(x,\varepsilon)\subset K^c\}$, and define, with the convention $d(
\cdot,\emptyset) =+\infty$, the function $r_\varepsilon(x)\coloneqq \frac{d(x,K)}{d(x,A_\varepsilon)+d(x,K)}$. Then $0\leq r_\varepsilon\leq 1$, $r_\varepsilon=0$ on $K$, $r_\varepsilon=1$ on $A_\varepsilon$, $r_\varepsilon\uparrow \mathbf{1}_{K^c}$ pointwise as $\varepsilon\downarrow 0$, and $r_\varepsilon$ is uniformly continuous (the latter is due
to
the fact that $d(A_\varepsilon,K)\geq \varepsilon$).
Hence, for every $f\in \mathcal{S}(E)$, the function ${f} r_\varepsilon$ belongs to $\mathcal{S}(E)$ and $|{f} r_\varepsilon|\uparrow |{f}\mathbf{1}_{K^c}|$ pointwise as $\varepsilon\downarrow 0$, which entails  $|f r_\varepsilon|_\infty\uparrow | {f} \mathbf{1}_{K^c}|_\infty$ as $\varepsilon\downarrow 0$.
We can then apply \eqref{eq:2015-04-29:02}
to every $fr_\varepsilon$ and pass to the limit for $\varepsilon\downarrow 0$
 to obtain
$$
| {f} \mathbf{1}_{K^c}|_\infty\leq L
\sum_{i=1}^n\left|\int_E {f} d(\mu_i\lfloor K^c)\right|,\qquad \forall  {f}\in \mathcal{S}(E),
$$
where
$\mu_i\lfloor K^c$ denotes the restriction of $\mu_i$ to $K^c$.
Let $\nu\in \mathbf{ca}(E)$ be such that $|\nu|(K)=0$. Then
%
$$
\left|\int_E{f} d\nu\right| = 
\left|\int_E{f}\mathbf{1}_{K^c} d\nu\right|
 \leq \ |\nu|_1
| {f}\mathbf{1}_{K^c}|_\infty
\leq |\nu|_1
L
\sum_{i=1}^n\left|\int_E {f} d(\mu_i\lfloor K^c)\right|, \qquad  \forall  {f}\in \mathcal{S}(E).
$$
Then, by  \cite[Lemma\ 3.9, p.\ 63]{Rudin1991} and by Proposition \ref{2015-08-07:01}, there exist $\alpha_1,\ldots,\alpha_n\in \mathbb{R}$ such that
$ \nu=\sum_{i=1}^n\alpha_i(\mu_i\lfloor K^c)$. By arbitrariness of $\nu$ this implies that $E\setminus K$ is finite, and then $E$ is compact.

\emph{(\ref{2015-12-05:30})}
 If $E$ is finite, clearly  $\sigma(\mathcal{S}(E),\mathbf{ca}(E))= \tau_\infty$.
 Conversely, assume that  $\sigma(\mathcal{S}(E),\mathbf{ca}(E))= \tau_\infty$.
 Then there exist $K\subset E$ compact, $\mu_1,\ldots,\mu_n\in \mathbf{ca}(E)$, and $L>0$ such that
\begin{equation}\label{2015-10-27:00}
|f|_\infty\leq L
\sum_{i=1}^n\left|\int_E {f} d\mu_i\right|,\qquad  \forall   {f}\in \mathcal{S}(E).
\end{equation}
By arguing as for concluding the proof of \emph{(\ref{2015-12-05:29})},
we obtain
$$
\mathbf{ca}(E)=\operatorname{Span} \left\{ \mu_1,\ldots,\mu_n \right\},
$$
and then $E$ must be finite.
\end{proof}
%

We recall the following definition.

\begin{Definition}
A locally convex topological vector space is said to be \emph{infrabarreled} if every closed, convex, balanced set, absorbing every bounded set, is a neighborhood of $0$.
\end{Definition}

\begin{Corollary}
  \label{cor:2015-04-29:05} We have the following characterizations.
  \begin{enumerate}[(i)]
  \item\label{2015-12-05:31}
 $(\mathcal{S}(E),\sigma(\mathcal{S}(E),\mathbf{ca}(E)))$ is  infrabarrelled if and only if $E$ is finite.
\item\label{2015-12-05:32} $(\mathcal{S}(E),\tau_\mathcal{K})$ is  infrabarrelled if and only if $E$ is compact.
\end{enumerate}
\end{Corollary}
\begin{proof}
If $E$ is finite (\emph{resp.}\ $E$ is compact), then,  by Proposition 
\ref{prop:2015-04-29:01},
$\sigma(\mathcal{S}(E),\mathbf{ca}(E))$ (\emph{resp.}\ $\tau_\mathcal{K}$)
coincides with the topology $\tau_\infty$ of the Banach space $(\mathcal{S}(E),|\cdot|_\infty)$, and then
it is infrabarreled, because every Banach space is so (see \cite[Theorem 4.5, p.\ 97]{Osborne2014}).

Conversely, let 
 $E$ be not 
finite (\emph{resp.}\  not compact) and consider the $|\cdot|_\infty$-closed ball 
 $$B_\infty(0,1] \coloneqq  \{f\in \mathcal{S}(E)\colon \ |f|_\infty\leq 1\}.
$$
The set $B_\infty(0,1]$
 is convex, balanced, absorbent. 
Moreover,
$$
B_\infty(0,1]=\bigcap_{x\in E}\left\{f \in \mathcal{S}(E) \colon \left|\int_E {f}d\delta_x\right|\leq 1\right\},
$$
where $\delta_x\in \mathbf{ca}(E)$ is the Dirac measure centered in $x$. Hence $B_\infty(0,1]$ is
$\sigma(\mathcal{S}(E),\mathbf{ca}(E))$-closed (and then
 $\tau_\mathcal{K}$-closed). So  $B_\infty(0,1]$ is a barrel for the topology $\sigma(\mathcal{S}(E),\mathbf{ca}(E))$  (\emph{resp.}\  
 $\tau_\mathcal{K}$). Moreover, by Proposition \ref{prop:2015-04-29:00}, it absorbs every
$\sigma(\mathcal{S}(E),\mathbf{ca}(E))$- (\emph{resp.}\ 
 $\tau_\mathcal{K}$-) bounded set. 
Assuming now, by contradiction, that  $(\mathcal{S}(E),\sigma(\mathcal{S}(E),\mathbf{ca}(E)))$ (\emph{resp.}\   $(\mathcal{S}(E),\tau_{\mathcal{K}})$)  is infrabarreled, we would have that  
$B_\infty(0,1]$ 
 is a
$\sigma(\mathcal{S}(E),\mathbf{ca}(E))$-neighborhood
(\emph{resp.}\  $\tau_\mathcal{K}$-neighborhood) of the origin. This would contradict  Proposition \ref{prop:2015-04-29:01}.
\end{proof}

\begin{Remark}
Corollary \ref{cor:2015-04-29:05} has an important consequence.
If $E$ is not finite (\emph{resp.}\  not compact), then
$\sigma(\mathcal{S}(E),\mathbf{ca}(E))$
(\emph{resp.}\  $(\mathcal{S}(E),\tau_\mathcal{K})$) is not infrabarreled, so the Banach-Steinhaus theorem  cannot be invoked to deduce that strongly continuous semigroups in
$(\mathcal{S}(E),\sigma(\mathcal{S}(E),\mathbf{ca}(E)))$
(\emph{resp.}\ 
 $(\mathcal{S}(E),\tau_\mathcal{K})$) are necessarily locally equicontinuous --- as it is usually done for $C_0$-semigroups in Banach spaces (cf. also Example \ref{exex}). 
\end{Remark}

We now investigate the relationship between $\tau_\mathcal{K}$ and $\tau_\mathcal{C}$,
where $\tau_\mathcal{C}$ denotes
 the topology on $\mathcal{S}(E)$ defined by the uniform convergence on compact sets of $E$, induced by the family of seminorms 
$$
\{ p_K = [\cdot]_K \colon \ K \ \mbox{non-empty compact subset of} \ E \}.
$$
Clearly 
$\tau_\mathcal{C}\subset \tau_\mathcal{K}$. 
In order to understand when  the equality  $\tau_\mathcal{C} = \tau_\mathcal{K}$  is possible,
we proceed with two preparatory lemmas.

\begin{Lemma}\label{lem:0000-00-00:00}
  \label{2015-07-29:04}
  $UC_b(E)\neq C_b(E)$ if and only if there exists a sequence $\{(x_n,y_n)\}_{n\in \mathbb{N}}\subset E\times E$  having the following properties.
  \begin{enumerate}[(i)]
  \item \label{2015-12-05:33}
 $\{d(x_n,y_n)\}_{n\in \mathbb{N}}$ is  a strictly positive sequence, converging to $0$;
  \item \label{2015-12-05:34} the sequence $\{d_n\}_{n\in \mathbb{N}}$ defined by $d_n\coloneqq d\left(\{x_n,y_n\},\bigcup_{k>n}\{x_k,y_k\}\right)$, for $n\in \mathbb{N}$, 
is strictly positive;
  \item \label{2015-12-05:35}  the sequence $\{x_n\}_{n\in \mathbb{N}}$ does not have any convergent subsequence.
  \end{enumerate}
\end{Lemma}
\begin{proof}
We first prove that, if $UC_b(E)\neq C_b(E)$, then there exists a sequence satisfying \emph{(\ref{2015-12-05:33})},\emph{(\ref{2015-12-05:34})},\emph{(\ref{2015-12-05:35})}.
  Let $f\in C_b(E)\setminus UC_b(E)$.  Then  there exist $\varepsilon>0$ and a sequence $\{(x_n,y_n)\}_{n\in \mathbb{N}}\subset E\times E$ such that $\lim_{n\rightarrow +\infty}d(x_n,y_n)=0$ and $\inf_{n\in \mathbb{N}}|f(x_n)-f(y_n)|\geq \varepsilon$. 
Then \emph{(\ref{2015-12-05:33})} is  satisfied by $\{(x_n,y_n)\}_{n\in \mathbb{N}}$.
Now we show that \emph{(\ref{2015-12-05:34})} holds. Assume, by contradiction, that  $d_{\hat n}=0$ for some $\hat n\in \mathbb{N}$. Then $d\left(z,\bigcup_{k>\hat n}\{x_k,y_k\}\right)=0$ for $z=x_{\hat n}$ or $z=y_{\hat n}$. Therefore $z$ is an accumulation point for $\bigcup_{k>\hat n}\{x_k,y_k\}$. Hence, as $d(x_n,y_n)\rightarrow 0$, there exists a subsequence $\{(x_{n_k}, y_{n_k})\}_{k\in \N}$ such that $x_{n_k}\rightarrow z$ and $y_{n_k}\rightarrow z$ as $k\rightarrow +\infty$.  Now, as $f$ is continuous, we have the contradiction  $f(z)-f(z)=\lim_{k\rightarrow +\infty} |f(x_{n_k})-f(y_{n_k})|\geq \varepsilon$.
Finally, property \emph{(\ref{2015-12-05:35})}
can be proved
 by using the same argument as for proving \emph{(\ref{2015-12-05:34})}.

Conversely, take 
a sequence $\{(x_n,y_n)\}_{n\in \mathbb{N}}\subset E\times E$  satisfying \emph{(\ref{2015-12-05:33})},\emph{(\ref{2015-12-05:34})},\emph{(\ref{2015-12-05:35})}.
Consider the balls 
\begin{equation}\label{balls}
B_n\coloneqq\left\{x\colon d(x_n,x)< \varepsilon_n \right\}, \quad n\in \mathbb{N},
\end{equation}
where $\{\varepsilon_n\}_{n\in\N}$ is recursively defined by 
$$
\begin{dcases}
  \varepsilon_0 \coloneqq \frac{d_0\wedge d(x_0,y_0)}{2}\\
 \varepsilon_n\coloneqq \frac{d_n\wedge d (x_n,y_n)\wedge \varepsilon_{n-1} }{2} & n\geq 1.
\end{dcases}
$$
By the properties \emph{(\ref{2015-12-05:33})},\emph{(\ref{2015-12-05:34})}, 
the balls $\{B_n\}_{n\in \mathbb{N}}$
are pairwise disjoint and $\lim_{n\rightarrow +\infty}\epsilon_n=0$. It is also clear that $y_n\notin B_n$, for  $n\in\N$.
For every $n\in\N$,  we can construct a uniformly continuous function $\rho_n$ such that $0\leq \rho_n\leq 1$, $\rho_n(x_n)=1$, and $\rho_n=0$ on $B_n^c$. 
For $n\in \mathbb{N}$, the function $f_n\coloneqq \sum_{i=0}^n\rho_i$ is uniformly continuous.
Let 
$f\coloneqq \sum_{i=0}^{+\infty} \rho_i$. 
By
\emph{(\ref{2015-12-05:35})} and since $\varepsilon_n\rightarrow 0$, 
one can show that every converging sequence in $E$ can intersect only a finite number of the pairwise disjoint balls
 $\{B_n\}_{n\in \mathbb{N}}$.
Hence, any compact set $K\subset E$ intersects only a finite number of balls $\{B_n\}_{n\in\N}$.
Then $f$ restricted to any compact set $K\subset E$ is actually a finite sum of the form $\sum_{i=1}^{n_K}\rho_i$, 
that is, it coincides with $f_{n_K}$,
for some $n_K\in \mathbb{N}$ depending on $K$. In particular, $f\in C_b(E)$.
On the other hand,
$f(x_n)-f(y_n)=1$ and $d(x_n,y_n)\rightarrow 0$ as $n\rightarrow+\infty$, so
 $f\not \in UC_b(E)$.
\end{proof}


\begin{Lemma}
  \label{2015-08-08:00}
  If $E$ is not complete, then $UC_b(E)\neq C_b(E)$.
\end{Lemma}
\begin{proof}
Let $\{x_n\}_{n\in \mathbb{N}}$ be a non-convergent Cauchy sequence in $E$ and define $y_n\coloneqq x_{2n}$, for $n\in \mathbb{N}$.  
We now show that, up to extract a subsequence, the sequence  $\{(x_n,y_n)\}_{n\in \mathbb{N}}$ 
satisfies 
\emph{(\ref{2015-12-05:33})},\emph{(\ref{2015-12-05:34})},\emph{(\ref{2015-12-05:35})}
of Lemma \ref{lem:0000-00-00:00}.

We prove  property \emph{(\ref{2015-12-05:33})}. As $\{x_n\}_{n\in \mathbb{N}}$ is Cauchy and non-convergent, up to extract a subsequence, we can assume that $x_n\neq x_k$, if $n\neq k$,  hence $d(x_n,y_{n})>0$. On the other hand, since $\{x_n\}_{n\in \mathbb{N}}$ is Cauchy, we have $\lim_{n\rightarrow +\infty}d(x_n,y_n)=0$. 

We prove property \emph{(\ref{2015-12-05:34})}.
Let $\{d_n\}_{n\in \mathbb{N}}$ be defined as in 
Lemma \ref{lem:0000-00-00:00}\emph{(\ref{2015-12-05:34})}.
Assume, by contradiction,  that $d_{\overline n}=0$ for some $\overline{n}\in\N$. Then $z=x_{\overline n}$ or $z=y_{\overline n}$
 should be an accumulation point for the sequence $\{x_n\}_{n\in \mathbb{N}}$ or for the sequence $\{y_n=x_{2n}\}_{n\in \mathbb{N}}$, which cannot be true by assumption on $\{x_n\}_{n\in \mathbb{N}}$. 
 
Finally, property \emph{(\ref{2015-12-05:35})} is clear from the fact that $\{x_n\}_{n\in \mathbb{N}}$ is Cauchy and non-convergent.
\end{proof}

\begin{Proposition}
\label{prop:2015-05-04:07}
$\tau_\mathcal{K}=\tau_\mathcal{C}$ on $\mathcal{S}(E)$  if and only if $E$ is  compact.
\end{Proposition}
\begin{proof}
If $E$ is compact, it is clear that $\tau_\mathcal{K}=\tau_\mathcal{C}$.
Suppose now that $E$ is not compact.  
We recall that $E$ is not compact if and only if $E$ is not complete or $E$ is not totally bounded. In both cases,
we will show that there exists a sequence $\{\varphi_n\}_{n\in \mathbb{N}}\subset UC_b(E)$ convergent to $0$ in $\tau_\mathcal{C}$, but  unbounded in $\tau_\mathcal{K}$. 

\emph{\underline{Case $E$ non-complete.}}
By Lemma \ref{2015-08-08:00}, there exists a sequence $\{(x_n,y_n)\}_{n\in \mathbb{N}}\subset E\times E$ satisfying
\emph{(\ref{2015-12-05:33})},\emph{(\ref{2015-12-05:34})},\emph{(\ref{2015-12-05:35})}
of Lemma \ref{lem:0000-00-00:00}.
 Let  $\{B_n\}_{n\in \mathbb{N}}$ and  $\{\rho_n\}_{n\in \mathbb{N}}$
be as in the second part of the proof of Lemma \ref{lem:0000-00-00:00}.
Define $\varphi_n\coloneqq 2^{2n}\rho_n$ for every $n\in \mathbb{N}$.
As proved in that lemma, any compact set $K\subset E$ intersects only a finite numbers of balls $\{B_n\}_{n\in\N}$, therefore $\lim_{n\rightarrow +\infty}\varphi_n=0$
in $(UC_b(E),\tau_\mathcal{C})$.

Now,  let $\mu\in\mathbf{ca}(E)$ be defined by $\mu\coloneqq \sum_{n\in \mathbb{N}}2^{-n}\delta_{x_n}$. We have
$$
\sup_{n\in \mathbb{N}}\left|\int_E \varphi_nd\mu\right|=\sup_{n\in \mathbb{N}}2^{-n}\varphi_n(x_n)=\sup_{n\in \mathbb{N}}2^n=+\infty,
$$
which shows that $\{\varphi_n\}_{n\in \mathbb{N}}$ is $\tau_\mathcal{K}$-unbounded. 

\emph{\underline{Case $E$ not totally bounded.}}
Let $\varepsilon>0$ be such that $E$ cannot be covered by a finite number of balls of radius $\varepsilon$. 
By induction, we can construct a sequence $\{x_n\}_{n\in \mathbb{N}}\subset E$ such that, for every $n\in \mathbb{N}$, $x_{n+1}\not \in\bigcup_{j=0}^n {B(x_j,\varepsilon)}$. For every $n\in \mathbb{N}$, let $\varphi_n\in UC_b(E)$ be such that $\varphi_n(x_n)=2^{2n}$, $\varphi_n(x)=0$ if $d(x,x_n)\geq \varepsilon/2$, $|\varphi_n|_\infty=2^{2n}$ 
 (\footnote{For instance,  
 $\varphi_n(x)\coloneqq 2^{2n}\frac{d(x,B(x_n,\varepsilon/2)^c)}{d(x,x_n)+d(x,{B(x_n,\varepsilon/2)}^c)}$.}). 
Then we conclude as in the previous case.
\end{proof}


 Propositions \ref{prop:2015-04-29:01} and \ref{prop:2015-05-04:07} yield the following inclusions of topologies in the space $\mathcal{S}(E)$ 
$$
\tau_\mathcal{C}\subset \tau_\mathcal{K}\ \subset \ \tau_\infty
$$
and state that such inclusions are equalities 
 if and only if $E$ is compact.
The following proposition makes clearer the connection between  $\tau_\mathcal{K}$ and  $\tau_\mathcal{C}$ when $E$ is not compact.

\begin{Proposition}
  \label{prop:2015-04-24:01}
The following statements hold.
  \begin{enumerate}[(i)]
\item\label{2015-10-13:10} If a  net $\{{f}_\iota\}_{\iota\in \mathcal{I}}$ is
bounded
and
convergent to ${f}$ in $(\mathcal{S}(E),\tau_\mathcal{K})$, then
$$
\sup_{\iota\in \mathcal{I}}| {f}_\iota|_\infty<+\infty\ \ \  \mbox{and}
\ \ \ \lim_\iota f_\iota=f\ \mbox{in }(\mathcal{S}(E),\tau_\mathcal{C}).
$$
If either $\mathcal{I}=\mathbb{N}$ or $E$ is
homeomorphic to a Borel subset of a Polish space,
 then also the converse holds true. 
  \item\label{2015-12-06:00}
If a  net $\{{f}_\iota\}_{\iota\in \mathcal{I}}$ is
bounded
and
Cauchy  in $(\mathcal{S}(E),\tau_\mathcal{K})$,
 then
$$
\sup_{\iota\in \mathcal{I}}| {f}_\iota|_\infty<+\infty\ \ \ \mbox{and}
\ \ \ \{{f}_\iota\}_\iota\mbox{ is Cauchy in }(\mathcal{S}(E),\tau_\mathcal{C}).
$$
If either $\mathcal{I}=\mathbb{N}$ or $E$ is 
homeomorphic to a Borel subset of a Polish space,
 then also the converse holds true.
\end{enumerate}
\end{Proposition}

\begin{proof}
\emph{(\ref{2015-10-13:10})}
Let $\{{f}_\iota\}_{\iota\in \mathcal{I}}$ be a $\tau_\mathcal{K}$-bounded net 
converging to $f$ in $(\mathcal{S}(E),\tau_\mathcal{K})$. 
By  Proposition \ref{prop:2015-04-29:00}
 we have $\sup_{\iota\in \mathcal{I}}| {f}_\iota|_\infty<+\infty$, and, 
since $\tau_\mathcal{C}\subset \tau_\mathcal{K}$, the net converges to $f$ also with respect to $\tau_\mathcal{C}$.

Conversely, let $\{f_\iota\}_{\iota\in \mathcal{I}}\subset \mathcal{S}(E)$ be such that $\sup_\iota | {f}_\iota|_\infty=M<+\infty$ and  $\lim_\iota f_\iota
=f$ in $(\mathcal{S}(K),\tau_\mathcal{C})$.
Then $\{f_\iota\}_{\iota\in \mathcal{I}}$ is $\tau_\mathcal{K}$-bounded, because $\tau_\mathcal{K}\subset  \tau_\infty$.
 We want to prove that $\{f_\iota\}_{\iota\in \mathcal{I}}$ is $\tau_\mathcal{K}$-convergent to $f$ if $\mathcal{I}=\mathbb{N}$ or if $E$
homeomorphic to a Borel subset of a Polish space.
Assume without loss of generality $f=0$.
 We already know that $[f_\iota]_K$ converges to $0$ for every compact set $K\subset E$, then it remains to show that $\int_E f_\iota d\mu$ converges to $0$ for every $\mu\in \mathbf{ca}(E)$.
 If $\mathcal{I}=\mathbb{N}$, this follows by dominated convergence theorem, because $\sup_\iota | {f}_\iota|_\infty<+\infty$.
If $E$ is
homeomorphic to a Borel subset of a Polish space,
 then $|\mu|$ is tight (see \cite[p.\ 29, Theorem 3.2]{Parthasarathy67}), so, given $\varepsilon>0$, there exists $K_\varepsilon\subset E$ compact such that $|\mu|(K_\varepsilon^c)<\varepsilon$. Let
 $\overline \iota\in \mathcal{I}$ be such that $\iota\succeq \overline \iota $ implies $\sup_{\iota\succeq \overline \iota}[{f}_\iota]_{ K_\varepsilon}<\varepsilon$ (this is possible by uniform convergence of $\{f_\iota\}_{\iota\in \mathcal{I}}$ to $0$ on compact sets).
Then
\begin{equation*}
  \begin{split}
\left|\int_E f_\iota d\mu\right|\leq  \int_E|{f}_\iota|d|\mu|
\leq
[{f}_\iota]_{K_\varepsilon}| \mu|_1+\int_{K_\varepsilon^c}|{f}_\iota|d|\mu|\leq
| \mu|_1\sup_{\iota\succeq \overline \iota}[{f}_\iota]_{ K_\varepsilon}+| {f}_\iota |_\infty|\mu|(K_\varepsilon^c)
\leq
(| \mu|_1+M)\varepsilon, \ \ \ \forall  \iota \succeq \overline \iota,
\end{split}
\end{equation*}
and we conclude by arbitrariness of $\varepsilon$. 

\emph{(\ref{2015-12-06:00})} The proof is analogous to that of \emph{(\ref{2015-10-13:10})}.
\end{proof}
%

We have a similar proposition relating $\sigma(\mathcal{S}(E),\mathbf{ca}(E))$ and the pointwise convergence in $\mathcal{S}(E)$. Actually, a part of this proposition is  implicitly provided by \cite[Theorem 2.2]{Priola99}, where the separability of $E$ and the choice $\mathcal{S}(E)=UC_b(E)$ play no role.

\begin{Proposition}
  \label{prop:2015-04-24:01BB}
The following statements hold.
  \begin{enumerate}[(i)]
\item\label{2015-10-30:05} If a  net $\{{f}_\iota\}_{\iota\in \mathcal{I}}$ is
bounded
and
convergent to ${f}$ in $(\mathcal{S}(E),\sigma(\mathcal{S}(E),\mathbf{ca}(E)))$, then
$$
\sup_{\iota\in \mathcal{I}}| {f}_\iota|_\infty<+\infty\ \ \  \mbox{and}
\ \ \ \lim_\iota f_\iota=f\ \mbox{pointwise}.
$$
If $\mathcal{I}=\mathbb{N}$ 
 then also the converse holds true. 
  \item\label{2015-12-06:01}
If a  net $\{{f}_\iota\}_{\iota\in \mathcal{I}}$ is
bounded
and
Cauchy  in $(\mathcal{S}(E),\sigma(\mathcal{S}(E),\mathbf{ca}(E)))$,
 then
$$
\sup_{\iota\in \mathcal{I}}| {f}_\iota|_\infty<+\infty\ \ \ \mbox{and}
\ \ \ \{{f}_\iota(x)\}_\iota\mbox{ is Cauchy for every }x\in E.
$$
If  $\mathcal{I}=\mathbb{N}$
 then also the converse holds true.
\end{enumerate}
\end{Proposition}
\begin{proof}
\emph{(\ref{2015-10-30:05})}
Let $\{{f}_\iota\}_{\iota\in \mathcal{I}}$ be a bounded net in $(\mathcal{S}(E),\sigma(\mathcal{S}(E),\mathbf{ca}(E)))$, converging to
$f$ in this space.
By  Proposition \ref{prop:2015-04-29:00}
 we have $\sup_{\iota\in \mathcal{I}}| {f}_\iota|_\infty<+\infty$, and, 
since $\mathbf{ca}(E)$ contains the Dirac measures,
 the net converges to $f$ also pointwise.
Conversely, let $\{f_n\}_{n\in \mathbb{N}}\subset \mathcal{S}(E)$ be such that $\sup_{n\in \mathbb{N}} | {f}_n|_\infty=M<+\infty$ and $\lim_{n\rightarrow +\infty} f_n
=f$ pointwise.
Then an application of Lebesgue's dominated convergence theorem provides
$\lim_{n\rightarrow +\infty}f_n=f$ in 
$(\mathcal{S}(E),\sigma(\mathcal{S}(E),\mathbf{ca}(E)))$.

\emph{(\ref{2015-12-06:01})} The proof is analogous to that of \emph{(\ref{2015-10-30:05})}.
\end{proof}
%
%

\begin{Proposition}\label{prop:seqcom}
The following statements hold. 
  \begin{enumerate}[(i)]
  \item\label{2015-12-06:02}
 $(B_b(E),\sigma(B_b(E),\mathbf{ca}(E)))$ and 
 $(B_b(E),\tau_{\mathcal K})$ are sequentially complete.
  \item\label{2015-11-10:07} $C_b(E)$ is $\tau_\mathcal{K}$-closed
    in $B_b(E)$ (hence, by (\ref{2015-12-06:02}), $(C_b(E),\tau_\mathcal{K})$ is sequentially complete).
  \item\label{2015-12-06:03} If $E$ is homeomorphic to a Borel subset of a Polish space, then $UC_b(E)$ 
    is dense in $(C_b(E),\tau_\mathcal{K})$.
  \item\label{2015-12-06:04} $(UC_b(E),\tau_\mathcal{K})$ is sequentially complete if and only if
    $UC_b(E)=C_b(E)$.
  \item\label{2015-12-06:05} 
$(\mathcal{S}(E),\tau_{\mathcal K})$ is metrizable if and only if $E$ is compact.
  \end{enumerate}
\end{Proposition}

\begin{proof}
\emph{(\ref{2015-12-06:02})} Let $\{{f}_n\}_{n\in \mathbb{N}}$ be $\tau_{\mathcal K}$-Cauchy in $B_b(E)$. Then, as every Cauchy sequence is bounded, by Proposition \ref{prop:2015-04-24:01}\emph{(\ref{2015-12-06:00})},  the  sequence is   $\tau_\infty$-bounded. Then  its pointwise limit  ${f}$  (that clearly exists)  belongs to $B_b(E)$. 
By Proposition \ref{prop:2015-04-24:01}\emph{(\ref{2015-12-06:00})},
the convergence is uniform on every compact subset of $E$. Then 
Proposition \ref{prop:2015-04-24:01}\emph{(\ref{2015-10-13:10})}
 implies that $\{{f}_n\}_{n\in \mathbb{N}}$ is $\tau_{\mathcal K}$-convergent to $f$. This shows that  $(B_b(E),\tau_{\mathcal K})$ is sequentially complete.

By using Proposition 
\ref{prop:2015-04-24:01BB},
a similar argument
shows that also
 $(B_b(E),\sigma(B_b(E),\mathbf{ca}(E)))$
 is sequentially complete.

\emph{(\ref{2015-11-10:07})}
Let $\{{f}_\iota\}_{\iota\in \mathcal{I}}\subset C_b(E)$ be a net $\tau_\mathcal{K}$-converging to $f$ in  $B_b(E)$. In particular,  the convergence
is uniform on compact sets, hence $f\in C_b(E)$.  

\emph{(\ref{2015-12-06:03})}
Let $f\in C_b(E)$, let $K$ be a compact subset of $E$,  let $\mu_1,\ldots,\mu_n\in \mathbf{ca}(E)$, and let $\varepsilon>0$.
We show that there exists $g\in UC_b(E)$ such that
$\max_{i=1,\ldots,n}p_{K,\mu_i}(f-g)\leq \epsilon$. This will prove the density of $UC_b(E)$ in $C_b(E)$ with respect to $\tau_\mathcal{K}$.
 Since
$E$ is homeomorphic to a Borel subset of a Polish space, 
the finite family 
 $|\mu_1|,\ldots,|\mu_n|$ is tight
 (see \cite[Theorem 3.2, p.\ 29]{Parthasarathy67}). Hence,   there exists a compact set $K_\varepsilon$ such that $\max_{i=1,\ldots,n}|\mu_i|(K^c_\varepsilon)<\frac{\varepsilon}{2(1+| {f}|_\infty)}$. Let $g\in UC_b(E)$ be a uniformly continuous extension of   ${f}_{|K\cup K_\varepsilon}$ such that $|g|_\infty\leq | f|_\infty$. Then
$$
\max_{i=1,\ldots,n}p_{K,\mu_i}(f- g) \leq \ [f-g]_{K}+\max_{i=1,\ldots,n}\int_E|f-g|d|\mu_i|\leq 2| {f}|_\infty
\max_{i=1,\ldots,n}|\mu_i|(K^c_\varepsilon) \leq \ \varepsilon.
$$

\emph{(\ref{2015-12-06:04})} 
If $UC_b(E)=C_b(E)$, then the sequential completeness of  $(UC_b(E),\tau_\mathcal{K})$ follows from \emph{(\ref{2015-11-10:07})} of the present proposition.

 Suppose that $UC_b(E)\neq C_b(E)$.
Let $\{B_n\}_{n\in \mathbb{N}}$, $\{f_n\}_{n\in \mathbb{N}}\subset UC_b(E)$, and $f\in C_b(E)\setminus UC_b(E)$ be as in the second part of the proof of  Lemma \ref{lem:0000-00-00:00}. To show that $UC_b(E)$ is not sequentially complete, we will show that $\lim_{n\rightarrow +\infty}f_n= f$ in
$(C_b(E),\tau_\mathcal{K})$.
Let $K\subset  E$ be  compact  and $\mu\in \mathbf{ca}(E)$. As observed in the proof of  Lemma \ref{lem:0000-00-00:00}, $f=\sum_{i=1}^{n_K}\rho_i$ on $K$, for some $n_K\in \mathbb{N}$ depending on $K$, and then $[f-f_n]_K=0$ for every $n\geq n_K$.
Then 
\begin{equation*}
  \begin{split}
    \limsup_{n\rightarrow+\infty}p_{K,\mu}(f-f_n)&
=
\limsup_{n\rightarrow+\infty}p_{K,\mu}\left(\sum_{i=n+1}^{+\infty} \rho_i\right)
=
\limsup_{n\rightarrow+\infty}\left |\int_E\left(\sum_{i=n+1}^{+\infty} \rho_i\right)d\mu\right|\\
&\leq
\lim_{n\rightarrow+\infty} \sum_{i=n+1}^{+\infty}\int_E \rho_id|\mu|
\leq
\lim_{n\rightarrow+\infty} \sum_{i=n+1}^{+\infty} |\mu|(B_i)\\
&=\lim_{n\rightarrow+\infty} |\mu|\left(\bigcup_{i\geq n+1}B_i\right)=
 |\mu|\left(\bigcap_{n\geq 1} \bigcup_{i\geq n+1}B_i\right).
\end{split}
\end{equation*}
As the balls $\{B_n\}_{n\in \mathbb{N}}$ are pairwise disjoint,  we have $\bigcap_{n\geq 1}\bigcup_{i\geq n}B_i=\emptyset$. Hence, the last term in the inequality above is $0$ and we conclude.

\emph{(\ref{2015-12-06:05})} 
If $E$ is compact, then Proposition \ref{prop:2015-04-29:01} yields $\tau_\mathcal{K}=\tau_\infty$, hence $(\mathcal{S}(E),\tau_\mathcal{K})$ is metrizable.

If $E$ is not compact, in order to prove that $(\mathcal{S}(E),\tau_\mathcal{K})$ is not metrizable, it will be sufficient to prove that every $\tau_\mathcal{K}$-neighborhood of $0$ contains a non-degenerate vector space.
Indeed, in such a case, if $\hat{d}$ was a metric inducing $\tau_\mathcal{K}$, there would exist a sequence $\{x_n\}_{\in \mathbb{N}}$, such that $\lim_{n\rightarrow+\infty}\hat d(x_n,0)=0$ and $\lim_{n\rightarrow \infty}|x_n|_\infty=+\infty$. But then $\{x_n\}_{n\in \mathbb{N}}$ would converge to $0$ in $\tau_\mathcal{K}$, and then the sequence would be $|\cdot|_{\infty}$-bounded, by Proposition \ref{prop:2015-04-24:01}\emph{(\ref{2015-10-13:10})}, providing the contradiction.

To show that every neighborhood of $0$ in $\tau_\mathcal{K}$ contains a non-degenerate vector space, let $K\subset E$ be  compact,  $\mu_1,\ldots,\mu_m\in \mathbf{ca}(E)$, $\varepsilon>0$, and consider the neighborhood
$$
\mathcal{I}\coloneqq \{f\in \mathcal{S}(E):\ p_{K,\mu_i}(f)<\varepsilon,\ \forall  i=1,\ldots,m\}.
$$
Since $E$ is not compact, by Lemma \ref{2015-08-08:00}, $UC_b(E)\neq C_b(E)$. Hence, we can construct  the sequence $\{\rho_n\}_{n\in \mathbb{N}}\subset UC_b(E)\subset \mathcal{S}(E)$ as in the second part of the proof of 
Lemma \ref{lem:0000-00-00:00}. This is  a sequence of linearly independent functions. Setting  
$$
Z_K\coloneqq \left\{f\in UC_b(E)\colon f(x)=0,\ \forall  x\in K\right\},
$$
we have $\rho_n\in Z_K$ for every $n\geq n_K$ (where $n_K$ is as in the proof of 
Lemma \ref{lem:0000-00-00:00}).
This shows that the subspace $Z_K\subset \mathcal{S}(E)$ 
is infinite dimensional. 
For $i=1,\ldots,m$, define the functionals
$$
\Lambda_i\colon Z_K\rightarrow \mathbb{R},\ \ \varphi\mapsto \int_E \varphi d\mu_i.
$$
Since $Z_K$ is infinite dimensional,  $\mathcal{N}\coloneqq \bigcap_{i=1}^m\ker \Lambda_i$ is infinite dimensional too. On the other hand, 
  by construction, $\mathcal{N}\subset \mathcal{I}$.
This concludes the proof.
\end{proof}

\subsubsection{Characterization   of $(\mathcal{S}(E),\tau_\mathcal{K})^*$}
The aim of this subsection is to provide a characterizion of $(\mathcal{S}(E),\tau_\mathcal{K})^*$, for the cases $\mathcal{S}(E)=B_b(E)$ and $\mathcal{S}(E)=C_b(E)$.
\label{sec:2015-05-04:06}
Denote by $\mathbf{ba}_\mathcal{C}(E)$ the subspace of $\mathbf{ba}(E)$ defined by
$$
\mathbf{ba}_\mathcal{C}(E)\coloneqq \{\mu\in \mathbf{ba}(E)\colon  \exists\  K\subset E\ \mathrm{compact}\colon |\mu|(K^c)=0\}.
$$
If $E$ is compact, we clearly have $\mathbf{ba}_\mathcal{C}(E)=\mathbf{ba}(E)$.
Conversely, if  $E$ is not compact, then $\mathbf{ba}_\mathcal{C}(E)$ is a non-closed
subspace of $\mathbf{ba}(E)$. Indeed, if the sequence $\{x_n\}_{n\in \mathbb{N}}$ in $E$ does not admit any convergent subsequence, then 
$$\mu_n=\sum_{k=1}^n 2^{-k}\delta_{x_k}\in \mathbf{ba}_\mathcal{C}(E), \ \forall  n\in\N, \mbox{ and }\lim_{n\rightarrow +\infty}\mu_n=\sum_{k=1}^{+\infty} 2^{-k}\delta_{x_k}\in \mathbf{ca}(E)\setminus  \mathbf{ba}_\mathcal{C}(E).$$
Denote by $C_b(E)^\perp$ the annihilator of $C_b(E)$ in $(B_b(E),|\cdot|_\infty)^*\cong (\mathbf{ba}(E),|\cdot|_1)$ (see \eqref{eqq1},\eqref{eqq2}), that is
$$
C_b(E)^\perp\coloneqq \left\{\mu\in \mathbf{ba}(E)\colon \ \int_E {f} d \mu=0,\  \ \forall   {f}\in C_b(E)\right\}.
$$
By
Lemma \ref{rem:2015-04-27:02},
we have  $C_b(E)^\perp\backslash\{0\}\subset \mathbf{ba}(E)\setminus\textbf{ca}(E)$.

\begin{Proposition}\label{prop:2015-04-27:00}
The following statements hold.
  \begin{enumerate}[(i)]
  \item \label{2015-12-06:06}
$(B_b(E),\tau_\mathcal{K})^*= \left( \mathbf{ba}_\mathcal{C}(E)\medcap C_b(E)^\perp \right)  \oplus \mathbf{ca}(E)$.
More explicitly,
for each $\Lambda\in (B_b(E),\tau_\mathcal{K})^*$ there exist  unique $\mu\in \mathbf{ba}_\mathcal{C}(E)\medcap C_b(E)^\perp$ and  $\nu\in \mathbf{ca}(E)$ such that 
$$
\Lambda({f})=\int_E {f} d(\mu+\nu)\qquad \forall   {f}\in B_b(E),
$$
where the integral is in the Darboux sense.
\item\label{2015-10-10:00} 
$(C_b(E),\tau_\mathcal{K})^*=\mathbf{ca}(E)$.
More explicitly,
 for each $\Lambda\in (C_b(E),\tau_\mathcal{K})^*$ there exists a unique $\nu\in \mathbf{ca}(E)$ such that 
\[
\Lambda({f})=\int_E{f} d\nu\qquad \forall  f\in C_b(E).
\]
\end{enumerate}
\end{Proposition}
\begin{proof}
\emph{(\ref{2015-12-06:06})}  Let $\Lambda\in (B_b(E),\tau_\mathcal{K})^*$. Then there exist $L>0$, a compact set $K\subset E$, a natural number $N$, and measures $\mu_1,\ldots, \mu_{N}\in \mathbf{ca}(E)$, such that
$$
|\Lambda (f)|\leq L\left([ {f}]_K+\sum_{n=1}^N\left|\int _E{f} d\mu_n\right|\right)\qquad\forall  {f}\in B_b(E).
$$
Define 
$\Lambda_{K^c}(f)= \Lambda (f\mathbf{1}_{K^c})$, for $f\in B_b(E)$. Then
\begin{equation}
  \label{eq:2015-12-06:07}
  |\Lambda_{K^c}({f})|\leq L
\sum_{n=1}^N\left|\int _E{f} d(\mu_n\lfloor K^{c})\right|,\qquad\forall  {f}\in B_b(E),
\end{equation}
where $\mu_n\lfloor K^c$ denotes the restriction of $\mu_n$ to $K^c$.
Hence $\Lambda_{K^c}\in (B_b(E),\tau_\mathcal{K})^*$.
Moreover, by \cite[Lemma\ 3.9, p.\ 63]{Rudin1991}, 
\eqref{eq:2015-12-06:07}
implies that
there exists
 $\nu\in \mathrm{Span} \{\mu_i\lfloor K^c\colon i=1,\ldots,N\}\subset \mathbf{ca}(E)$ such that
$$
\Lambda_{K^c}({f})=\int_E {f} d\nu\qquad \forall  f\in B_b(E).
$$
Define $\Lambda_{K}(f)\coloneqq\Lambda (f\mathbf{1}_{K})$, for $f\in B_b(E)$.
Since $\Lambda_K=\Lambda-\Lambda_{K^c}$,
$\Lambda_K\in (B_b(E),\tau_\mathcal{K})^*$.
By the identification  $(B_b(E),|\cdot|_\infty)^*\cong (\mathbf{ba}(E),|\cdot|_1)$ (see \eqref{eqq1}--\eqref{eqq2}),
there exists a unique $\mu\in \mathbf{ba}(E)$ such that 
\[
\Lambda_K({f})=\int_E {f} d\mu\qquad  \forall  f\in B_b(E),
\]
 where the integral above is defined in the Darboux sense. We notice that $\mu(A)=0$ for every Borel set $A\subset K^c$. Hence
$\mu\in \mathbf{ba}_\mathcal{C}(E)$, and the
existence part of the claim is proved.

As regarding uniqueness, let $\mu_1+\nu_1$ and $\mu_2+\nu_2$ be two decompositions as in the statement. Then $\nu_1-\nu_2\in \mathbf{ca}(E)\medcap C_b(E)^\perp$.
 Therefore,
by Lemma \ref{rem:2015-04-27:02},
 $\nu_1-\nu_2=0$, and then $\mu_1=\mu_2$.

\emph{(\ref{2015-10-10:00})}
Let $\Lambda\in  (C_b(E),\tau_\mathcal{K})^*$. Since $\tau_\mathcal{K}$ is locally convex, by the Hahn-Banach Theorem we can extend $\Lambda$ to some  $\overline \Lambda\in (B_b(E),\tau_\mathcal{K})^*$. Let $\mu+\nu$ be the decomposition of  $\overline \Lambda$ provided by \emph{(\ref{2015-12-06:06})}, with $\mu\in \mathbf{ba}_\mathcal{C}(E)\medcap C_b(E)^\perp$ and $\nu\in \mathbf{ca}(E)$. Then
$$
\Lambda({f})
=
\overline \Lambda({f})
=\int_E {f} d(\mu+\nu)
=\int_E {f} d\nu\qquad  \forall  {f}\in C_b(E).
$$
Uniqueness is provided by Lemma \ref{rem:2015-04-27:02}.
\end{proof}

\begin{Remark}
\label{rem:2015-04-30:00}
In general, 
 the dual of $(C_b(E),\tau_\infty)$ cannot be identified with $\mathbf{ca}(E)$
through the integral,  that is, the isometric embedding \eqref{eq:2015-08-07:00} is not onto(\footnote{For a characterization of $(C_b(E),\tau_\infty)^*$, see
\cite[Sec. 14.2]{Aliprantis2006}.}).
An example where $(C_b(E),\tau_\infty)^*\neq \mathbf{ca}(E)$ is provided by the case $E=\mathbb{N}$. 
Then 
$C_b(\mathbb{N})=\ell^\infty$ and $(C_b(\mathbb{N}),\tau_\infty)^*=(\ell^\infty)^*\supsetneq \ell^1\cong\mathbf{ca}(\mathbb{N})$ (where the symbol ``$\cong$'' 
is consistent with the action of $\ell^1$ and of $\mathbf{ca}(\mathbb{N})$ on $\ell^\infty$).
In view of this observation, 
Proposition \ref{prop:2015-04-27:00}{(\ref{2015-10-10:00})} cannot be seen, in its generality, as a straightforward consequence of the inclusions
$
\sigma(C_b(E),\mathbf{ca}(E))\subset \tau_\mathcal{K}\subset \tau_\infty.
$
\end{Remark}

\subsection{Relationship with  weakly continuous semigroups}\label{sec:wc} 

In this subsection we first recall the notions of
 $\mathcal{K}$-convergence and of weakly continuous semigroup in the space $UC_b(E)$, introduced and studied first in \cite{Cerrai94, Cerrai01book} in the case of  $E$  separable Banach space
(\footnote{In order to avoid misunderstanding, we stress that \cite{Cerrai01book} uses the notation $C_b(E)$ to denote the space of \emph{uniformly} continuous bounded functions on $E$, i.e., our space $UC_b(E)$. Also we notice that the separability of $E$ is not needed here for our discussion.}).
 So, throughout this subsection $E$ is assumed to be a Banach space. We will show that every weakly continuous semigroup is a $C_0$-sequentially locally equicontinuous semigroup and, up to  a  renormalization,  a $C_0$-sequentially equicontinuous semigroup on $(UC_b(E),\tau_\mathcal{K})$  (Proposition \ref{2015-10-13:06}).

 The notion of $\mathcal{K}$-convergence was introduced in \cite{Cerrai94, Cerrai01book} for sequences. We recall it in its natural extension to nets.
A net of functions $\{f_\iota\}_{\iota\in \mathcal{I}}\subset UC_b(E)$  is said $\mathcal{K}$-convergent to  $f\in UC_b(E)$ if it is $|\cdot|_\infty$-bounded  and  if $\{f_\iota\}_{\iota\in \mathcal{I}}$ converges to $f$ uniformly on compact sets of $E$, that is
\begin{equation}
  \label{eq:2015-10-30:00}
  \begin{dcases}
      \sup_{\iota\in \mathcal{I}}|f_\iota|_\infty<+\infty\\
      \lim_{\iota}[f_\iota-f]_K=0 & \mbox{for every non-empty compact }K\subset E.
    \end{dcases}
  \end{equation}
In such a case, we write $f_\iota\stackrel{\mathcal{K}}{\longrightarrow} f$. 
 If $E$ is  separable,
in view of  Proposition \ref{prop:2015-04-24:01}\emph{(\ref{2015-10-13:10})}, the convergence
\eqref{eq:2015-10-30:00} is equivalent to the convergence with respect to the locally convex topology $\tau_\mathcal{K}$.
In this sense, $\tau_\mathcal{K}$ is the natural vector topology to treat weakly continuous semigroups (whose definition is recalled below) within the framework of $C_0$-sequentially locally equicontinuous semigroups. 

\begin{Definition}
A \emph{weakly continuous semigroup} on  
$UC_b(E)$
 is  a  family $T=\{T_t\}_{t\in \mathbb{R}^+}$ of bounded linear operators on $(UC_b(E),|\cdot|_\infty)$ satisfying the following conditions.
\begin{enumerate}[(P1)]
\item\label{2015-10-29:15}    $T_0=I$ and $T_t  T_s= T_{t+s}$ for $t,s\in \mathbb{R}^+$.
\item\label{2015-10-13:09} There exist $M\geq 1$ and $\alpha\in\R$ such that $|T_tf|_\infty\leq M e^{\alpha t}|f|_\infty$ for every $t\in \mathbb{R}^+$, $f\in UC_b(E)$.
\item\label{2015-10-29:14} For every $f\in UC_b(E)$ and every $\hat t>0$, the family of functions $\{T_tf\colon E\rightarrow \R\}_{t\in [0,\hat{t}]} $ is equi-uniformly continuous, that is, there exists a modulus of continuity $w$ (depending on $\hat t$) such that
  \begin{equation}
    \label{eq:2015-10-13:15}
    \sup_{t\in[0,\hat t]}|T_tf(\xi)-T_tf(\xi')|\leq w(|\xi-\xi'|_E),\qquad \forall  \xi,\xi'\in E.
  \end{equation}
\item\label{2015-10-13:12}
 For every  $f\in UC_b(E)$, we have $T_tf \stackrel{\mathcal{K}}{\longrightarrow}f$ as $t\rightarrow 0^+$; in view of {(P\ref{2015-10-13:09})} the latter convergence is equivalent to
 \begin{equation}
   \label{eq:2015-10-13:08}
\lim_{t\rightarrow 0^+} [T_tf-f]_K=0\ \mbox{for every non-empty compact }K\subset E.
 \end{equation}
\item\label{2015-10-13:13}
If $f_n\stackrel{\mathcal{K}}{\longrightarrow}f$, then $T_t f_n \stackrel{\mathcal{K}}{\longrightarrow}T_tf$ uniformly in $t\in [0,\hat t]$ for every $\hat t >0$; in view of 
(P\ref{2015-10-13:09}), the latter convergence is  equivalent to 
\begin{equation}
  \label{eq:2015-10-13:07}
  \lim_{n\rightarrow +\infty}\sup_{t\in[0,\hat t]}[T_tf_n-T_tf]_K=0\ \mbox{for every non-empty compact }K\subset E, \ \forall  \hat t\in \mathbb{R}^+.
\end{equation}
\end{enumerate} 
\end{Definition}
\begin{Proposition}\label{2015-10-13:06}
Let $T\coloneqq \{T_t\}_{t\in \mathbb{R}^+}$ be  a
  weakly continuous semigroup on $UC_b(E)$.
Then $T$
 is a $C_0$-sequentially locally equicontinuous semigroup on $ ( UC_b(E),\tau_\mathcal{K} )$ and, for every $\lambda>\alpha$ (where $\alpha$ is as in (P\ref{2015-10-13:09})),  $\{e^{-\lambda t}T_t\}_{t\in \mathbb{R}^+}$ is a 
$C_0$-sequentially  equicontinuous semigroup on $ ( UC_b(E),\tau_\mathcal{K} )$ satisfying 
Assumption \ref{ass:int}.

Conversely, if $T$ is a $C_0$-sequentially locally equicontinuous semigroup on $(UC_b(E),\tau_\mathcal{K})$ satisfying (P\ref{2015-10-29:14}), then $T$ is a weakly continuous semigroup on $UC_b(E)$.
\end{Proposition}
\begin{proof}
Let $f\in UC_b(E)$. By 
\emph{(P\ref{2015-10-13:12})} and 
 by 
Proposition \ref{prop:2015-04-24:01}\emph{(\ref{2015-10-13:10}), $T_{t}f\rightarrow f$ in $(UC_b(E),\tau_\mathcal{K})$ when $t\rightarrow 0^+$}. This shows the strong continuity of $T$ in $(UC_b(E),\tau_\mathcal{K})$.

Now let $\{f_n\}_{n\in \mathbb{N}}$ be a sequence converging to $0$ in $(UC_b(E),\tau_\mathcal{K})$ and let $\hat t\in \mathbb{R}^+$.
 By 
Proposition \ref{prop:2015-04-24:01}\emph{(\ref{2015-10-13:10})},
 it follows that $f_n\stackrel{\mathcal{K}}{\longrightarrow}0$.
By
(P\ref{2015-10-13:13})
 we then have
$T_{t}f_n\stackrel{\mathcal{K}}{\longrightarrow}0$ uniformly in $t\in[0,\hat t]$.
Using again  Proposition \ref{prop:2015-04-24:01}\emph{(\ref{2015-10-13:10})}, we conclude that 
 $T$ is locally sequentially equicontinuous in $(UC_b(E),\tau_\mathcal{K})$.

By (P\ref{2015-10-13:09}) and by 
Proposition \ref{prop:2015-04-29:00}, we can apply
 Proposition \ref{prop:2015-08-01:00}\emph{(\ref{2015-10-13:14})} to $T$ conclude that $\{e^{-\lambda t}T_t\}_{t\mathbb{R}^+}$ is a $C_0$-sequentially equicontinuous semigroup on $(UC_b(E),\tau_\mathcal{K})$.

We finally show that, for $\lambda>\alpha$, $\{e^{-\lambda t}T_t\}_{t\in \mathbb{R}^+}$ satisfies Assumption \ref{ass:int}. Let $\alpha< \lambda'<\lambda$ and $f\in UC_b(E)$.
By Proposition \ref{prop:seqcom},
$(C_b(E),\tau_\mathcal{K})$ is sequentially complete.
 By 
Proposition \ref{prop:2015-05-01:02a}, the map
$$
\mathbb{R}^+\rightarrow (UC_b(E),\tau_\mathcal{K}),\ t \mapsto 
e^{-\lambda' t}T_tf,
$$
is continuous and bounded. It then follows that the Riemann integral $R(\lambda)f$ exists in $C_b(E)$. We show that $R(\lambda)f\in UC_b(E)$.
Since the Dirac measures are contained in $(C_b(E),\tau_\mathcal{K})^*$, by 
Proposition \ref{prop:2015-05-01:07} we have
$$
R(\lambda)f(\xi)=\int_0^{+\infty}e^{-\lambda t}T_tf(\xi)dt \qquad\forall  \xi\in E.
$$
On the other hand, by 
(P\ref{2015-10-13:09}), for every $\epsilon>0$ there exists $\hat t\in \mathbb{R}^+$ such that
$$
\sup_{\xi\in E}\int_{\hat t}^{+\infty}e^{-\lambda t}T_tf(\xi)dt<\epsilon.
$$
Hence, to prove that $R(\lambda)f\in UC_b(E)$, it suffices to show that, for every $\hat t\in \mathbb{R}^+$, 
\begin{equation}
  \label{eq:2015-10-13:16}
  \int_0^{\hat t}e^{-\lambda t}T_tfdt\in UC_b(E).
\end{equation}
Let us define the set
$$
C\coloneqq  \left\{ g\in C_b(E)\colon \sup_{\xi,\xi'\in E}|g(\xi)-g(\xi')|\leq w(|\xi-\xi'|_E)  \right\},
$$
where $w$ 
is as in \eqref{eq:2015-10-13:15}.
Clearly $C$ is a subset of $UC_b(E)$, it is convex, it contains the origin,
and
 is 
 closed in $(C_b(E),\tau_\mathcal{K})$.
By
\eqref{eq:2015-10-13:15},
 $\{e^{-\lambda't}T_tf\}_{t\in [0,\hat t]}\subset C$. Hence,  we conclude by
Proposition \ref{prop:2015-07-30:00}  that
$$
\int_0^{t_0}e^{-\lambda t}T_t f dt\in \frac{1}{\lambda-\lambda'}C\quad  \forall  \lambda>\lambda',
$$
which shows 
\eqref{eq:2015-10-13:16}, concluding the proof of the first part of the proposition.

Now let $T$ be a $C_0$-sequentially locally equicontinuous on $(UC_b(E),\tau_\mathcal{K})$ satisfying 
 (P\ref{2015-10-29:14}). 
We only need to show that $T$ verifies 
(P\ref{2015-10-13:09}),
(P\ref{2015-10-13:12}), and (P\ref{2015-10-13:13}).
Now, (P\ref{2015-10-13:09}) follows from 
Proposition \ref{prop:2015-04-29:00} and 
Proposition \ref{prop:2015-08-01:00}, whereas (P\ref{2015-10-13:12}) comes once again by Proposition \ref{prop:2015-04-24:01}\emph{(\ref{2015-10-13:10})}.
Finally, (P\ref{2015-10-13:13}) is due to Proposition \ref{prop:2015-04-24:01}\emph{(\ref{2015-10-13:10})} and to sequential local equicontinuity of $T$.
\end{proof}

\subsection{Relationship with  $\pi$-semigroups}\label{2015-10-13:03}

In this subsection we provide a connection between  the notion of  $\pi$-semigroups in $UC_b(E)$
introduced in \cite{Priola99}
 and bounded $C_0$-sequentially continuous semigroups  (see Definition \ref{2015-10-13:01}) in the space $(UC_b(E),\sigma(UC_b(E),\mathbf{ca}(E)))$ (\footnote{Also in this case, in order to avoid misunderstanding, we stress that \cite{Priola99} uses the notation $C_b(E)$ to denote the space of \emph{uniformly} continuous bounded functions on $E$, i.e. our space $UC_b(E)$. We also notice that in \cite{Priola99} the metric space $E$ is assumed to be separable, but, for our discussion, this is not needed.}).
We recall that the assumption $E$ Banach space was standing only in the latter subsection, and that in the present subsection we restore the assumption that $E$ is a generic metric space.
We start by recalling the definition of $\pi$-semigroup in $UC_b(E)$.
\begin{Definition}\label{2015-10-30:01bis}
A \emph{$\pi$-semigroup  on 
$UC_b(E)$}
 is  a  family $T=\{T_t\}_{t\in \mathbb{R}^+}$ of bounded linear operators on $(UC_b(E),|\cdot|_\infty)$ satisfying the following conditions.
\begin{enumerate}[(P1)]
\item\label{2015-10-29:15bis}    $T_0=I$ and $T_t  T_s= T_{t+s}$ for $t,s\in \mathbb{R}^+$.
\item\label{2015-10-13:09bis} There exist $M\geq 1$, $\alpha\in\R$ such that $|T_tf|_\infty\leq M e^{\alpha t}|f|_\infty$ for every $t\in \mathbb{R}^+$, $f\in UC_b(E)$.
  \item\label{2015-10-30:04bis} For each $\xi\in E$ and $f\in UC_b(E)$, the map $\mathbb{R}^+\rightarrow \mathbb{R}, \ t \mapsto T_tf(\xi)$ is continuous.
  \item\label{2015-10-30:06bis} If a sequence $\{f_n\}_{n\in \mathbb{N}}\subset UC_b(E)$ is such that
$$
\sup_{n\in \mathbb{N}}|f_n|_\infty<+\infty\qquad \mathrm{and}\qquad \lim_{n\rightarrow +\infty}f_n=f\qquad\mathrm{pointwise},
$$
then, for every $t\in \mathbb{R}^+$,
$$
 \lim_{n\rightarrow +\infty}T_tf_n=T_tf\ \ \ \mathrm{pointwise}.
$$
  \end{enumerate}
\end{Definition}

\begin{Proposition}\label{2015-10-30:08}
  $T$ is a $\pi$-semigroup in $UC_b(E)$ if and only if $\{e^{-\alpha t}T_t\}_{t\in \mathbb{R}^+}$
is a bounded $C_0$-sequentially continuous semigroup
in 
$(UC_b(E),\sigma(UC_b(E),\mathbf{ca}(E)))$
(see Definition \ref{2015-10-13:01}).
\end{Proposition}
\begin{proof}
Let us  denote $\sigma\coloneqq \sigma(UC_b(E),\mathbf{ca}(E))$.
Let $T$ be a $\pi$-semigroup in $UC_b(E)$.
By Definition \ref{2015-10-30:01bis}\emph{(P\ref{2015-10-13:09bis})},\emph{(P\ref{2015-10-30:06bis})} and
Proposition \ref{prop:2015-04-24:01BB}\emph{(\ref{2015-10-30:05})}, we have 
 $\{e^{-\alpha t}T_t\}_{t\in \mathbb{R}^+}\subset \mathcal{L}_0((UC_b(E),\sigma))$.
By  Definition \ref{2015-10-30:01bis}\emph{(P\ref{2015-10-13:09bis})},\emph{(P\ref{2015-10-30:04bis})} and  by 
Proposition \ref{prop:2015-04-24:01BB}\emph{(\ref{2015-10-30:05})}, the map $\mathbb{R}^+\rightarrow (UC_b(E),\sigma), \ t \mapsto  e^{-\alpha t}T_tf$ is continuous for every $f\in UC_b(E)$.  Moreover, by 
Definition \ref{2015-10-30:01bis}\emph{(P\ref{2015-10-13:09bis})} and by
Proposition
\ref{prop:2015-04-29:00}, it is also bounded.
This shows that $\{e^{-\alpha t}T_t\}_{t\in \mathbb{R}^+}$ is a 
bounded $C_0$-sequentially continuous semigroup
in 
$(UC_b(E),\sigma)$.

Conversely, let
$\{e^{-\alpha t}T_t\}_{t\in \mathbb{R}^+}$ be a 
bounded $C_0$-sequentially continuous semigroup
in 
$(UC_b(E),\sigma)$.
By  Proposition
\ref{prop:2015-04-29:00},
for every $f\in UC_b(E)$ the family $\{e^{-\alpha t}T_tf\}_{t\in \mathbb{R}^+}$ is bounded in $(UC_b(E), |\cdot|_\infty)$. By the Banach-Steinhaus theorem we conclude that there exists $M>0$ such that
$$
|e^{-\alpha t}T_t|_{L((UC_b(E),|\cdot|_\infty))}\leq M \qquad  \forall  t\in \mathbb{R}^+,
$$
which provides $T\subset L((UC_b(E),|\cdot|_\infty))$ and
\emph{(P\ref{2015-10-13:09bis})}. Then,
\emph{(P\ref{2015-10-30:04bis})} is implied by the fact that the map  $\mathbb{R}^+\rightarrow (UC_b(E),\sigma), \ t \mapsto e^{-\alpha t}T_tf$, is continuous and that Dirac measures are contained in $\sigma$.
Finally,
\emph{(P\ref{2015-10-30:06bis})} is due to the assumption $\{e^{-\alpha t}T_t\}_{t\in \mathbb{R}^+}\subset \mathcal{L}_0((UC_b(E),\sigma))$ and to 
Proposition \ref{prop:2015-04-24:01BB}\emph{(\ref{2015-10-30:05})}.
\end{proof}

As observed in Subsection \ref{2015-10-30:07},
if the Laplace transform \eqref{Lap} of a bounded $C_0$-sequentially continuous semigroup in $(UC_b(E),\sigma(UC_b(E),\mathbf{ca}(E)))$ is well-defined,
several results that we stated for $C_0$-sequentially equicontinuous semigroups still hold.
 Nevertheless, 
some other important results, as the generation theorem, or the fact that two semigroups with the same generator are equal, cannot be proved
for bounded $C_0$-sequentially continuous semigroups
 within the approach of the previous sections.
Due to
Proposition \ref{2015-10-30:08}, 
 this is reflected in the fact that, as far as we know, such results are not available in the literature for $\pi$-semigroups.

\subsection{Relationship with locally equicontinuous semigroups  with respect to the \emph{mixed} topology}\label{sub:goldys}

 When $E$ is a separable Hilbert space, in \cite{GoldysKocan01} the so called \emph{mixed topology} (introduced in \cite{Wiweger61}) is employed in the space $C_b(E)$ to frame a class of Markov transition semigroups within the theory of $C_0$-locally equicontinuous semigroups.
The same topology, but in the more general case of $E$ separable Banach space, is used in \cite{Goldys2003} to deal with Markov transition semigroups associated to the Ornstein-Uhlenbeck processe in Banach spaces.

 In this subsection, we assume that $E$ is a separable Banach space and we briefly precise
 what is the relation between the mixed topology and $\tau_\mathcal{K}$ in the space $C_b(E)$, and
between $C_0$-locally equicontinuous semigroups with respect to the mixed topology and
$C_0$-sequentially locally equicontinuous semigroups with respect to $\tau_\mathcal{K}$.



The mixed topology on $C_b(E)$, denoted by $\tau_\mathcal{M}$,
 can be defined by seminorms as follows.
Let $\mathbf{K}\coloneqq \{K_n\}_{n\in \mathbb{N}}$ be a sequence of compact subsets of $E$, and let $\mathbf{a}\coloneqq \{a_n\}_{n\in \mathbb{N}}$ be a sequence of strictly positive real numbers such that $a_n\rightarrow 0$. Define
\begin{equation}\label{goldyssemi}
p_{\mathbf{K},\mathbf{a}}({f})=\sup_{n\in \mathbb{N}}\{a_n[f]_{K_n}\}\qquad  \forall   f\in C_b(E).
\end{equation}
Then $p_{\mathbf{K},\mathbf{a}}$ is a seminorm and   $\tau_\mathcal{M}$ is 
defined as the locally convex  topology   induced by the family of seminorms $p_{\mathbf{K},\mathbf{a}}$, when $\mathbf{K}$ ranges on the set of countable families of compact subsets of $E$, and $\mathbf{a}$ ranges on the set of sequences of strictly positive real numbers converging to $0$.

It can be proved (see \cite[Theorem 2.4]{Sentilles1972}), that   $\tau_\mathcal{M}$ is the finest locally convex topology on $C_b(E)$ such that a net $\{f_\iota\}_{\iota\in \mathcal{I}}$ is bounded in the uniform norm and converges to $f$ in $\tau_\mathcal{M}$ if and only if it is $\mathcal{K}$-convergent, that is, if and only if \eqref{eq:2015-10-30:00} is verified. 



To establish the relation between $\tau_\mathcal{M}$ and $\tau_\mathcal{K}$, we start with a lemma.

\begin{Lemma}
\label{2015-12-07:04}
  Let $S\subset E$  be a Borel set and assume that $S$ is a retract of $E$, that is,  there exists a continuous map $r\colon E\rightarrow S$ such that $r(s)=s$ for every $s\in S$.
We denote by  $\tau_\mathcal{K}^S$
  the topology $\tau_\mathcal{K}$ when considered in the spaces ${C}_b(S)$.
Then 
\begin{equation}\label{Psi}
\Psi\colon C_b(E)\rightarrow  C_b (S),\ \ f \mapsto f_{|S}
\end{equation}
is continuous and open as a map from 
$(C_b(E),\tau_\mathcal{K})$ onto $ (C_b (S), \tau_\mathcal{K}^S)$.
\end{Lemma}
\begin{proof}
First we show that $\Psi$ is continuous. Let 
 $\{f_\iota\}_{\iota\in \mathcal{I}}\subset C_b(E)$ be a net converging to $0$ in $\tau_\mathcal{K}$, let $K\subset S$ be  compact, and let $\mu\in \mathbf{ca}(S)$.
Since $K$ is also compact in $E$, we immediately have $[f_{\iota|S}]_K\rightarrow 0$.
Moreover, since $S$ is Borel, the set function $\mu^S$ defined by $\mu^S(A)\coloneqq \mu(A\cap S)$, $A\in \mathcal{E}$, belongs to $\mathbf{ca}(E)$. Then we also have
$\int_S f_{\iota|S}d\mu=\int_E f_\iota d\mu^S\rightarrow 0$. So  $\Psi$ is continuous. 

Let us prove that $\Psi$ is open. Let $K\subset E$ be compact, $\mu_1,\ldots,\mu_n\in \mathbf{ca}(E)$, $\epsilon>0$.
Define the neighborhood of $0$ in $(C_b(E),\tau_\mathcal{K}^E)$
$$
U\coloneqq \left\{f\in C_b(E)\colon [f]_K<\epsilon,\ \left|\int_E fd\mu_i\right|<\epsilon,\ i=1,\ldots,n  \right\}
$$
and define the the neighborhood of $0$ in $(C_b(S),\tau_\mathcal{K}^S)$
$$
V\coloneqq \left\{g\in C_b(S)\colon [g]_{r(K)}<\epsilon,\ \left|\int_E gd(r_{\#}\mu_i)\right|<\epsilon,\ i=1,\ldots,n  \right\}
$$
where $r_{\#}\mu_i$ is the pushforward measure of $\mu_i$ through $r$.
Then  $g\in V$ if and only if $f\coloneqq g\circ r\in U$.
As  $g=(g\circ r)_{|S}$ for every $g\in C_b(S)$, we see that  $V\subset \Psi(U)$.
Hence,  we conclude that $\Psi$ is open.
\end{proof}

\begin{Proposition}
If $\dim E\geq 1$, then
$\tau_\mathcal{K}\subsetneq\tau_\mathcal{M}$ on $C_b(E)$.
\end{Proposition}
\begin{proof}
We already observed that $\tau_\mathcal{M}$ is 
the finest locally convex topology $\tau_\mathcal{M}$  such that $\{f_\iota\}_{\iota\in \mathcal{I}}$ is bounded in the uniform norm and converges to $f$ in $\tau_\mathcal{M}$ if and only if it is $\mathcal{K}$-convergent.
Then, by Proposition \ref{prop:2015-04-24:01}\emph{(\ref{2015-10-13:10})}, we have $\tau_\mathcal{K}\subset \mathcal{\tau}_\mathcal{M}$.

Now we show that
$ \tau_\mathcal{M}\not\subset \tau_\mathcal{K}$
 if $\dim(E)\geq 1$.
Let $S$ be a one dimensional subspace of $E$ and let 
$$
\Psi\colon C_b(E)\rightarrow C_b(S),\ f \mapsto f_{|S}.
$$
By using the seminorms defined in \eqref{goldyssemi}, one checks that $\Psi$, defined in \eqref{Psi}, is continuous from $(C_b(E),\tau_\mathcal{M})$ onto $(C_b(S),\tau^S_\mathcal{M})$,  where $\tau_\mathcal{M}^S$  denotes the topology $\tau_\mathcal{M}$ in the space ${C}_b(S)$. 
Clearly $S$ is a retract of $E$. Then, by   Lemma \ref{2015-12-07:04}, to show that $ \tau_\mathcal{M}\not\subset \tau_\mathcal{K}$ on $C_b(E)$, it is sufficient to show that 
$ \tau_\mathcal{M}^S\not\subset \tau_\mathcal{K}^S$
on $C_b(S)$. Let us identify $S$ with $\R$.
Let  $W$ be a  Wiener process in $\mathbb{R}$ on some probability space
$(\Omega,\mathcal{F}, \mathbb{P})$.
By
\cite[Theorem 4.1]{GoldysKocan01},
 the transition semigroup $T\coloneqq \{T_t\}_{t\in \mathbb{R}^+}$ defined by
$$
T_t\colon C_b(\R)\rightarrow C_b(\R),\ f \mapsto \mathbb{E} \left[ f(\cdot+ W_t) \right], 
$$
is a $C_0$-locally equicontinuous semigroup in $(C_b(\R),\tau^\R_\mathcal{M})$.
But Example \ref{2015-05-04:07-II} below
shows that $T$ is not locally 
equicontinuous in
$(C_b(\R),\tau_\mathcal{K}^\R)$. 
Then $\tau^\R_\mathcal{M}\not\subset \tau^\R_\mathcal{K}$. Since we already know that $\tau^\R_\mathcal{K}\subset \tau^\R_\mathcal{M}$, we deduce that $\tau^\R_\mathcal{M}\not\subset \tau^\R_\mathcal{K}$ and conclude.
\end{proof}


By
Proposition \ref{prop:2015-04-24:01}\emph{(\ref{2015-10-13:10})}, every sequence convergent in $\tau_\mathcal{K}$ is bounded and convergent uniformly on compact sets, and then it is convergent in $\tau_\mathcal{M}$.
Since we also know $\tau_\mathcal{K}\subset \tau_\mathcal{M}$, we immediately obtain the following

\begin{Proposition}
  A semigroup $T$ is $C_0$-sequentially (locally) equicontinuous
 in $(C_b(E),\tau_\mathcal{M})$ if and only if it is $C_0$-sequentially
(locally)
 equicontinuous
 in 
$(C_b(E),\tau_\mathcal{K})$.
\end{Proposition}

\section{Application to  transition semigroups}
\label{sec:trans}

 In this section we apply the results of Section \ref{sec:SE} to  transition semigroups in spaces of (not necessarily bounded) continuous functions.

\subsection{Transition semigroups in $(C_b(E), \tau_\mathcal{K})$}\label{sub2B}
 Let  $\boldsymbol{\mu}\coloneqq\{\mu_t(\xi,\cdot)\}_{\substack{t\in\R^+\\\xi\in E}}$ be a subset of $\mathbf{ca}^+(E)$ and consider the following assumptions.
\begin{Assumption}\label{2015-10-28:30}
 The family  $\boldsymbol{\mu}\coloneqq\{\mu_t(\xi,\cdot)\}_{\substack{t\in\R^+\\\xi\in E}} \subset \mathbf{ca}^+(E)$ has 
the following properties.
  \begin{enumerate}[(i)]
  \item \label{2015-11-10:00} The family $\boldsymbol{\mu}$ is bounded in $\mathbf{ca}^+(E)$ and $p_0(\xi,\Gamma)=\mathbf{1}_\Gamma(\xi)$
 for every $\xi\in E$ and every  $\Gamma\in\mathcal{E}$. 
  \item \label{2015-11-09:02}
For every $f\in C_b(E)$ and $t\in \mathbb{R}^+$, the map
\begin{equation}
  \label{eq:2015-11-10:01}
  E\rightarrow \mathbb{R},\ \xi \mapsto \int_Ef(\xi')\mu_t(\xi,d\xi')
\end{equation}
is continuous.
  \item\label{2015-11-10:02}
For every $f\in C_b(E)$, every $t,s\in \mathbb{R}^+$, and every $\xi\in E$,
$$
\int_Ef(\xi')\mu_{t+s}(\xi,d\xi')=
\int_E \left( \int_Ef(\xi'')\mu_t(\xi',d\xi'') \right) 
\mu_s(\xi,d\xi').
$$
\item\label{2015-10-28:32} For every $\hat t>0$ and every compact $K\subset E$, the family 
$
\{\mu_t(\xi,\cdot)\colon t\in [0,\hat t],\ \xi\in K\}
$
 is tight, that is, for every $\epsilon>0$, there exists a compact set $K_0\subset E$ such that
$$
\mu_t(\xi,K_0)>\mu_t(\xi,E)-\epsilon\qquad  \forall  t\in [0,\hat t],\ \forall \xi\in K.
$$
\item\label{2015-10-28:31} For every $ r >0$ and every non-empty compact $K\subset E$,
    \begin{equation}
      \label{eq:2015-11-10:04}
      \lim_{t\rightarrow 0^+}\sup_{\xi\in K}
      |\mu_t(\xi,B(\xi, r )) -1|=0,
    \end{equation}
  where $B(\xi, r )$ denotes the open ball $B(\xi, r )\coloneqq \{\xi'\in E\colon d(\xi,\xi')< r \}$.
  \end{enumerate}
\end{Assumption}

We observe that in Assumption \ref{2015-10-28:30} it is not required that $p_t(\xi,E)=1$ for every $t\in\R^+$, $\xi\in E$, that is the family  $\boldsymbol{\mu}$ is not necessarily a probability kernel in $(E,\mathcal{E})$.
Assumptions \ref{2015-10-28:30}\emph{(\ref{2015-11-09:02})},\emph{(\ref{2015-11-10:02})} can be rephrased by saying that
$$
T_t\colon C_b(E)\rightarrow C_b(E),\ f \mapsto \int_Ef(\xi)\mu_t(\cdot,d\xi)
$$
is well defined for all $t\in \mathbb{R}^+$ and $T\coloneqq \{T_t\}_{t\in \mathbb{R}^+}$ is a transition semigroup in $C_b(E)$. If  $\boldsymbol{\mu}$ is a probability kernel, then $T$ is a Markov transition semigroup.

\begin{Proposition}\label{2015-10-29:12}
 Let Assumption \ref{2015-10-28:30} holds and let $T\coloneqq \{T_t\}_{t\in \mathbb{R}^+}$ be defined as in 
\eqref{eq:2015-11-10:01}. Then $T$  is a $C_0$-sequentially locally  equicontinuous semigroup on $(C_b(E),\tau_\mathcal{K})$. Moreover, 
for every $\alpha>0$, the normalized semigroup $\{e^{-\alpha t}T_t\}_{t\in \mathbb{R}^+}$ is a
$C_0$-sequentially equicontinuous semigroup on $(C_b(E),\tau_\mathcal{K})$ satisfying
Assumption \ref{ass:int}.
\end{Proposition}
\begin{proof}
Assumptions \ref{2015-10-28:30}\emph{(\ref{2015-11-10:00})},\emph{(\ref{2015-11-09:02})},\emph{(\ref{2015-11-10:02})} 
imply that $T$ maps $C_b(E)$ into itself and that it  is a semigroup.
We show that the  $C_0$-property holds,
 that is   $\lim_{t\rightarrow 0^+}T_tf=f$ in $(C_b(E),\tau_\mathcal{K})$ for every $f\in C_b(E)$.
Let
$
M\coloneqq \sup_{\substack{t\in \mathbb{R}^+\\\xi\in E}}|\mu_t(\xi,E)|.
$ 
By Assumption \ref{2015-10-28:30}\emph{(\ref{2015-11-10:00})}, $M<+\infty$
and
\begin{equation}\label{eqqu}
|T_tf|_\infty\leq M|f|_\infty \qquad\forall  f\in C_b(E),\ \ \forall  t\in \mathbb{R}^+.
\end{equation}
Let $f\in C_b(E)$. By \eqref{eqqu} and by  Proposition \ref{prop:2015-04-24:01}\emph{(\ref{2015-10-13:10})}, in order to show that $\lim_{t\rightarrow 0^+}T_tf=f$ in $(C_b(E),\tau_\mathcal{K})$,
 it is sufficient to show that
$\lim_{t\rightarrow 0^+}[T_tf-f]_K=0$, for every $K\subset E$ non-empty compact. Let $K\subset E$ be such a set.
We claim that
    \begin{equation}
      \label{eq:2015-11-10:05}
      \lim_{t\rightarrow 0^+}\sup_{\xi \in K}|\mu_t(\xi,E)-1|=0.
    \end{equation}
Indeed, let $\epsilon$ and $K_0$ as in 
Assumption \ref{2015-10-28:30}\emph{(\ref{2015-10-28:32})}, when $\hat t=1$, and let
$r\coloneqq \sup_{(\xi,\xi')\in K\times K_0}d(\xi,\xi')+1$.
Then $K_0\subset B(\xi,r)$ for every $\xi\in K$.
For $t\in [0,1]$ and $\xi\in K$, we have
\begin{equation*}
  \begin{split}
  |\mu_t(\xi,E)-1|&\leq  |\mu_t(\xi,E\setminus B(\xi,r)|+|\mu_t(\xi,B(\xi,r))-1|\\
&\leq |\mu_t(\xi,E\setminus K_0)|+|\mu_t(\xi,B(\xi,r))-1|\\
&\leq \epsilon +|\mu_t(\xi,B(\xi,r))-1|.
\end{split}
\end{equation*}
By taking the supremum over $x\in K$, by passing to the limit as $t\rightarrow 0^+$, by using \eqref{eq:2015-11-10:04}, and by arbitrariness of $\epsilon$, we obtain \eqref{eq:2015-11-10:05}.
In particular,
\eqref{eq:2015-11-10:05} implies
\begin{equation}
  \label{eq:2015-11-10:03}
  \lim_{t\rightarrow 0^+}\sup_{\xi\in K}|f(\xi)-\mu_t(\xi,E)f(\xi)|=0,
\end{equation}
and then $T_tf\rightarrow f$ in $\tau_\mathcal{K}$ as $t\rightarrow 0^+$ if and only if
\begin{equation}
  \label{2015-11-10:06}
  \lim_{t\rightarrow 0^+}\sup_{\xi\in K}|T_tf(\xi)-\mu_t(\xi,E)f(\xi)|=0.
\end{equation}
Again, let $\epsilon>0$ and $K_0$ be as in 
Assumption \ref{2015-10-28:30}\emph{(\ref{2015-10-28:32})}, when $\hat t=1$.
Let $w$ be a modulus of continuity for  $f_{|K_0}$. For $\delta>0$, $t\in [0,1]$, and $\xi\in K$, we write
\begin{equation*}
  \begin{split}
    |T_tf(\xi)-\mu_t(\xi,E)f(\xi)|\leq& \int_E|f(\xi')-
f(\xi)
|\mu_t(\xi,d\xi')    =
    \int_{K_0\cap B(\xi,\delta)}|f(\xi')-f(\xi)|\mu_t(\xi,d\xi')\\
    &+
    \int_{K_0\cap B(\xi,\delta)^c}|f(\xi')-f(\xi)|\mu_t(\xi,d\xi')+
    \int_{K_0^c}|f(\xi')-f(\xi)|\mu_t(\xi,d\xi')\\
    \leq& w(\delta)+2|f|_\infty \left(\mu_t(\xi,B(\xi,\delta)^c)+\epsilon  \right).
  \end{split}
\end{equation*}
We then obtain
$$
\sup_{\xi\in K}|T_tf(\xi)-\mu_t(\xi,E)f(\xi)|\leq \omega(\delta)+2|f|_\infty \left(\sup_{\xi\in K} \mu_t(\xi,B(\xi,\delta)^c)+ \epsilon \right)\qquad  \forall  \delta>0, \ \forall  t\in [0,1],\  \forall  \xi\in K.
$$
By passing to the limit as $t\rightarrow 0^+$, by 
\eqref{eq:2015-11-10:04}, by
\eqref{eq:2015-11-10:05},
and by arbitrariness of $\delta$ and $\epsilon$, we obtain 
\eqref{2015-11-10:06}.

We now show that  $\{T_t\}_{t\in [0,\hat t]}$ is sequentially equicontinuous for every $\hat t>0$. Let $\{f_n\}_{n\in \mathbb{N}}$ be a sequence converging to $0$ in $(C_b(E),\tau_\mathcal{K})$ and let $\hat t>0$. By Proposition \ref{prop:2015-04-24:01}\emph{(\ref{2015-10-13:10})}, $\{|f_n|_\infty\}_{n\in \mathbb{N}}$ is bounded by some $b>0$. Then, by \eqref{eqqu}, $\{T_tf_n\}_{t\in \mathbb{R}^+,n\in \mathbb{N}}$ is bounded. To show that $T_tf_n\rightarrow 0$ in $(C_b(E),\tau_\mathcal{K})$, uniformly for $t\in[0,\hat t]$, it is then sufficient to show that
$$
\lim_{n\rightarrow +\infty}\sup_{t\in[0,\hat t]}[T_tf_n]_K=0\qquad  \forall  K\subset E\mbox{ non-empty compact.}
$$
Let $\varepsilon>0$
and  $K_0$ be as in
Assumption \ref{2015-10-28:30}\emph{(\ref{2015-10-28:32})},
when $\hat t=1$.
Then, 
for $t\in[0,\hat t]$, $\xi\in K$, $n\in \mathbb{N}$, we have
\begin{equation*}
  \begin{split}
    |T_tf_n(\xi)|\leq \int_{K_0}|f_n(\xi')|\mu_t(\xi,d\xi')+\int_{K_0^c}|f_n(\xi')|\mu_t(\xi,d\xi')\leq M[f_n]_{K_0}+ b\epsilon .
  \end{split}
\end{equation*}
Since $[f_n]_{K_0}\rightarrow 0$ as $n\rightarrow +\infty$, by arbitrariness of $\epsilon$ we conclude $\sup_{t\in [0,\hat t]}[T_tf_n]_K\rightarrow 0$ as $n\rightarrow +\infty$.
This concludes the proof that  
 $T$ is a $C_0$-sequentially locally equicontinuous semigroup on $(C_b(E),\tau_\mathcal{K})$.

Next,  by Proposition \ref{prop:2015-04-29:00} and by \eqref{eqqu}, we can apply
Proposition \ref{prop:2015-08-01:00} and obtain that, for every $\alpha>0$,  $\{e^{-\alpha t}T_t\}_{t\in \mathbb{R}^+}$ is 
$C_0$-sequentially locally equicontinuous semigroup on $(C_b(E),\tau_\mathcal{K})$.
Finally, by 
Remark 
\ref{2015-10-29:13}
 and Proposition \ref{prop:seqcom}\emph{(\ref{2015-11-10:07})}, we conclude that 
Assumption \ref{ass:int} holds true for $\{e^{-\alpha t}T_t\}_{t\in \mathbb{R}^+}$.
\end{proof}

\subsection{Extension to weighted spaces of continuous functions} \label{sub2}

In this subsection, we briefly discuss how to deal with  transition semigroups in weighted spaces  of continuous functions.
Let $\gamma\in C(E)$ such that $\gamma>0$. We introduce the following $\gamma$-weighted space of continuous functions
$$
C_\gamma(E)\coloneqq   \left\{ f\in C(E): f\gamma \in C_b(E) \right\} .
$$
A typical case is when $E$ is an unbounded subset of a Banach space and $\gamma(x)=(1+|x|_E^p)^{-1}$, for some $p\in \N$. Then $C_\gamma(E)$ is the space of continuous functions on $E$ having at most polynomial growth of order $p$. 
By the very definition of $C_\gamma(E)$, the multiplication by $\gamma$
$$
\varphi_\gamma\colon C_\gamma(E)\rightarrow C_b(E),\  f \mapsto  f\gamma,
$$
defines an algebraic isomorphism. Hence, by endowing $C_\gamma(E)$ with the topology $\tau^{\gamma}_{\mathcal{K}}\coloneqq \gamma^{-1}(\tau_\mathcal{K})$, this space becomes a locally convex Hausdorff topological vector space.
A family of seminorms inducing $\tau^\gamma_\mathcal{K}$ is given by
$$
p^\gamma_{K,\mu}\coloneqq [f\gamma]_K+\left|\int_E f\gamma d\mu\right|\qquad \ \forall  f\in C_\gamma(E),
$$
when $K$ ranges on the set of non-empty compact subsets of $E$ and $\mu$ ranges on $\mathbf{ca}(E)$.
 Clearly,  $ \left( C_\gamma(E),\tau^\gamma_\mathcal{K} \right) $ and $(C_b(E),\tau_\mathcal{K})$ enjoy the same topological properties and $\gamma$ is an isomorphism of topological vector spaces.
This basic observation will be used now to frame $C_0$-sequentially locally equicontinuous semigroups on $(C_\gamma(E),\tau^{\gamma}_\mathcal{K})$ induced by transition functions.

Let $\boldsymbol{\mu}\coloneqq \{\mu_t(\xi,\cdot)\}_{t\in \mathbb{R}^+,\xi\in E}\subset \mathbf{ca}^+(E)$ and let $t\in \R^+$. Define
the family $\boldsymbol{\mu}^\gamma\coloneqq \{\mu_t^\gamma(\xi,\cdot)\}_{t\in \mathbb{R}^+,\xi\in E}$ by
\begin{equation}
  \label{eq:2015-11-12:06}
  \mu^\gamma_t(\xi,\Gamma)\coloneqq \gamma(\xi)\int_\Gamma \gamma^{-1}(\xi')\mu_t(\xi,d\xi')\qquad  \forall \Gamma\in \mathcal{E}, \ \forall \xi\in E,
\end{equation}
and
\begin{equation}\label{TTT}
T_tf(\xi)\coloneqq \int_E f(\xi')\mu_t(\xi,d\xi') \qquad\forall \xi \in E, \ \forall f\in C_\gamma(E). 
\end{equation}
Given $f\in C_\gamma(E)$, $\xi\in E$,
the latter is well defined and finite if and only
if and only if 
$$
\int_E \varphi_\gamma(f)(\xi')\mu^\gamma_t(\xi,d\xi')=
\gamma(\xi)\int_E f(\xi')\mu_t(\xi,d\xi')
$$
is well defined and  finite. Then, $T_tf(\xi)$ is well defined and finite if and only if, 
setting $g\coloneqq f\gamma$, 
$$
T^\gamma_tg(\xi)\coloneqq \int_E g(\xi')\mu^\gamma_t(\xi,d\xi')
$$
is well defined and  finite.
At the end, we get that $T_tf(\xi)$ is well defined and finite for every $\xi\in E$ and every $f\in C_\gamma(E)$ if and only if $T_t^\gamma g (f)$ is well defined and finite for every  $\xi\in E$ and every $g\in C_b(E)$. In such a case 
\begin{equation}
  \label{eq:2015-11-10:08}
  (\varphi^{-1}_\gamma \circ T^\gamma_t\circ \varphi_\gamma )f=T_tf \qquad \forall f\in C_\gamma(E),
\end{equation}
that is, the diagram 
\[
\begin{tikzpicture}[scale=6.5]
  \node (A) at (0,0.3) {$C_\gamma(E)$};
  \node (B) at (0.38,0.3) {$C_b(E)$};
  \node (C) at (0,0) {$C_\gamma(E)$};
  \node (D) at (0.38,0) {$C_b(E)$};
  \path[->,font=\scriptsize]
  (A) edge node [above] {$\varphi_\gamma$} (B);
  \path[->,font=\scriptsize]
  (D) edge node [below] {$\varphi_\gamma^{-1}$} (C);
  \path[->,font=\scriptsize]
  (A) edge node [right] {$T_t$} (C);
  \path[->,font=\scriptsize]
  (B) edge node [right] {$T^\gamma_t$} (D);
\end{tikzpicture}
\]
is commutative.
Due to this fact, Proposition \ref{2015-10-29:12} can be immediately stated in the following equivalent form.

\begin{Proposition}\label{2015-11-10:09}
 Let $\boldsymbol{\mu}\coloneqq \{\mu_t(\xi,\cdot)\}_{t\in \mathbb{R}^+,\xi\in E}\subset \mathbf{ca}^+(E)$ and 
let ${\boldsymbol{\mu}}^\gamma\coloneqq\{\mu^\gamma_t(\xi,\cdot)\}_{\substack{t\in\mathbb{R}^+\\\xi\in E}}$ be defined starting from $\boldsymbol{\mu}$ through \eqref{eq:2015-11-12:06}. Assume that ${\boldsymbol{\mu}}^\gamma$ satisfies Assumption \ref{2015-10-28:30} (when ${\boldsymbol{\mu}}$ is replaced by ${\boldsymbol{\mu}}^\gamma$).
Then $T\coloneqq \{T_t\}_{t\in \mathbb{R}^+}$ defined in \eqref{TTT} is a $C_0$-sequentially locally  equicontinuous semigroup on $(C_\gamma(E),\tau^\gamma_\mathcal{K})$. Moreover, 
for every $\alpha>0$, the normalized semigroup $\{e^{-\alpha t}T_t\}_{t\in \mathbb{R}^+}$ is a
$C_0$-sequentially equicontinuous semigroup on $(C_\gamma(E),\tau_\mathcal{K}^\gamma)$ satisfying
Assumption \ref{ass:int}.
\end{Proposition}

\subsection{Markov transition semigroups associated to stochastic differential equations}
\label{2015-10-29:16}

Propositions \ref{2015-10-29:12} and \ref{2015-11-10:09} have a straightforward application  to transition functions associated to mild solutions of stochastic differential equations in Hilbert spaces.
Let $(U,|\cdot|_U)$, $(H,|\cdot|_H)$ be  separable Hilbert spaces, let $(\Omega,\mathcal{F},\{\mathcal{F}_t\}_{t\in \mathbb{R}^+},\mathbb{P})$ be a  complete filtered probability space, let $Q$ be a positive self-adjoint operator, and let $W^Q$ be a $U$-valued $Q$-Wiener process defined on  $(\Omega,\mathcal{F},\{\mathcal{F}_t\}_{t\in \mathbb{R}^+},\mathbb{P})$ (see \cite[Ch.\ 4]{DaPratoZabczyk14}). Denote by $L_2(U_0,H)$ the space of Hilbert-Schmidt operators from $U_0\coloneqq Q^{1/2}(U)$~(\footnote{The scalar product on $U_0$ is defined by $\langle u,v\rangle_{U_0}\coloneqq \langle Q^{-1/2} u, Q^{-1/2} v\rangle_H$.})
into $H$, let  $A$ be the generator of a strongly continuous semigroup $\{S_A(t)\}_{t\in \R^+}$ in $(H,|\cdot|_H)$, and let  $F\colon H\rightarrow H$, $B\colon H\rightarrow L_2(U_0,H)$. 
Then, under suitable assumptions on the coefficients $F$ and $B$ (e.g.,  \cite[p.\ 187, Hypotehsis\ 7.1]{DaPratoZabczyk14}), for every $\xi\in H$,
the  SDE in the space $H$
\begin{equation}\label{SDEinf}
  \begin{dcases}
    dX(t)= AX(t)+ F(X(t))dt+B(X(t)) dW^Q(t) & t\in(0,T]\\
    X(0)=\xi,
\end{dcases}
\end{equation}
admits a unique (up to undistinguishability) mild solution $X(\cdot,\xi)$ 
with continuous trajectories (see \cite[p.\ 188, Theorem\ 7.2]{DaPratoZabczyk14}),
that is, there exists a unique $H$-valued process  $X(\cdot,\xi)$ with continuous trajectories satisfying the integral equation
\begin{equation}\label{mild}
X(t,\xi)=S_A(t) \xi+\int_0^t S_A(t-s) F(X(s,\xi))ds
+\int_0^t S_A(t-s) B(X(s,\xi))dW^Q(s)\qquad \forall t\in\R^+.
\end{equation} 
  By standard estimates 
(see, e.g., \cite[p.\ 188, Theorem\ 7.2]{DaPratoZabczyk14}(\footnote{The constant in that estimate can be taken exponential in time, because the SDE is autonomous.})), for every $p\geq 2$  we have, for some $K_p>0$ and  $\hat \alpha_p\in \mathbb{R}$, 
  \begin{equation}
    \label{eq:2015-11-12:05}
    \mathbb{E} \left[ |X(t,\xi)|_H^p \right] \leq K_pe^{\hat \alpha_p t}(1+|\xi|_H^p)\qquad \forall
(t,\xi)\in \mathbb{R}^+\times H.
\end{equation}
Moreover, by \cite[p.\ 235, Theorem 9.1]{DaPratoZabczyk14},
\begin{equation}\label{contt}
(t,\xi)\ \mapsto \ X(t,\xi) \mbox{ is stochastically continuous}.
\end{equation}

\begin{Proposition}\label{2015-11-12:00}
Let \cite[Hypothesis\ 7.1]{DaPratoZabczyk14} hold and 
let $X(\cdot,\xi)$ be the mild solution to \eqref{SDEinf}.
\begin{enumerate}[(i)]
\item\label{2015-12-07:00} Define
\begin{equation}\label{Semmark}
T_t f(\xi)\coloneqq \E\left[f(X(t,\xi))\right] \qquad \forall f\in C_b(H) \ \forall \xi\in H, \ \forall t\in\R^+ .
\end{equation}
Then 
  $T\coloneqq \{T_t\}_{t\in\R^+}$ is a $C_0$-sequentially locally equicontinuous semigroup in $(C_b(H),\tau_{\mathcal{K}})$.
 Moreover, $ \{e^{-\alpha t} T_t\}_{t\in\R^+}$ is    a $C_0$-sequentially equicontinuous semigroup in $(C_b(H),\tau_{\mathcal{K}})$ for every $\alpha>0$.
\item \label{2015-12-07:01}
Let $p\geq 2$  and set $\gamma(\xi)\coloneqq (1+|\xi|_H^p)^{-1}$ for $\xi\in H$. 
Define
\begin{equation}\label{Semmarkgamma}
T_t f(\xi)\coloneqq \E\left[f(X(t,\xi))\right]\qquad\forall f\in C_\gamma(H), \ \forall \xi\in H, \ \forall t\in\R^+ .
\end{equation}
 Then $T\coloneqq \{T_t\}_{t\in\R^+}$ is a $C_0$-sequentially locally equicontinuous semigroup in $(C_\gamma(H),\tau^\gamma_{\mathcal{K}})$. Moreover, $ \{e^{-\alpha t} T_t\}_{t\in\R^+}$ is a $C_0$-sequentially equicontinuous semigroup in $(C_\gamma(H),\tau^\gamma_{\mathcal{K}})$ for every $\alpha>\hat \alpha_p$, where $\hat \alpha_p$ is the constant appearing in \eqref{eq:2015-11-12:05}.
 \end{enumerate}
\end{Proposition}
\begin{proof}
\emph{(\ref{2015-12-07:00})}
Define
\begin{equation}
  \label{eq:2015-12-07:02}
  \mu_t(\xi,\Gamma)\coloneqq
 \mathbb{P}(X(t,\xi)\in \Gamma)\qquad \forall t\in \mathbb{R}^+,\ \forall  \xi\in H, \ \forall \Gamma\in \mathcal{B}(H).
\end{equation}
We show that we can apply Proposition
\ref{2015-10-29:12} with the family $\boldsymbol{\mu}\coloneqq\{\mu_t(\xi,\cdot)\}_{\substack{t\in\R^+\\\xi\in H}}$ given by \eqref{eq:2015-12-07:02}.
The condition of Assumption \ref{2015-10-28:30}\emph{(\ref{2015-11-10:00})}
 is clearly verified. The condition of Assumption \ref{2015-10-28:30}\emph{(\ref{2015-11-09:02})} is consequence of \eqref{contt}. The condition of Assumption \ref{2015-10-28:30}\emph{(\ref{2015-11-10:02})}
is  verified by  \cite[p.\ 249, Corollaries 9.15 and 9.16]{DaPratoZabczyk14}.

Now we verify the condition of Assumption \ref{2015-10-28:30}\emph{(\ref{2015-10-28:32})}.
 Let $\hat t>0$ and let $K\subset E$ compact. 
By \eqref{contt} the map
$$
\mathbb{R}^+\times H \rightarrow  \left( \mathbf{ca}(H)
,\sigma \left( \mathbf{ca}(H),C_b(H) \right)  \right) 
,\ \  (t,\xi) \mapsto \mu(\xi,\cdot)
$$
is continuous. 
Then the family of probability measures $\{\mu_t(\xi,\cdot)\}_{(t,\xi)\in [0,\hat t]\times H}$ is 
$\sigma \left( \mathbf{ca}(H),C_b(H) \right)$-compact. Hence,
 by \cite[p.\ 519, Theorem 15.22]{Aliprantis2006},
it  is tight.

We finally verify the condition of
Assumption \ref{2015-10-28:30}\emph{(\ref{2015-10-28:31})}.
Let $r>0$, let $\{t_n\}_{n\in \mathbb{N}}\subset \mathbb{R}^+$ be a sequence converging to $0$, and let $\{\xi_n\}_{n\in \mathbb{N}}$ be sequence converging to $\xi$ in $H$. By \eqref{contt}  and  recalling that $X(0,\xi)=\xi$, we get
\begin{equation*}
    \lim_{n\rightarrow +\infty}
\mu_{t_n} \left( \xi_n,B(\xi_n,r) \right) =
\lim_{n\rightarrow +\infty}
\mathbb{P} \left(|\xi_n-X(t_n,\xi_n)|_H < r\right) =0.
\end{equation*}
By arbitrariness of the sequences $\{t_n\}_{n\in\N}$, $\{\xi_n\}_{n\in\N}$ and of $\xi$, this implies the condition of
Assumption \ref{2015-10-28:30}\emph{(\ref{2015-10-28:31})}.

\emph{(\ref{2015-12-07:01})}
First, we notice that $T_tf$ in \eqref{Semmarkgamma} is well defined due to \eqref{eq:2015-11-12:05}.
Consider now the family $\boldsymbol{\mu}\coloneqq\{\mu_t(\xi,\cdot)\}_{\substack{t\in\R^+\\\xi\in H}}$ defined in
\eqref{eq:2015-12-07:02}
 and  the renormalized family 
$\boldsymbol{\nu}\coloneqq\{\nu_t(\xi,\cdot)\}_{\substack{t\in\R^+\\\xi\in H}}$ defined by $\nu_t(\xi,\cdot)\coloneqq e^{-\hat \alpha_p t} \mu_t$. Then, consider the weighted family $\boldsymbol{\nu}^\gamma\coloneqq\{\nu^\gamma_t(\xi,\cdot)\}_{\substack{t\in\R^+\\\xi\in H}}$ 
\begin{equation*}
  \nu^{\gamma}_t(\xi,\Gamma)\coloneqq \frac{1}{1+|\xi|^p}\int_\Gamma (1+|\xi'|^p) \nu_t(\xi,d\xi')\qquad \forall \Gamma\in \mathcal{B}(H), \ \forall \xi\in H.
\end{equation*}
We have
$$
T_t f(\xi) =e^{\hat \alpha_p t} \int_H f(\xi') \nu_t(\xi,d\xi')\qquad  \forall  f\in C_\gamma(H), \ \forall \xi\in H,  \ \forall t\in\R^+. 
$$
Hence, 
by Proposition
\ref{2015-11-10:09}, the proof reduces to show that   Assumption  \ref{2015-10-28:30} is  verified by $\boldsymbol{\nu}^\gamma$.
The latter follows straightly from its definition by taking into account the properties already proved for $\boldsymbol{\mu}$ in part \emph{(i)} of the proof and \eqref{eq:2015-11-12:05}.
\end{proof}

\begin{Example}
\label{2015-05-04:07-II}
Let $H$ be a non-trivial separable Hilbert space with inner product $\langle\cdot,\cdot\rangle$.
Let $Q\in L(H)$ be a positive  self-adjoint trace-class operator and  let $W^Q$ be a $Q$-Wiener process in $H$ on some filtered probability space
$(\Omega,\mathcal{F},\{\mathcal{F}_t\}_{t\in\R^+}, \mathbb{P})$ (see \cite[Ch.\ 4]{DaPratoZabczyk14}). 
Let  $T=\{T_t\}_{t\in \mathbb{R}^+}$ be defined by
$$
T_tf (\xi)\coloneqq \mathbb{E}[f(\xi+W_t^Q)] = \int_H f(\xi')\mu_t(\xi,d\xi') \qquad \forall f\in C_b(H), \  \forall\xi\in H, \ \forall t\in \R^+,
$$
where 
$\mu_t(\xi,\cdot)$ denotes the law of $\xi+W^Q_t$.
Then, by Proposition  \ref{2015-11-12:00}, $T$ is a $C_0$-sequentially locally equicontinuous semigroup in $(C_b(H),\tau_\mathcal{K})$. 
 We claim that $T$ is not locally equicontinuous. Indeed,
if $T$ was locally equicontinuous, for any fixed $\hat t>0$,
there should exist $L>0$, a compact set $K\subset H$, and $\eta_1,\ldots,\eta_n\in \ca(H)$ such that
\begin{equation}
  \label{eq:2015-07-23:00-II}
  \sup_{t\in [0,\hat t]}|T_tf(0)|\leq L\left([f]_{K}+\sum_{i=1}^n\left |\int_Hf d\eta_i\right|\right)\qquad  \forall    f\in C_b(H).
\end{equation}
Let   $v\in H\setminus \{0\}$ and let $a\coloneqq \max_{h\in K}|\langle v,h\rangle|$.
Then, denoting by $\lambda_t$ the pushforward measure of  $\mu_t(0,\cdot)$ through the application $\langle v,\cdot \rangle$ (that is the law of the real-valued random variable $\langle v,W_t^Q\rangle$), and by 
$\nu_i$, $i=1,...,n$, the pushforward measure of  $\eta_i$ through the same application, inequality 
\eqref{eq:2015-07-23:00-II} provides, in particular,
\begin{equation}
  \label{eq:2015-07-23:01-II}
  \sup_{t\in [0,\hat t]}
\left|\int_{a}^{+\infty}g
d \lambda_t\right|\leq L\sum_{i=1}^n\left |\int_a^{+\infty}g d\nu_i\right|,\qquad \ \forall   g\in C_{0,b}([a,+\infty)),
\end{equation}
where $C_{0,b}([a,+\infty))$ is the space of bounded continuous functions $f$ on $[a,+\infty)$ such that $ f(a)=0$.
Then, by \cite[p.\ 63, Lemma\ 3.9]{Rudin1991},
 every $\lambda_t$ restricted to $(a,+\infty)$ must be a linear combination of the measures $\nu_1\ldots, \nu_n$ restricted to $(a,+\infty)$. In particular, choosing any sequence $0<t_1<\ldots<t_n<t_{n+1}\leq \hat t$,
the family $\{\lambda_{t_i}\lfloor (a,+\infty)\}_{i=1,\ldots,n+1}$ is linearly dependent.
This is  not possible, as they are restrictions of nondegenerate Gaussian laws having all different variances.
\end{Example}
\begin{Remark}
In this subsection we have considered a Hilbert space setting, as the theory of SDEs in Hilbert spaces is very well developed and the properties of their solutions allow to state our results for a large class of SDEs. Nevertheless, the same kind of results hold for suitable classes of SDEs in Banach spaces (see e.g.\  \cite{Goldys2003}).     
\end{Remark}

\addcontentsline{toc}{section}{References}
\bibliographystyle{plain}
\bibliography{biblio.bbl}

\end{document}